\documentclass[letterpaper,11pt,oneside,reqno]{article}

\usepackage[pdftex,backref=page,colorlinks=true,linkcolor=blue,citecolor=red]{hyperref}
\usepackage[alphabetic,nobysame]{amsrefs}

\usepackage{amsmath,amssymb,amsthm,amsfonts,mathtools}
\usepackage{graphicx,color}
\usepackage{upgreek}
\usepackage[mathscr]{euscript}

\allowdisplaybreaks
\numberwithin{equation}{section}

\usepackage{tikz}
\usetikzlibrary{shapes,arrows,positioning,decorations.markings}

\usepackage{array}
\usepackage{adjustbox}
\usepackage{cleveref}
\usepackage{enumerate}

\usepackage[DIV=12]{typearea}

\newcommand{\ssp}{\hspace{1pt}}

\newtheorem{proposition}{Proposition}[section]
\newtheorem{lemma}[proposition]{Lemma}
\newtheorem{corollary}[proposition]{Corollary}
\newtheorem{theorem}[proposition]{Theorem}
\theoremstyle{definition}
\newtheorem{definition}[proposition]{Definition}
\newtheorem{remark}[proposition]{Remark}

\begin{document}
\title{Rewriting History in Integrable Stochastic Particle Systems}

\author{Leonid Petrov and Axel Saenz}

\date{}

\maketitle

\begin{abstract} 
	Many integrable stochastic particle systems in one space dimension (such as TASEP --- Totally Asymmetric Simple Exclusion Process --- and its $q$-deformation, the $q$-TASEP) remain integrable if we equip each particle with its own speed parameter. In this work, we present intertwining relations between Markov transition operators of particle systems which differ by a permutation of the speed parameters. These relations generalize our previous works \cite{PetrovSaenz2019backTASEP}, \cite{petrov2019qhahn}, but here we employ a novel approach based on the Yang-Baxter equation for the higher spin stochastic six vertex model. Our intertwiners are Markov transition operators, which leads to interesting probabilistic consequences.

	First, we obtain a new Lax-type differential equation for the Markov transition semigroups of homogeneous, continuous-time versions of our particle systems. Our Lax equation encodes the time evolution of multipoint observables of the $q$-TASEP and TASEP in a unified way, which may be of interest for the asymptotic analysis of multipoint observables of these systems.

	Second, we show that our intertwining relations lead to couplings between probability measures on trajectories of particle systems which differ by a permutation of the speed parameters. The conditional distribution for such a coupling is realized as a ``rewriting history'' random walk which randomly resamples the trajectory of a particle in a chamber determined by the trajectories of the neighboring particles. As a byproduct, we construct a new coupling for standard Poisson processes on the positive real half-line with different rates.
\end{abstract}

\bigskip

\setcounter{tocdepth}{1}
\tableofcontents
\setcounter{tocdepth}{3}

\section{Introduction}
\label{sec:intro}

\subsection{Overview}
\label{sub:overview}

Integrable (also called ``exactly solvable'') stochastic interacting particle systems on the line are Markov chains on configurations of particles on $\mathbb{Z}$, which feature exact formulas governing their distributions at any given time. These formulas lead to precise control of the asymptotic behavior of these Markov chains in the limit to large scale and long times. In the past two decades, integrable particle systems have been instrumental in uncovering new universal asymptotic phenomena, including those present in the Kardar-Parisi-Zhang universality class. See Corwin \cite{CorwinKPZ}, \cite{Corwin2016Notices}, Halpin-Healy--Takeuchi \cite{halpin2015kpzCocktail}, and Quastel--Spohn \cite{QuastelSpohnKPZ2015}.

Initial progress for integrable stochastic particle systems was achieved through the use of determinantal techniques, e.g., see~Johansson \cite{johansson2000shape} for the asymptotic fluctuations of TASEP (Totally Asymmetric Simple Exclusion Process). More recently, new tools arising from quantum integrability, Bethe ansatz, and symmetric functions were applied to deformations of TASEP and related models. These deformations include ASEP, where particles may jump in both directions with asymmetric rates (Tracy--Widom \cite{TW_ASEP1}, \cite{TW_ASEP2}); random polymers (O'Connell \cite{Oconnell2009_Toda}, Corwin--O'Connell--Sepp\"al\"ainen--Zygouras \cite{COSZ2011}, O'Connell--Sepp\"al\"ainen--Zygouras \cite{OSZ2012}, Sepp\"al\"ainen \cite{Seppalainen2012}, Barraquand--Corwin \cite{CorwinBarraquand2015Beta}); and various $q$-deformations of the TASEP which modify its jump rates. Among the latter, in this paper, we consider the $q$-TASEP introduced by Borodin--Corwin \cite{BorodinCorwin2011Macdonald} (see also Sasamoto--Wadati \cite{SasamotoWadati1998}), and the $q$-Hahn TASEP introduced and studied in Povolotsky \cite{Povolotsky2013} and Corwin \cite{Corwin2014qmunu}.

One of the most recent achievements in the study of the structure of integrable stochastic particle systems is their unification under the umbrella of \emph{integrable stochastic vertex models} initiated in Borodin--Corwin--Gorin \cite{BCG6V}, Corwin--Petrov \cite{CorwinPetrov2015}, and Borodin--Petrov \cite{BorodinPetrov2016inhom}, \cite{BorodinPetrov2016_Hom_Lectures}. The integrability of the stochastic vertex models is powered by the \emph{Yang–Baxter equation}, which is a local symmetry of the models arising from the underlying algebraic structure. This is the central starting point for studying the stochastic vertex models.

\medskip

Ever since the original works on TASEP, it was clear that integrability in particle systems like TASEP is preserved when we introduce countably many extra parameters; see Gravner--Tracy--Widom \cite{Gravner-Tracy-Widom-2002a} and Its--Tracy--Widom \cite{Its-Tracy-Widom-2001a}. A typical example is when each particle has its own jump rate. One can trace the ability to perform such a multiparameter deformation to the underlying algebraic structure of the model, which connects it to a particular family of symmetric polynomials (e.g.,~the probability distribution of the TASEP may be written in terms of the Schur polynomials). In the framework of symmetric functions, interacting particle systems with different particle speeds already appeared in Vershik--Kerov \cite{Vershik1986} in connection with the Robinson--Schensted--Knuth (RSK) correspondence; see also O'Connell \cite{OConnell2003Trans}, \cite{OConnell2003} for further probabilistic properties of the RSK.

The multiparameter deformation of integrable stochastic particle systems should be contrasted with $q$-deformations, like the one turning TASEP into the $q$-TASEP. The latter introduces just one extra parameter while at the same time deforming the underlying symmetric functions in a nontrivial way (for $q$-TASEP, passing from the Schur functions to $q$-Whittaker functions). On the other hand, our multiparameter deformations rely on the presence of symmetry itself and can be readily combined with $q$-deformations.

Let us remark that TASEP in inhomogeneous space (when the jump rate of a particle depends on its location) does not seem to be integrable; see Costin--Lebowitz--Speer--Troiani \cite{costin2012blockage}, Janowsky--Lebowitz \cite{janowsky1992slow_bond}, and Sepp\"al\"ainen \cite{seppalainen2001slow_bond}. For this reason, control of the asymptotic fluctuations in this process requires very delicate asymptotic analysis; see Basu--Sidoravicius--Sly, \cite{Basuetal2014_slowbond}, Basu--Sarkar--Sly \cite{basu2017invariant}. Moreover, it is not known whether ASEP has any integrable multiparameter deformations. The stochastic six vertex model introduced and studied by Gwa--Spohn \cite{GwaSpohn1992} and Borodin--Corwin--Gorin \cite{BCG6V} scales to ASEP and admits such a multiparameter deformation, see Borodin--Petrov \cite{BorodinPetrov2016inhom}. However, the scaling to ASEP destroys this structure. Recently other families of spatially inhomogeneous integrable stochastic particle systems in one and two space dimensions were studied by Assiotis \cite{theodoros2019_determ}, Borodin--Petrov \cite{BorodinPetrov2016Exp}, Knizel--Petrov--Saenz \cite{SaenzKnizelPetrov2018}, and Petrov \cite{Petrov2017push}.

\medskip

Due to the underlying algebraic structure powered by symmetric functions, certain joint distributions in integrable stochastic particle systems are \emph{symmetric} under (suitably restricted classes of) permutations of their speed parameters. This symmetry is far from being evident from the definition of a particle system and is often observed only as a consequence of explicit formulas. In our previous works (Petrov--Saenz \cite{PetrovSaenz2019backTASEP}, Petrov \cite{petrov2019qhahn}, Petrov--Tikhonov \cite{PetrovTikhonov2019}), we explored various probabilistic consequences of these distributional symmetries. In particular, we constructed natural monotone couplings between fixed-time distributions in particle systems which differ by a permutation of the speed parameters. 

However, the analysis in our previous works was severely restricted to the case of the distinguished \emph{step} initial configuration in the particle systems and only to couplings of fixed-time distributions. This paper presents a new approach based on the Yang-Baxter equation and widely extends the scope of distributional symmetries and monotone couplings in integrable stochastic particle systems. In particular, we extend the previous results to both the general initial conditions and to couplings of measures on whole trajectories (and not only fixed-time distributions).

In the rest of the Introduction, we formulate the paper's main results. We use a simplified notation for some Markov operators that slightly differs from the notation later used in the rest of the text. We begin by presenting concrete new probabilistic results in the well-known setting of the two-particle continuous time TASEP and the Poisson processes on $(0,+\infty)$ in \Cref{sub:intro_2_cars,sub:intro_Poisson}.

\subsection{Coupling in the two-particle TASEP with different particle speeds}
\label{sub:intro_2_cars}

Before presenting our general results in \Cref{sub:intro_couplings} below,
let us illustrate them
in the simplest
nontrivial case, the continuous time TASEP with two
particles. Think of them as two cars, one fast and one slow,
driving on a one-lane road evolving in continuous time. The
speeds of the cars are $\alpha_0>\alpha_1>0$. These are the rates of independent exponential clocks associated with the cars.
When a clock rings, the car jumps by $1$ to the right if the
destination is not occupied by a car in front of it. 
In particular, the car in front performs a simple continuous
time Poisson random walk, while the motion of the second car
is more complicated than that of the first one due to blocking from the first car.

We consider two systems, the fast-slow (FS) and the
slow-fast (SF), depending on which car is in front. 

\begin{proposition}
	\label{prop:2cars_step_intro}
	If the cars start next to each other in both the FS and SF systems, then the probability law of the whole
	trajectory $\{ x_2(t) \}_{ t\in \mathbb{R}_{\ge0} }$ 
	of the car in the back is the same in both systems.
\end{proposition}
In other words, the trajectory of the car in the back, $x_2(t)$,
depends on the parameters $\alpha_0,\alpha_1$ in a symmetric way. See
\Cref{fig:2cars_step} for an illustration.

\begin{proof}[Idea of proof of \Cref{prop:2cars_step_intro}]
	This
	statement can be traced back to the
	RSK construction of TASEP (presented,
	e.g., in Vershik--Kerov \cite{Vershik1986} or O'Connell
	\cite{OConnell2003Trans}, \cite{OConnell2003}).
	The RSK shows that $(x_1(t),x_2(t))$
	is a deterministic function of the system of 
	two independent continuous time Poisson random walks
	with rates $\alpha_0$ and $\alpha_1$. 
	While these deterministic functions are
	different in the SF and the FS systems, 
	one can readily verify (for example, via the Bender--Knuth involution \cite{bender1972enumeration})
	that they produce the same distribution of the trajectory
	$\{x_2(t)\}_{t\in \mathbb{R}_{\ge0}}$.
\end{proof}

\begin{figure}[htpb]
	\centering
	\includegraphics[width=.9\textwidth]{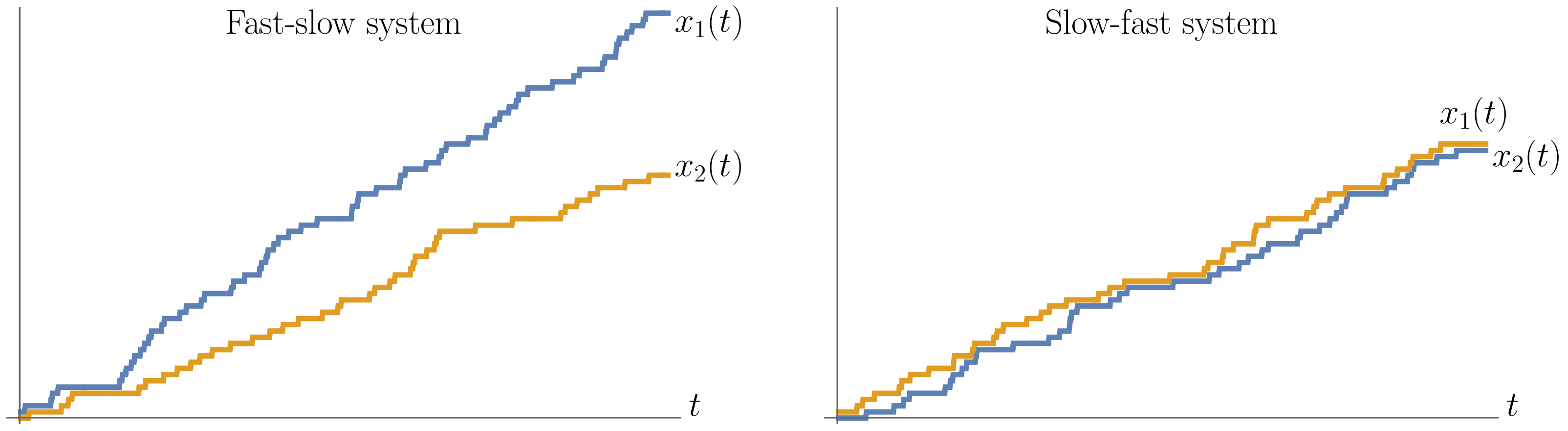}
	\caption{FS (left) and SF (right) systems of two cars
	started from the step initial configuration. In the FS
	system, the first car quickly runs to the front, and the
	evolution of the second (slow) car becomes an independent
	Poisson walk. In the SF system, the slow car in front often
	blocks the fast car in the back. However, in both systems,
	the trajectory of the car in the back, $\{x_2(t)\}_{t\in
	\mathbb{R}_{\ge0}}$, is the same in distribution.}
	\label{fig:2cars_step}
\end{figure}

The assumption that the cars start next to each other is
crucial for \Cref{prop:2cars_step_intro}. Indeed, consider
the initial condition $x_1^\circ>x_2^\circ$ for the FS and
SF systems so that $x_1^\circ-x_2^\circ$ is large. In the SF
system, the trajectory of the car in the back first evolves
as a random walk with the faster slope $\alpha_0$ and, then,
has slope $\alpha_1$ after catching up with the slow car.
This is very different from the behavior of the car in the
back in the FS system, where the slope is $\alpha_1$ the
whole time. So, \Cref{prop:2cars_step_intro} fails for an
arbitrary initial condition $(x_1^\circ,x_2^\circ)$.  See
\Cref{fig:2cars_general}, left and center, for an
illustration.

In this paper, we suitably modify
\Cref{prop:2cars_step_intro} to generalize it to arbitrary initial
conditions $(x_1^\circ,x_2^\circ)$. The modification involves
a randomization of the initial condition in the SF system. 
Define
\begin{equation*}
	y_1^\circ\coloneqq x_2^\circ+1+\min(G, x_1^\circ-x_2^\circ-1),
\end{equation*}
where $G$ is an independent geometric random variable 
with parameter 
$\alpha_1/\alpha_0$, that is,
\begin{equation}
	\label{eq:geometric_jump_operator}
	\mathbb{P}(G=k)=\left( 1-\alpha_1/\alpha_0 \right)(\alpha_1/\alpha_0)^{k},
	\qquad 
	k\in \mathbb{Z}_{\ge0}.
\end{equation}
The Markov map which turns $(x_1^\circ,x_2^\circ)$ into $(y_1^\circ,x_2^\circ)$ is an instance of the Markov swap operator $P^{(n)}$ 
(with $n=1$ here) entering \Cref{prop:swap_intro} below. 
For the next statement, the gap $x_1^\circ-x_2^\circ$ can be arbitrary, not necessarily large.
Denote the FS and SF systems with the corresponding initial conditions
by $\mathrm{FS}_{x_1^\circ,x_2^\circ}$ and $\mathrm{SF}_{y_1^\circ,x_2^\circ}$.

\begin{theorem}
	\label{thm:2cars_general_intro}
	The trajectory of the car in the back, $\{x_2(t)\}_{t\in \mathbb{R}_{\ge0}}$,
	is the same in distribution for
	$\mathrm{FS}_{x_1^\circ,x_2^\circ}$ and $\mathrm{SF}_{y_1^\circ,x_2^\circ}$,
	where $y_1^\circ$ (the initial condition for the car in the front in 
	SF) is random and given by \eqref{eq:geometric_jump_operator}.
\end{theorem}

See \Cref{fig:2cars_general} for an illustration. When $x_1^\circ=x_2^\circ+1$, from \eqref{eq:geometric_jump_operator} we almost surely have $y_1^\circ=x_1^\circ$, and so \Cref{thm:2cars_general_intro} reduces to \Cref{prop:2cars_step_intro}. For general initial conditions, the intertwining result, i.e.~\Cref{prop:swap_intro} introduced in \Cref{sub:intro_Lax} below, is not enough to conclude the equality in distribution of the whole trajectories. Namely, \Cref{prop:swap_intro} only implies the equality in distribution of $x_2(t)$ in $\mathrm{FS}_{x_1^\circ,x_2^\circ}$ and $\mathrm{SF}_{y_1^\circ,x_2^\circ}$ at each fixed time $t$, but not \emph{jointly} for all times. We need a stronger coupling between measures briefly described in \Cref{sub:intro_couplings} below. To point to the relevant results in the main text, \Cref{thm:2cars_general_intro} follows from the general \Cref{thm:coupling_trajectories_with_Pn_swap} and its continuous-time corollary, \Cref{prop:cont_time_qTASEP_down_result} (in particular, see \Cref{rmk:down_new_swap_TASEP} for the TASEP case).

\begin{figure}[htpb]
	\centering
	\includegraphics[width=\textwidth]{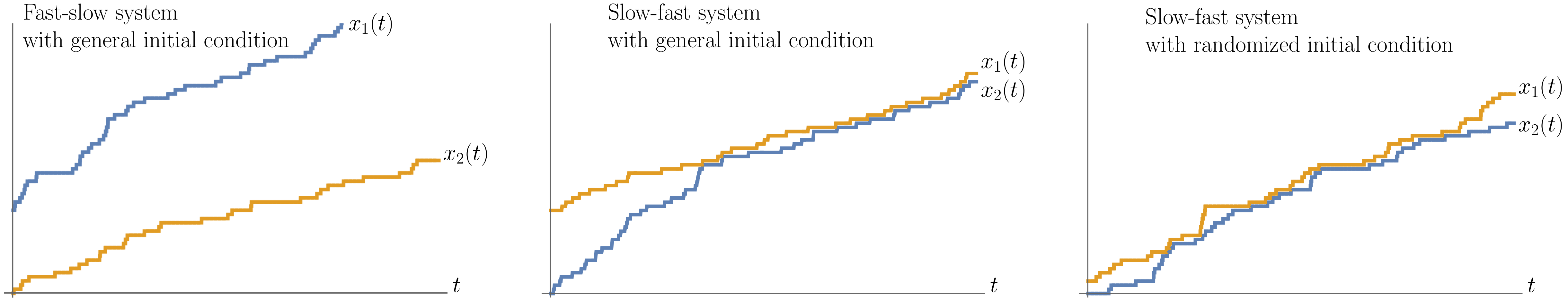}
	\caption{Left: The FS system started from a general fixed
	initial condition $(x_1^\circ,x_2^\circ)$. Center: The SF system started from the
	same initial condition $(x_1^\circ,x_2^\circ)$. Right: The SF system started
	from a randomized initial condition $(y_1^\circ,x_2^\circ)$
	depending on $(x_1^\circ,x_2^\circ)$. 
	The trajectories of
	the car in the back, $\{x_2(t)\}_{t\in \mathbb{R}_{\ge0}}$,
	are the same in distribution in the left and the right
	systems but are different from the trajectory in the center
	system.}
	\label{fig:2cars_general}
\end{figure}

Let us describe the coupling
between the two-particle systems 
$\mathrm{FS}_{x_1^\circ,x_2^\circ}$ and $\mathrm{SF}_{y_1^\circ,x_2^\circ}$ which leads to 
\Cref{thm:2cars_general_intro}.
Fix a terminal time $M\in \mathbb{R}_{\ge0}$. The coupling
(\emph{rewriting history operator from future to past}, in our terminology) replaces the 
trajectory $\{x_1(t)\}_{0\le t\le M}$ in 
$\mathrm{FS}_{x_1^\circ,x_2^\circ}$
by a new trajectory $\{y_1(t)\}_{0\le t\le M}$
such that $x_2(t)<y_1(t)\le x_1(t)$ for all~$t$ (so that the coupling is \emph{monotone}).
The construction of $y_1(t)$ proceeds in two steps:
\begin{enumerate}[$\bullet$]
	\item First, at time $t=M$, set $y_1(M)=x_2(M)+1+\min(G,x_1(M)-x_2(M)-1) $, 
		where $G$ is an independent geometric random variable with parameter $\alpha_1/\alpha_0$,
		as in \eqref{eq:geometric_jump_operator}.
	\item Then, start a continuous time Poisson random walk $y_1(t)$ in reverse time from $M$ to $0$ in the chamber
		$x_2(t)<y_1(t)\le x_1(t)$, under which the car $y_1$ jumps down by $1$ at rate
		$\alpha_0$ if $y_1(t)<x_1(t)$ and at rate $\alpha_0-\alpha_1$ if $y_1(t)=x_1(t)$.
		If $y_1(t)=x_2(t)+1$, the jump down is blocked. Also, if the top boundary of the chamber
		goes down by $1$, the walk $y_1(t)$ is deterministically pushed down.
\end{enumerate}

See \Cref{fig:down_intro} for an illustration.
\Cref{prop:cont_time_qTASEP_down_result} implies that
the distribution of the two-particle trajectory
$\{y_1(t),x_2(t)\}_{0\le t\le M}$
is the same as the trajectory of the whole two-particle system 
$\mathrm{SF}_{y_1^\circ,x_2^\circ}$.
In particular, the position $y_1(0)$ has the distribution
$y_1^\circ$ given by \eqref{eq:geometric_jump_operator}, 
and \Cref{thm:2cars_general_intro} follows.

\begin{figure}[htpb]
	\centering
	\includegraphics[width=.5\textwidth]{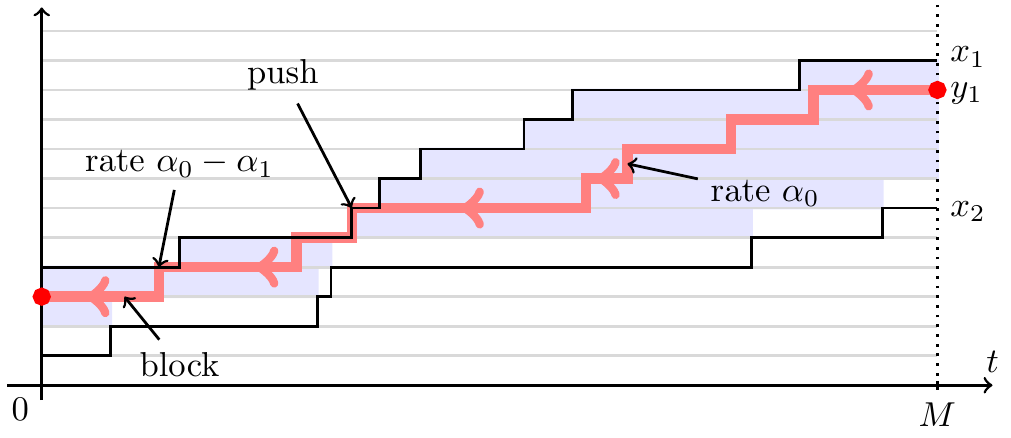}
	\caption{Construction of the coupling from FS to SF systems. The random walk
	$y_1(t)$ in reverse time is confined to the highlighted chamber. 
	In the bulk of the chamber, its rate of jumping down is $\alpha_0$, 
	and on the top boundary the rate is $\alpha_0-\alpha_1$.
	On the bottom boundary, the jumps down are blocked.}
	\label{fig:down_intro}
\end{figure}

The resampling of a trajectory as a random walk in a chamber
is reminiscent of the Brownian Gibbs property from
Corwin--Hammond \cite{corwin2014brownian}. However, there
are several notable differences. First, our resampling does
not preserve the distribution but interchanges the rates
$\alpha_0 \leftrightarrow \alpha_1$ (though in a continuous
time limit as in \Cref{sub:intro_Lax}, this issue
disappears). Second, the resampled trajectory is not a
simple random walk \emph{conditioned} to stay in the chamber
like in the Brownian Gibbs property, but rather a random
walk \emph{reflected} from the bottom wall and
\emph{sticking} to the top wall. In the Brownian setting,
processes sticking to one of the walls appeared in Warren
\cite{warren1997branching} and Howitt--Warren
\cite{howitt_warren2009dynamics} in a setting similar to
ours, see Section 1.3 in Petrov--Tikhonov
\cite{PetrovTikhonov2019} for a related discussion. Finally,
our resampling changes the trajectory and its endpoints,
while under the Brownian Gibbs property, the endpoints stay
fixed. We plan to explore Brownian limits of our
constructions in future work.

\subsection{Coupling of Poisson processes}
\label{sub:intro_Poisson}

In the previous \Cref{sub:intro_couplings}, we described a
monotone coupling from FS to SF, which turns the trajectory
$x_1(t)$, $0\le t\le M$, in FS having speed $\alpha_0$ into
a trajectory $y_1(t)$ which lies below $x_1(t)$ and has
lower speed $\alpha_1<\alpha_0$. The law of $y_1(t)$ depends
on the trajectory $x_2(t)$ of the particle in the back. 

Along with this monotone coupling, there is a monotone
coupling in another direction, \emph{rewriting history
operator from past to future}, which turns the particle's
trajectory in front of speed $\alpha_1$ into a trajectory
with a higher speed $\alpha_0>\alpha_1$. Recall that in
TASEP, the trajectories of the particles in front 
are continuous time Poisson simple random
walks, that is, they are \emph{counting functions of the standard
Poisson point processes} on $(0,+\infty)$. Remarkably, when
FS and SF start from the step initial configuration
$x_1(0)=y_1(0)=0$, $x_2(0)=-1$, the operator for rewriting
history from past to future does not depend on the
trajectory $x_2(t)$. Thus, it produces a monotone coupling
between two Poisson simple random walks of rates $\alpha_1$
and $\alpha_0>\alpha_1$, respectively. Let us describe this
coupling.

Let $y_1(t)$, $t\in \mathbb{R}_{\ge0}$, be a Poisson simple
random walk of rate $\alpha_1$ started from $y_1(0)=0$. Fix
$\alpha_0>\alpha_1$. Start another random walk $x_1(t)$
from $x_1(0)=0$ which lives in the chamber $x_1(t)\ge
y_1(t)$ for all $t$, and jumps up by $1$ at rate $\alpha_0$
if $x_1(t)>y_1(t)$, and at rate $\alpha_0-\alpha_1$ if
$x_1(t)=y_1(t)$. When the bottom boundary of the chamber
goes up by $1$ and $x_1(t)=y_1(t)$, the walk $x_1(t)$ is deterministically
pushed up. See \Cref{fig:Poisson_coupling_intro}
for an illustration.

\begin{theorem}
	\label{thm:Poisson_discrete_intro}
	The process $x_1(t)$ defined above (right before the Theorem) is a continuous time Poisson simple random walk with rate
	$\alpha_0$.
\end{theorem}

\begin{figure}[htpb]
	\centering
	\includegraphics[width=.5\textwidth]{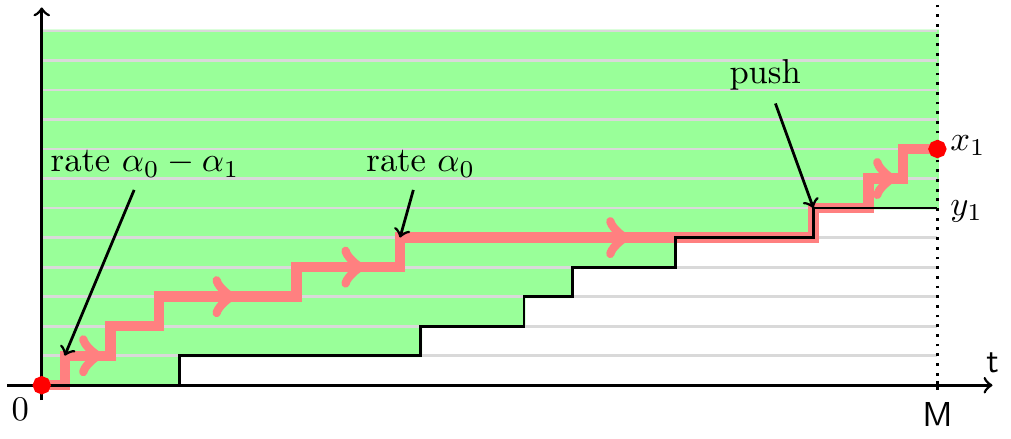}
	\caption{Monotone coupling from a Poisson random walk 
		$y_1(t)$
		of rate $\alpha_1$ to a Poisson random walk
		$x_1(t)$ of rate $\alpha_0>\alpha_1$. 
		The new process must lie in the highlighted chamber.
		In the bulk of the chamber, it has jump rate $\alpha_0$,
		and on the bottom boundary the jump rate is $\alpha_0-\alpha_1$.}
	\label{fig:Poisson_coupling_intro}
\end{figure}

\Cref{thm:Poisson_discrete_intro}
follows from our main result on coupling,
\Cref{thm:coupling_trajectories_with_Pn_swap}, and its
corollary for rewriting history from past to future in continuous time, \Cref{prop:cont_time_qTASEP_up_result} (in
particular, see \Cref{rmk:up_new_swap_TASEP} for the TASEP case).

\begin{remark} 
	We note that there is a well-known and simple randomized coupling between two
	Poisson processes on $(0,+\infty)$ (from a rate $\alpha_1$
	to the higher rate $\alpha_0>\alpha_1$). It
	is called \emph{thickening} and consists of simply adding
	to the process of rate $\alpha_1$ an extra independent
	Poisson point process configuration of rate
	$\alpha_0-\alpha_1$. The union of the two point
	configurations is a Poisson process of rate $\alpha_0$.
	The construction in \Cref{thm:Poisson_discrete_intro} is
	very different from such an independent thickening: it
	does not preserve all the points from the original process
	of rate $\alpha_1$, and has a Markov (and not independent) nature.
	
	Surprisingly, there also exist \emph{deterministic}
	couplings between Poisson processes on the entire line
	$\mathbb{R}$ from higher to lower rates which are
	translation invariant (constructed by Ball \cite{ball2005poisson}),
	and also non-translation invariant ones on an arbitrary set in both $\alpha_0$
	to $\alpha_1$ and the reverse directions, see
	Angel--Holroyd--Soo \cite{angel2011deterministic} and
	Gurel-Gurevich--Peled \cite{gurel2013poisson}.
\end{remark}

It is possible to make the Poisson rates $\alpha_0$ and
$\alpha_1$ equal to each other by taking a continuous time
Poisson type limit. This limit produces an interesting (and,
to the best of our knowledge, new) continuous time coupling
between Poisson processes on $(0,+\infty)$ of all possible
rates. This coupling is also monotone; that is, it increases
the rate while almost surely increasing the trajectory of
the Poisson process' counting function. This continuous time
coupling is described in \Cref{prop:neat_last_fact}.

\subsection{Intertwining relations for stochastic vertex models} 
\label{sub:intro_vertex_intertwining}

After highlighting two concrete applications in \Cref{sub:intro_2_cars,sub:intro_Poisson}, let us present our results in a more general setting.  

In the setting of the fully fused higher spin stochastic six
vertex model, we prove an \emph{intertwining} (also called
\emph{quasi-commutation}) relation between the transfer
matrices of two models, which differ by a permutation of the
speed parameters. This vertex model is defined in
Corwin--Petrov \cite{Corwin2014qmunu} and Borodin--Petrov
\cite{BorodinPetrov2016inhom}; we recall it in
\Cref{sec:stoch_vert_models,sec:YBE} below. We formulate the
intertwining as follows. Let us denote by $T$ and
$T_{\sigma_{n-1}}$ the one-step Markov operators (transfer
matrices) on the space $\left\{ 0,1 \right\}^{\mathbb{Z}}$
of particle configurations on $\mathbb{Z}$. The parameter
sequences in $T$ and $T_{\sigma_{n-1}}$ differ by the
elementary transposition $\sigma_{n-1}=(n-1,n)$ which
permutes the parameters associated with two neighboring
particles $x_n > x_{n+1}$.\footnote{Note that we shift the
indices for a better correspondence between particle systems
and vertex models.}

\begin{proposition}[\Cref{prop:Pn_T_intertw} in the text]
	\label{prop:swap_intro}
	There exists a one-step Markov transition operator denoted by $P^{(n)}$ such that
	\begin{equation}
		\label{eq:general_intertwining_intro}
		T\ssp P^{(n)}=P^{(n)}\ssp T_{\sigma_{n-1}},
	\end{equation}
	under a certain restriction on the parameters associated with the particles $x_n$ and $x_{n+1}$.
\end{proposition}
The action of the Markov operator $P^{(n)}$ only moves the particle $x_n$ while preserving the locations of all other particles. Here and throughout the paper, we interpret the product of Markov operators as acting on measures from the right. That is, \eqref{eq:general_intertwining_intro} states that if we start from a fixed particle configuration, apply a random Markov step according to $T$, and then apply a Markov step according to $P^{(n)}$, then the resulting random particle configuration has the same distribution as the random particle configuration obtained by the action of $P^{(n)}$ followed by $T_{\sigma_{n-1}}$.

The intertwining relation \eqref{eq:general_intertwining_intro} is a consequence of the Yang-Baxter equation for the higher spin stochastic six vertex model. Our crucial observation is that under certain restrictions on the parameters, the intertwiner $P^{(n)}$ (coming from the corresponding R-matrix for the vertex model) is itself a one-step Markov transition operator. The restrictions on the parameters are required to make the transition probabilities of $P^{(n)}$ nonnegative.

Furthermore, we show
that a sequential application of the
operators $P^{(n)}$ over all $n=1,2,3,\ldots $ (denoted by $B$) intertwines
the transfer matrix $T$ with another transfer matrix $T_{\mathrm{shift}}$
obtained from $T$ by the one-sided shift of the parameter
sequence (which eliminates the first of the parameters with index $0$):
\begin{equation}
	\label{eq:shift_intertwining_intro}
	TB=B\ssp T_{\mathrm{shift}}.
\end{equation}
See \Cref{thm:B_T_intertw_relation} in the text.
We call $B$ the \emph{Markov shift operator}.
See
\Cref{sec:comm_rel}
for detailed formulations and proofs of the
general intertwining relations \eqref{eq:general_intertwining_intro}, \eqref{eq:shift_intertwining_intro}.

\subsection{Intertwining and Lax equation for the continuous time $q$-TASEP}
\label{sub:intro_Lax}

In \Cref{sub:intro_vertex_intertwining}, the transfer
matrices $T,T_{\mathrm{shift}}$ and the intertwiner $B$
are one-step Markov transition operators. 
It is well-known that the vertex model transfer matrices $T$ admit a
Poisson type limit 
to the $q$-TASEP.

Recall that the $q$-TASEP, introduced in Borodin--Corwin
\cite{BorodinCorwin2011Macdonald}, is a continuous time
Markov chain on particle configurations $\mathbf{x}=(x_1>x_2>x_3>\ldots)$ 
in $\mathbb{Z}$. Each particle $x_n$ has an
independent exponential clock of rate\footnote{That is, the random time $\xi$
after which the clock rings is distributed as
$\operatorname{Prob}(\xi>s)=$ $e^{-\lambda s}, s>0$, where
$\lambda=\alpha_{n-1}(1-q^{x_n-x_{n+1}-1})$ is the rate.}
$\alpha_{n-1}(1-q^{x_n-x_{n+1}-1})$. When the clock attached to
$x_n$ rings, this particle jumps by $1$ to the right. Note
that the jump rate of $x_n$ is zero 
when $x_n=x_{n+1}+1$, meaning that a particle cannot jump into an occupied location.
When $q=0$, the $q$-TASEP turns into the usual TASEP in
which each particle jumps to the right at rate $1$ unless the destination is occupied. 

Now, take the particle speeds in the $q$-TASEP to be the
geometric progression, $\alpha_j=r^j$, $j\in
\mathbb{Z}_{\ge0}$, where $0<r<1$. Sending $r\to1$ leads to
the $q$-TASEP with homogeneous speeds. Let
$\{T(t)\}_{t\in \mathbb{R}_{\ge0}}$ denote the corresponding Markov
transition semigroup. Before taking the limit $r \to 1$, the
application of the Markov shift operator $B$ as in \eqref{eq:shift_intertwining_intro}
turns the
sequence of speeds $(1,r,r^2,\ldots )$ into
$(r,r^2,r^3,\ldots )$.  This shift is the same as
multiplying all particle speeds by $r$ or, equivalently,
turning the time parameter $t$ into $rt$. Taking a second
Poisson-type limit in $B$ as $r\to 1$, we obtain an
intertwining relation for the continuous time $q$-TASEP with
homogeneous speeds.

Let us now describe the continuous time limit of the Markov shift operators.
This is a Markov semigroup
$\{B(\tau)\}_{\tau\in \mathbb{R}_{\ge0}}$
on the space of \emph{left-packed} particle
configurations $\mathbf{x}=(x_1>x_2>\ldots)$, i.e.,~configurations with
$x_n=-n$ for all sufficiently large $n$. 
Under $B(\tau)$, each 
particle $x_n$ has an independent
exponential clock with rate $n(x_n-x_{n+1}-1)$.
When a clock rings, 
the corresponding particle $x_n$ instantaneously jumps backwards to a new location
$x_n'$, $x_{n+1}< x_n'<x_n$, with probability
\begin{equation}
	\label{eq:q_Hammersley_process}
	\frac{1}{(x_n-x_{n+1}-1)(1-q^{x_n-x_n'})}
	\frac
	{\prod_{i=1}^{x_n-x_{n+1}-1}(1-q^i)}
	{\prod_{i=1}^{x_n'-x_{n+1}-1}(1-q^i)}.
\end{equation}
In particular, the particles almost surely jump to the left.
For left-packed configurations, the sum of the jump rates of
all possible particle jumps is finite, meaning that
$B(\tau)$ is well-defined. The process $B(\tau)$ was
introduced in Petrov \cite{petrov2019qhahn}, and its $q=0$ version
(for which the probabilities \eqref{eq:q_Hammersley_process} become uniform)
appeared under the name \emph{backwards Hammersley process}
in Petrov--Saenz \cite{PetrovSaenz2019backTASEP}. See \Cref{fig:intro_BHP}
for an illustration of the latter process.

\begin{figure}[htpb]
	\centering
	\includegraphics{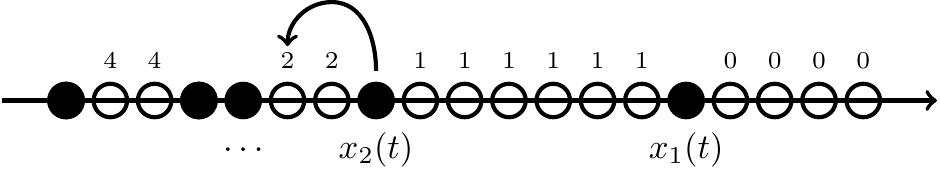}
		\caption{The $q=0$ version of the backwards process
		$B(\tau)$ may be alternatively defined as follows. Each
		hole in $\mathbb{Z}$ has an independent exponential clock
		with a rate equal to the number of particles to the right of
		this hole. When the clock at a hole rings, the leftmost of
		the particles to the right of the hole instantaneously jumps
		into this hole.}
	\label{fig:intro_BHP}
\end{figure}

We prove
the following intertwining relation between the $q$-TASEP and the backwards $q$-TASEP processes:
\begin{theorem}[\Cref{thm:mapping_qTASEP_TASEP_back_general_IC} in the text]
	\label{thm:cont_intertwining_intro}
	For any $t,\tau\in \mathbb{R}_{\ge0}$, we have
	\begin{equation}
		\label{eq:intro_qTASEP_intertwining}
		T(t)\ssp B(\tau) = B(\tau)\ssp T(e^{-\tau}t).
	\end{equation}
\end{theorem}

Let us reformulate the intertwining relation
\eqref{eq:intro_qTASEP_intertwining} in probabilistic terms.
Fix a left-packed configuration $\mathbf{y}$, and let $\delta_{\mathbf{y}} B(\tau)$ be a \emph{random} configuration obtained from $\mathbf{y}$ by running the backwards
$q$-TASEP dynamics for time~$\tau$. Then, denote by $\mathbf{x}(t)$ the configuration of the $q$-TASEP at
time $t$ started with initial condition
$\mathbf{x}(0)=\mathbf{y}$. Now, fix $\tau$, and run the
backwards $q$-TASEP dynamics from the configuration
$\mathbf{x}(\mathsf{t})$ for time $\tau$. Then, the
distribution of the resulting configuration is \emph{the
same} as the distribution of the $q$-TASEP at time
$e^{-\tau}\mathsf{t}$ but started from the \emph{random}
initial configuration $\delta_{\mathbf{y}} B(\tau)$.

Identity \eqref{eq:intro_qTASEP_intertwining} reduces
to the result obtained earlier in 
Petrov--Saenz \cite{PetrovSaenz2019backTASEP}
and
Petrov \cite{petrov2019qhahn} when applied to the process started from the distinguished
\emph{step} initial configuration
$\mathbf{y}=\mathbf{x}_{step}$ for which $x_n(0)=-n$ for all
$n\in \mathbb{Z}_{\ge1}$. That is, 
\begin{equation}
	\label{eq:intro_qTASEP_intertwining_step}
	\delta_{\mathbf{x}_{step}} 
	\ssp
	T(t)\ssp B(\tau)
	=
	\delta_{\mathbf{x}_{step}} 
	\ssp
	T(e^{-\tau}t).
\end{equation}
Indeed, the action of $B(\tau)$ preserves the step configuration since it moves all the particles left. In probabilistic terms, \eqref{eq:intro_qTASEP_intertwining_step} shows that the backwards process $B(\tau)$ produces a coupling in the reverse time direction of the fixed-time distributions of the $q$-TASEP with the step initial configuration.

\medskip

We arrive at the following \emph{Lax equation} for the $q$-TASEP semigroup:
\begin{equation}
	\label{eq:intro_Lax}
	t\ssp \frac{d}{dt}\ssp T(t)=\left[ \mathsf{B},T(t) \right],
\end{equation}
where $\mathsf{B}$ is the infinitesimal Markov generator of the backwards $q$-TASEP. This follows by differentiating \eqref{eq:intro_qTASEP_intertwining} 
in $\tau$ at $\tau=0$ and slightly rewriting the result using the Kolmogorov (a.k.a~the Fokker--Planck) equation. Equivalently, in terms of expectations, for any left-packed initial configuration $\mathbf{y}$ and a generic function $F$ of the configuration, we have the following evolution equation for the observables:
\begin{equation}
	\label{eq:intro_Lax_2}
	t\ssp
	\frac{d}{dt}\ssp
	\mathbb{E}_{\mathbf{y}}\left( F(\mathbf{x}(t)) \right)=
	\mathsf{B} \ssp \mathbb{E}_{\mathbf{y}}
	\left[ F(\mathbf{x}(t)) \right]-
	\mathbb{E}_{\mathbf{y}}\left[ (\mathsf{B}\ssp F)(\mathbf{x}(t)) \right].
\end{equation}
We use the notation $\mathbb{E}_{\mathbf{y}}$ to denote the expectation with respect to the $q$-TASEP started from $\mathbf{y}$. The first generator $\mathsf{B}$ in the right-hand side of \eqref{eq:intro_Lax_2} acts on $\mathbb{E}_{\mathbf{y}} \left[ F(\mathbf{x}(t)) \right] $ as a function of $\mathbf{y}$, while the second one acts on the function $F$.

It is intriguing that while $\mathsf{B}$ and $T(t)$ depend on $q$, the form of the Lax equations \eqref{eq:intro_Lax}--\eqref{eq:intro_Lax_2} is \emph{the same} for the $q$-TASEP and its $q=0$ specialization, the TASEP. Though the structure and asymptotics of multipoint observables of TASEP is well-studied by now (e.g., see Liu \cite{liu2022multipoint}, and Johansson--Rahman \cite{JohanssonRahman2019}), its extension to $q$-TASEP is mostly conjectural at this point, see Dotsenko \cite{dotsenko2013two}, Prolhac--Spohn \cite{prolhac2011two}, Imamura--Sasamoto--Spohn \cite{imamura2013equal}, and Dimitrov \cite{dimitrov2020two} for related results.

We believe that our Lax equation could be employed to study multipoint asymptotics of the $q$-TASEP and, in a scaling limit, lead to Kadomtsev–Petviashvili (KP) or Korteweg–de Vries (KdV) type equations for limits of the observables \eqref{eq:intro_Lax_2}. The KP and KdV equations were recently derived by Quastel--Remenik \cite{quastel2019kp} for the KPZ fixed point process introduced earlier by Matetski--Quastel--Remekin \cite{matetski2017kpz}. We leave the asymptotic analysis of the Lax equation to future work.

\subsection{Coupling of measures on trajectories}
\label{sub:intro_couplings}

Let us briefly outline the scope of couplings between trajectories of various integrable stochastic particle systems obtained in the present paper. All these couplings, including the examples from \Cref{sub:intro_2_cars,sub:intro_Poisson}, are obtained from the intertwining relations like \eqref{eq:general_intertwining_intro}, \eqref{eq:intro_qTASEP_intertwining} through the \emph{bijectivisation} procedure. This idea originated in Diaconis--Fill \cite{DiaconisFill1990} and was later developed in the context of integrable stochastic particle systems in Borodin--Ferrari \cite{BorFerr2008DF}, Borodin--Gorin \cite{BorodinGorin2008}, Borodin--Petrov \cite{BorodinPetrov2013NN}, and Bufetov--Petrov \cite{BufetovPetrovYB2017}. The building block of all the intertwining relations is the Yang-Baxter equation. We first apply the bijectivisation procedure to the Yang-Baxter equation and obtain elementary Markov steps. They are conditional distributions corresponding to a coupling between two marginal distributions coming from two sides of the Yang-Baxter equation. The bijectivisation of the Yang-Baxter equation is not unique because the coupling is not unique. We focus on the simplest case, which has the ``maximal noise''. This is the case that introduces the most randomness and independence. We recall these constructions in \Cref{sec:bijectivisation}. Let us now describe the couplings we obtain.

\medskip

First, we begin with the intertwining relation
\eqref{eq:general_intertwining_intro} for the fully general fused stochastic
higher spin six vertex model. Graphically it is represented as follows:
\begin{equation*}
	T\ssp P^{(n)}\stackrel{\eqref{eq:general_intertwining_intro}}{=}P^{(n)}\ssp T_{\sigma_{n-1}}
	\qquad \Leftrightarrow
	\qquad 
	\begin{tikzpicture}[baseline=-4pt]
		\def\h{1.6}
		\def\v{.6}

		\node (g0) at (0,\v) {$\mathbf{x}$};
		\node (g1) at (1*\h,\v) {$\mathbf{x}'$};
		\node (gh0) at (0,   -1*\v) {$\mathbf{y}$};
		\node (gh1) at (1*\h,-1*\v) {$\mathbf{y}'$};

		\draw [densely dashed,thick,->] (g0)--(g1) node[midway,yshift=15pt,anchor=north] {$T$};
		\draw [densely dashed,thick,->] (gh0)--(gh1) node[midway,yshift=19pt,anchor=north] {$T_{\sigma_{n-1}}$};
		\draw [densely dashed,thick,->] (g0)--(gh0) node[midway,anchor=east] {$P^{(n)}$};
		\draw [densely dashed,thick,->] (g1)--(gh1) node[midway,anchor=west] {$P^{(n)}$};
	\end{tikzpicture}
\end{equation*}
In words,
fix a particle configuration $\mathbf{x}$, apply the Markov operator $T$, and then $P^{(n)}$.
Relation \eqref{eq:general_intertwining_intro} implies that the
distribution of the resulting random configuration $\mathbf{y}'$ 
is the same as if we first applied $P^{(n)}$ and then $T_{\sigma_{n-1}}$.
\Cref{sub:up_down_on_configurations} defines
two couplings based on the intertwining \eqref{eq:general_intertwining_intro}.
The coupling $D^{(n)}$ from the left- to the right-hand side 
of \eqref{eq:general_intertwining_intro}
(\emph{future to past} in our terminology)
samples $\mathbf{y}$ given $\mathbf{x},\mathbf{x}',\mathbf{y}'$.
The coupling $U^{(n)}$ in the opposite direction (\emph{past to future} in our terminology)
samples $\mathbf{x'}$ given $\mathbf{x},\mathbf{y},\mathbf{y}'$.
Both couplings are compatible with \eqref{eq:general_intertwining_intro}
in the sense that they satisfy a \emph{detailed balance equation}; see
\eqref{eq:Dn_Un_detailed_balance_equation} in the text.

\medskip

Iterating the couplings 
$D^{(n)},U^{(n)}$
over time, we obtain a bijectivisation of
the relation $T^k
P^{(n)}=P^{(n)}\ssp T_{\sigma_{n-1}}^k$ for any time
$k\in\mathbb{Z}_{\ge1}$. This
leads to two discrete-time Markov operators for
rewriting history (in two directions). Their general
construction is given in
\Cref{def:rewriting_history_operators} after
\Cref{thm:coupling_trajectories_with_Pn_swap}, and the
general definitions are expanded in terms of particle
systems in \Cref{sec:discrete_time_bijectivisation}.

We write down
concrete operators for rewriting history in $q$-TASEP 
in
\Cref{sub:rewriting_history_qTASEP_new,sub:rewriting_history_qTASEP_UP_new} by keeping $n$ fixed,
taking the Poisson type limit as $k\to\infty$ to continuous
time, and specializing to $q$-TASEP.
For $q=0$ and $n=1$, these rewriting history operators give rise to the coupling results
for the two-particle TASEP and Poisson processes
presented in \Cref{sub:intro_2_cars,sub:intro_Poisson}
above.

\medskip

Furthermore, we get a bijectivisation of the intertwining relation $T^kB=B\ssp T_{\mathrm{shift}}^k$ for any time $k\in \mathbb{Z}_{\ge1}$. This relation is an iteration of \eqref{eq:shift_intertwining_intro} over $k$. Namely, a bijectivisation of $T^kB=B\ssp T_{\mathrm{shift}}^k$ is obtained by iterating $D^{(n)}, U^{(n)}$ over $n=1,2,\ldots $ and then iterating over time $k$, respectively. 

Next, we specialize to $q$-TASEP, take the Poisson type limit as $k\to\infty$ to continuous time, and further take another Poisson type limit in the particle speeds $\alpha_j=r^j$ as $r\to1$ as explained in \Cref{sub:intro_Lax}. This leads to a bijectivisation of the continuous time relation $T(t)\ssp B(\tau)=B(\tau)\ssp T(e^{-\tau}t)$. The latter bijectivisation is a pair of continuous time Markov processes on the space of $q$-TASEP trajectories, which either speeds up or slows down the time in the process with homogeneous particle speeds. These rewriting history processes in continuous time are constructed and described in \Cref{sec:ind_bij_cont_time_new}. See \Cref{prop:action_Xi_down_cont_new,prop:action_Xi_up_cont} for the main results.

\medskip

While our exploration of couplings is extensive, in this paper we only
describe some of the possible constructions. One can
continue our methods in the following directions:
\begin{enumerate}[$\bullet$]
	\item 
		Colored stochastic higher spin vertex models, introduced and
		studied in Borodin--Wheeler
		\cite{borodin_wheeler2018coloured} and further works, also
		possess stochastic Yang-Baxter equations leading to
		intertwining, couplings, and corresponding Lax equations.
	\item 
		Within uncolored systems (the setting of the present
		paper), there are two natural directions. First, taking
		different bijectivisations of the Yang-Baxter equation
		(which are not maximally independent) could produce more
		couplings of particle system trajectories and Poisson
		processes with other nontrivial properties. 
	\item 
		We
		mainly restricted our couplings to the $q$-TASEP, for which
		the intertwiner $B$ preserves the distinguished step
		configuration. In a second natural
		direction within uncolored systems,
		focusing on the 
		Schur vertex model (discussed in
		\Cref{sec:schur_vertex_model} below), we see that the intertwiner \emph{does
		not} preserve the step configuration. Thus, the resulting
		couplings would not be monotone. It would be interesting to
		see which probabilistic properties these couplings still
		satisfy. 
\end{enumerate}
We plan to address these directions in future work.

\subsection{Outline}
\label{sub:outline}

The paper consists of two parts. In the first part, we derive intertwining relations and study their consequences. In more detail, in \Cref{sec:stoch_vert_models}, we recall the stochastic higher spin six vertex models, and in \Cref{sec:YBE}, write down a ``vertical'' Yang-Baxter equation for them. We also investigate conditions under which the cross vertex weights (that is, the R-matrix) are nonnegative and thus lead to Markov transition operators. Then, in \Cref{sec:comm_rel}, we prove our main intertwining results which in full generality follow directly from the Yang-Baxter equation. In \Cref{sec:comm_rel_specializations} and \Cref{sec:schur_vertex_model}, we specialize the general intertwining relations to concrete particle systems such as the $q$-Hahn TASEP, the $q$-TASEP, the TASEP, and the Schur vertex model.

In the second part, we use the intertwining relations to construct couplings between probability measures on trajectories of particle systems which differ by a permutation of the speed parameters. The couplings are based on the bijectivisation procedure for the Yang-Baxter equation, which we review in \Cref{sec:bijectivisation}. The conditional distribution for such a coupling is realized as a ``rewriting history'' process that randomly resamples a particle system's trajectory. In \Cref{sec:discrete_time_bijectivisation}, we construct rewriting history processes for general discrete time integrable stochastic interacting particle systems. In \Cref{sec:ind_bij_cont_time_new}, we specialize our constructions to concrete rewriting history dynamics for the continuous time $q$-TASEP and TASEP. Finally, in \Cref{sec:limit_equal_speeds_new}, we take a limit of our couplings to the case of homogeneous particle speeds, making the rewriting history dynamics evolve in continuous time. As a byproduct, in \Cref{sub:TASEP_equal_speeds}, we construct a new coupling of the standard Poisson processes on the positive real half-line with different rates.

\subsection{Acknowledgments}
\label{sub:acknowledgements}

LP is grateful to Alexei Borodin and Douglas Rizzolo for helpful discussions. The work of LP was partially supported by the NSF grants DMS-1664617 and DMS-2153869, and the Simons Collaboration Grant for Mathematicians 709055. This material is based upon work supported by the National Science Foundation under grant DMS-1928930 while both authors participated in the program ``Universality and Integrability in Random Matrix Theory and Interacting Particle Systems'' hosted by the Mathematical Sciences Research Institute in Berkeley, California, during the Fall 2021 semester.

\newpage
\part{Intertwining Relations for Integrable Stochastic Systems}

In the first part, we obtain new 
intertwining
relations 
between
transfer matrices (viewed as one-step Markov transition operators)
of the stochastic higher spin six vertex model
with different sequences of parameters. The intertwining
operators come from the R-matrix in the
vertical Yang-Baxter equation, and are also Markov transition operators.

\section{Stochastic higher spin six vertex model and exclusion process}
\label{sec:stoch_vert_models}

Here we recall the most general 
integrable 
stochastic
particle system considered in the paper,
in both vertex model and 
exclusion process
settings.
This material is well-developed in several works on 
stochastic vertex models.
In our exposition
we follow
\cite{CorwinPetrov2015arXiv} and
\cite{BorodinPetrov2016inhom}.

\subsection{The $q$-deformed beta-binomial distribution}
\label{sub:phi_distribution}

We need
the \emph{$q$-deformed beta-binomial distribution}
$\varphi_{q,\mu,\nu}$ from \cite{Povolotsky2013}, \cite{Corwin2014qmunu}.
Let $q\in[0,1)$.
Throughout the paper, we use the following notation for the 
$q$-Pochhammer symbols
\begin{equation}
 \label{eq:qPochhammer}
 (a;q)_k\coloneqq 
 (1-a)(1-aq)\ldots
 (1-aq^{k-1}),\quad k\ge1;
 \qquad 
 (a;q)_0\coloneqq 1,
 \qquad 
 (a;q)_{\infty}\coloneqq\prod_{i=0}^{\infty}(1-aq^i).
\end{equation}
For $k\le -1$, we use the standard convention 
\begin{equation}
 \label{eq:qPochhammer_negative}
 (a;q)_k=\frac{1}{(a/q;1/q)_{-k}}.
\end{equation}
For $m\in \mathbb{Z}_{\ge0}$, consider the following 
distribution on 
$\left\{ 0,1,\ldots,m \right\}$:
\begin{equation}
 \varphi_{q,\mu,\nu}(j\mid m)=
 \mu^j\,\frac{(\nu/\mu;q)_j(\mu;q)_{m-j}}{(\nu;q)_m}
 \frac{(q;q)_m}{(q;q)_j(q;q)_{m-j}}
 ,
 \qquad 
 0\le j\le m.
 \label{eq:phi_finite}
\end{equation}
Throughout the paper, we sometimes write
$\varphi_{q,\mu,\nu}(j\mid m)$ when $j>m$ or $j<0$, and agree that this expression equals zero in those cases.

When $m=+\infty$, extend the definition as
\begin{equation}
 \label{eq:phi_infinite}
 \varphi_{q,\mu,\nu}(j\mid \infty)=
 \mu^j\frac{(\nu/\mu;q)_j}{(q;q)_j}\frac{(\mu;q)_\infty}{(\nu;q)_\infty},
 \qquad 
 j\in \mathbb{Z}_{\ge0}.
\end{equation}
The quantities \eqref{eq:phi_finite} and \eqref{eq:phi_infinite} sum to one:
\begin{equation*}
 \sum_{j=0}^{m}\varphi_{q,\mu,\nu}(j\mid m)=1,\qquad m \in\left\{ 0,1,\ldots \right\}
 \cup\left\{ +\infty \right\}.
\end{equation*}

The distribution $\varphi_{q,\mu,\nu}$ depends on $q\in[0,1)$ 
and two other parameters~$\mu,\nu$.
We will use the following two cases in which
the weights
$\varphi_{q,\mu,\nu}(j\mid m)$ are nonnegative (and hence define a probability distribution):\footnote{These two cases do not
exhaust the full range of parameters $(q,\mu,\nu)$
for which the weights are nonnegative.
See, e.g., \cite[Section 6.6.1]{BorodinPetrov2016inhom}
for additional cases leading to nonnegative weights.}
\begin{align}
 \label{eq:phi_cases_nonnegative_1}
 &\textnormal{$0\le \mu\le 1$ and $\nu\le \mu$;}
 \\
 \label{eq:phi_cases_nonnegative_2}
 &\textnormal{$\nu\le 0$ and $\mu=q^J\nu$ for some $J\in \mathbb{Z}_{\ge0}$.}
\end{align}

\subsection{Stochastic vertex weights}
\label{sub:stochastic_fused_weights}

We consider the \emph{stochastic higher spin six vertex weights}
$L^{(J)}_{u,s}$ which depend on the following parameters:
\begin{equation}
 \label{eq:parameters_for_LJ_with_restrictions}
 q\in [0,1),
 \qquad 
 u\in [0,+\infty), 
 \qquad 
 s\in (-1,0],
 \qquad J\in \mathbb{Z}_{\ge1}.
\end{equation}
Here $q$ is the main ``quantum'' parameter, fixed throughout the paper, and all other parameters may vary from vertex to vertex.
The weights $L_{u,s}^{(J)}(i_1,j_1;i_2,j_2)$ are
indexed by a quadruple of integers, where
$i_1,i_2 \in \mathbb{Z}_{\ge0}$ and $j_1,j_2\in \left\{ 0,1,\ldots,J \right\}$,
and are defined as
\begin{equation}
 \label{eq:LJ_definition}
 \begin{split}
 &
 L^{(J)}_{u,s}(i_1,j_1;i_2,j_2)\coloneqq
 \mathbf{1}_{i_1+j_1=i_2+j_2}
 \frac{(-1)^{i_1}q^{
 \frac12 i_1(i_1+2j_1-1)}
 u^{i_1}s^{j_1+j_2-i_2}(u s^{-1};q)_{j_2-i_1}}
 {(q;q)_{i_2} (s u;q)_{i_2+j_2}
 (q^{J+1-j_1};q)_{j_1-j_2}}
 \\&\hspace{220pt}
 \times{}_{4}\bar{\phi}_3\left(\begin{matrix} q^{-i_2};q^{-i_1},s u q^{J},q s/u\\
 s^{2},q^{1+j_2-i_1},q^{J+1-i_2-j_2}\end{matrix}
 \bigg|\, q,q\right).
 \end{split}
\end{equation}
Here and throughout the paper
the notation $\mathbf{1}_{A}$ means
the indicator of an event or a condition $A$, and 
$_4\bar{\phi}_3$ is the regularized
(terminating)
$q$-hypergeometric series, where
\begin{equation}
 \label{eq:r_phi_r_reg_function}
 \begin{split}
 {}_{r+1}\bar{\phi}_r
 \left(
 \begin{matrix} 
 q^{-n};a_1,\ldots,a_r \\
 b_1, \ldots ,b_r
 \end{matrix}
 \bigg|\, q,z\right)
 &\coloneqq
 {}_{r+1}{\phi}_r
 \left(
 \begin{matrix} 
 q^{-n},a_1,\ldots ,a_r\\
 b_1,\ldots ,b_r
 \end{matrix}
 \bigg|\, q,z\right)
 \prod_{i=1}^{r}(b_i;q)_n
 \\
 &=
 \sum_{k=0}^{n}
 \frac{z^k (q^{-n};q)_k}{(q;q)_k}
 \prod_{i=1}^{r}(a_i;q)_k(b_iq^k;q)_{n-k}.
 \end{split}
\end{equation}
The condition that $L^{(J)}_{u,s}(i_1,j_1;i_2,j_2)$ vanishes unless $i_1+j_1=i_2+j_2$ is the \emph{path conservation property}: the total number of incoming paths (from below and from the left) is equal to the total number of outgoing paths (to the right and upwards) at a vertex; see \Cref{fig:LJ_vertex} for an illustration.

\begin{figure}[htb]
 \centering
 \includegraphics[width=.75\textwidth]{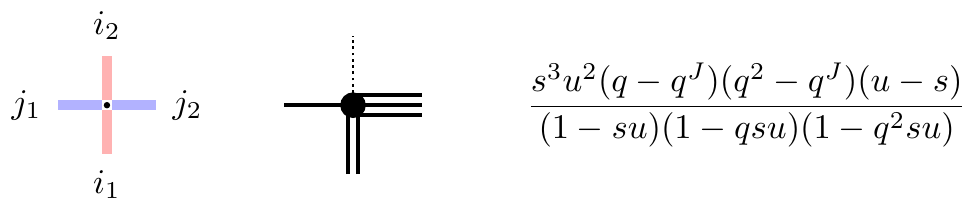}
 \caption{Left: Notation of incoming and outgoing path counts at a vertex.
 Center: The vertex of type $(2,1;0,3)$.
 Right: The weight $L^{(J)}_{u,s}(2,1;0,3)$.}
 \label{fig:LJ_vertex}
\end{figure}

The vertex weights $L^{(J)}_{u,s}$ are called \emph{stochastic} because they satisfy the following
properties:
\begin{proposition}
 \label{prop:LJ_properties}
 If the parameters $q,u,s,J$ satisfy \eqref{eq:parameters_for_LJ_with_restrictions},
 then
 \begin{enumerate}[\bf1.\/]
 \item We have
 $0\le L^{(J)}_{u,s}(i_1,j_1;i_2,j_2)\le 1$
 for all 
 $i_1,i_2 \in \mathbb{Z}_{\ge0}$ and
 $j_1,j_2\in \left\{ 0,1,\ldots,J \right\}$.
 \item For any fixed 
 $i_1\in \mathbb{Z}_{\ge0}$ and $j_1
 \in\left\{ 0,1,\ldots,J \right\}$,
 we have
 \begin{equation}
 \label{eq:LJ_sum_to_one_property}
 \sum_{i_2=0}^{\infty}\sum_{j_2=0}^{J}
 L^{(J)}_{u,s}(i_1,j_1;i_2,j_2)
 =1.
 \end{equation}
 Note that due to the path conservation, this sum is always finite.
 \end{enumerate}
\end{proposition}
\begin{proof}[Idea of proof]
 First, observe that for $J=1$
 the sum in \eqref{eq:r_phi_r_reg_function}
 reduces to at most one term, 
 and the 
 vertex weights $L^{(1)}_{u,s}$ become 
 the following
 explicit rational functions:
 \begin{equation}
 \label{eq:L1_weights_explicit}
 \begin{array}{ll}
 L^{(1)}_{u,s}(g,0;g,0)=\dfrac{1-q^g s u}{1-s u}
 ,&\hspace{30pt}
 L^{(1)}_{u,s}(g,0;g-1,1)=
 \dfrac{-su(1-q^g)}{1-su};
 \\[10pt]
 L^{(1)}_{u,s}(g,1;g,1)=
 \dfrac{-su+q^g s^2}{1-su}
 ,&\hspace{30pt}
 L^{(1)}_{u,s}(g,1;g+1,0)=
 \dfrac{1-q^g s^2}{1-su}
 .
 \end{array}
 \end{equation}
 Then, both statements of the proposition are immediate for $J=1$. 
 The case of arbitrary $J$ follows from the $J=1$ case 
 using the \emph{stochastic fusion procedure}.
 This involves stacking
 $J$ vertices with weights
 $L^{(1)}_{u,s},
 L^{(1)}_{qu,s},
 \ldots, L^{(1)}_{q^{J-1}u,s}$,
 on top of each other, and 
 summing over all possible combinations of 
 outgoing paths. That is, the vertex weight $L^{(J)}_{u,s}(i_1,j_1;i_2,j_2)$ can
 be represented as a convex combination of products of the $J=1$ vertex weights with varying spectral parameters. 
 We refer to
 \cite[Theorem 3.15]{CorwinPetrov2015arXiv},
 \cite[Section~5]{BorodinPetrov2016inhom}, or
 \cite[Appendix B]{borodin_wheeler2018coloured}
 for details.
 We also remark that formula \eqref{eq:LJ_definition}
 for the fused weights is essentially due to 
 \cite{Mangazeev2014},
 and the fusion itself dates back to 
 \cite{KulishReshSkl1981yang}.
\end{proof}

Thanks to \Cref{prop:LJ_properties}, 
we can view each vertex weight 
$L^{(J)}_{u,s}(i_1,j_1;i_2,j_2)$
with fixed incoming path counts $(i_1,j_1)$
as a probability distribution on 
all possible combinations of outgoing paths $(i_2,j_2)$.
In \Cref{sub:particle_systems} below, we use this probabilistic
interpretation to 
build stochastic particle systems out of 
the vertex weights 
$L^{(J)}_{u,s}$.

Let us also make two remarks on the fusion 
procedure which was
mentioned in the proof of \Cref{prop:LJ_properties}.
\begin{remark}[Fusion]
 \label{rmk:LJ_rational_in_qJ}
 \begin{enumerate}[\bf1.\/]
 \item 
 The fused weights $L^{(J)}_{u,s}$
 \eqref{eq:LJ_definition}
 are manifestly rational in $q^J$.
 Therefore, 
 $q^J$ may be treated as an independent parameter and, moreover,
 may be specialized 
 to a complex number not necessarily from the set 
 $q^{\mathbb{Z}_{\ge1}}=\left\{ q,q^2,q^3,\ldots \right\}$. 
 This analytic continuation preserves the sum to one property
 \eqref{eq:LJ_sum_to_one_property} for the
 vertex weights
 $L^{(J)}_{u,s}(i_1,j_1;i_2,j_2)$ when summed over
 $i_2,j_2\ge0$
 (the path counts $i_1,j_1,i_2,j_2$ are always assumed to be nonnegative
 integers).
 Note however that the
 nonnegativity of the vertex weights $L^{(J)}_{u,s}$ 
 has to be checked separately after such a
 continuation.
 \item 
 For $J=1$, the weights $L_{u,s}^{(1)}$ may also be viewed as a stochastic fusion of the stochastic six vertex weights
 along the vertical edges.
 The latter arise from $L^{(1)}_{u,s}(i_1,j_1;i_2,j_2)$ 
 by taking the specialization $s=q^{-\frac{1}{2}}$, which forces
 the path counts $i_1,i_2 \in \left\{ 0,1 \right\}$
 (in addition to the constraint $j_1,j_2 \in\left\{ 0,1 \right\}$
 due to $J=1$).
 Note that 
 $s=q^{-\frac{1}{2}}$ falls outside of $(-1,0]$, contradicting the assumption
 \eqref{eq:parameters_for_LJ_with_restrictions},
 but one readily
 checks that the vertex weights $L^{(1)}_{u,\ssp q^{-1/2}}$ are also nonnegative for $q\in[0,1)$ and
 $u\ge q^{-\frac{1}{2}}$.
 \end{enumerate}
\end{remark}

The vertex weights $L^{(J)}_{u,s}$ generalize the
$q$-beta-binomial
distribution $\varphi_{q,\mu,\nu}$ described
in \Cref{sub:phi_distribution},
and reduce to it in two cases. First, 
setting $u=s$, we have
\cite[Proposition 6.7]{Borodin2014vertex}:
\begin{equation}
 \label{eq:qHahn_reduction_of_LJ}
 L^{(J)}_{s,s}(i_1,j_1;i_2,j_2)=
 \mathbf{1}_{i_1+j_1=i_2+j_2}
 \cdot\mathbf{1}_{j_2\le i_1}
 \cdot
 \varphi_{q,q^Js^2,s^2}(j_2\mid i_1).
\end{equation}
In order to make \eqref{eq:qHahn_reduction_of_LJ}
nonnegative, we should treat $\mu=q^Js^2$ as a
parameter independent of~$\nu$ (with $q^J$ not from $q^{\mathbb{Z}_{\ge1}}$),
and require that 
$0\le \mu\le 1$ and $\nu\le \mu$, as in 
the first case in \eqref{eq:phi_cases_nonnegative_1}.
This analytic continuation is necessary since the substitution
$u=s$ falls outside of the parameter range 
\eqref{eq:parameters_for_LJ_with_restrictions}.

Second, 
in the limit 
as $i_1,i_2\to+\infty$, we have
(e.g., see \cite[Appendix A.2]{BufetovMucciconiPetrov2018}):
\begin{equation}
 \label{eq:L_J_infinity_weights}
 L^{(J)}_{u,s}(\infty,j_1;\infty,j_2)=
 \varphi_{q,suq^J,su}(j_2\mid \infty).
\end{equation}
Here $j_2\in \mathbb{Z}_{\ge0}$ is arbitrary, 
\eqref{eq:L_J_infinity_weights}
does not depend on $j_1$, and 
the path conservation property disappears.
The weights 
\eqref{eq:L_J_infinity_weights}
are nonnegative for $J\in \mathbb{Z}_{\ge1}$; see
\eqref{eq:phi_cases_nonnegative_2}.
Another choice to make 
\eqref{eq:L_J_infinity_weights} nonnegative is to
take an analytic continuation with $su\ge0$ and $\mu=q^Jsu$
treated as an independent parameter with $q^J\notin q^{\mathbb{Z}_{\ge1}}$.

\subsection{Particle systems}
\label{sub:particle_systems}

We define two state spaces for two versions of our Markov dynamics.
\begin{definition}
 \label{def:state_spaces}
 The \emph{vertex model state space} is
 \begin{equation}
 \label{eq:G_state_space}
 \mathscr{G}\coloneqq
 \biggl\{ \mathbf{g}=(g_1,g_2,\ldots )\colon g_i\in \mathbb{Z}_{\ge0},\ 
 \sum_{i=1}^{\infty}g_i<\infty \biggr\}.
 \end{equation}
 The last condition means that only finitely many of the $g_i$'s are nonzero.
 
 The \emph{exclusion process state space} is
 \begin{equation}
 \label{eq:X_state_space}
 \mathscr{X}\coloneqq
 \Bigl\{
 \mathbf{x}=
 \left( x_1>x_2>x_3>\ldots \right)\colon x_i\in \mathbb{Z},
 \ \,
 x_n=-n \ \textnormal{for all
 sufficiently large $n$}
 \Bigr\}.
 \end{equation}
 We view $\mathbf{x}$ as a particle configuration in 
 $\mathbb{Z}$, which is empty far to the right and 
 densely packed far to the left. 
 In other words, every $\mathbf{x}\in \mathscr{X}$
 differs from the distinguished \emph{step configuration}
 $\mathbf{x}_{step}\coloneqq\left\{ -1>-2>-3>\ldots \right\}$
 by finitely many particle jumps to the right by one, when a particle may 
 only jump to an unoccupied location.
\end{definition}

\begin{definition}[Gap-particle transformation]
 \label{def:gap_particle_transform}
 Let the (well-known) bijection $\mathscr{X} \rightarrow \mathscr{G}$ with $\mathbf{x} \mapsto \mathbf{g}$ be defined as 
 \begin{equation}
 \label{eq:gap_particle_transform}
 g_i=x_{i}-x_{i+1}-1,\qquad i\ge1; 
 \end{equation}
 see \Cref{fig:stochastic_vertex_model_exclusion_process} 
 for an illustration. Additionally, we use the convention $g_0=x_0=+\infty$, 
 which extends \eqref{eq:gap_particle_transform} to $i=0$.
\end{definition}

We refer to \eqref{eq:gap_particle_transform} as the 
\emph{gap-particle transformation}. Note, in particular, that the distinguished step configuration of particles, $\mathbf{x}_{step}$, corresponds to the empty configuration $\mathbf{g}_{step}\coloneqq (0,0,\ldots )$. In the special case, when the updates are parallel and not sequential, this is the same as the zero range process (ZRP) / ASEP transformation, e.g., see \cite{Povolotsky2013}. 

We are now in a position to describe 
the \emph{fused stochastic higher spin six vertex} (\emph{FS6V\/}) model $\mathbf{g}(t)$
and its \emph{exclusion process} counterpart, $\mathbf{x}(t)$.
Both models were introduced in
\cite{CorwinPetrov2015arXiv} for homogeneous parameters $u_i\equiv u$, $s_i\equiv s$,
and their inhomogeneous versions were considered
in \cite{BorodinPetrov2016inhom}.
Here and throughout the rest of the section,
$t\in \mathbb{Z}_{\ge0}$ stands for discrete time in
$\mathbf{g}(t)$.

The time-homogeneous Markov
process $\{\mathbf{g}(t)\}_{t\in \mathbb{Z}_{\ge0}}$ on $\mathscr{G}$
depends on two sequences of parameters
\begin{equation}
 \label{eq:parameters_for_higher_spin_vertex_model}
 \mathbf{s}=(s_0,s_1,s_2,\ldots ),\quad
 s_i\in (-1,0];
 \qquad 
 \mathbf{u}=(u_0,u_1,u_2,\ldots ),\quad
 u_i\in [0,+\infty),
\end{equation}
as well as on the parameters $q\in [0,1)$ and $J \in \mathbb{Z}_{\geq 1}$, as in \eqref{eq:parameters_for_LJ_with_restrictions}.
For convergence reasons discussed in 
\Cref{lemma:higher_spin_model_is_well_defined} below, we
assume that
\begin{equation}
 \label{eq:uniformly_bounded_propagation_condition}
 \frac{(-s_i)(u_i-s_i)}{1-u_is_i}<1-\varepsilon<1
\end{equation}
for some fixed $\varepsilon>0$ and all $i$ large enough.
For future use, let us write 
$(\mathbf{u},\mathbf{s})\in \mathcal{T}$
if the parameters satisfy the conditions
\eqref{eq:parameters_for_higher_spin_vertex_model}--\eqref{eq:uniformly_bounded_propagation_condition}.

\begin{remark}
 \label{rmk:epsilon_condition}
 One readily sees that if $-1<s_i< 0$ and $u_i>0$, then 
 $\frac{(-s_i)(s_i-u_i)}{1-s_iu_i}< 1$.
 The condition \eqref{eq:uniformly_bounded_propagation_condition} 
 is stronger 
 than \eqref{eq:parameters_for_higher_spin_vertex_model}
 in that it does not allow these ratios
 to get arbitrarily close to $1$ as $i$ goes to infinity.
\end{remark}

Let us describe how to randomly update the FS6V model $\mathbf{g}(t)$ in (discrete) time $t$. Fix time $t\in \mathbb{Z}_{\geq 0}$, set $\mathbf{g}=\mathbf{g}(t)\in \mathscr{G}$, and let $\mathbf{g}'=\mathbf{g}(t+1)\in \mathscr{G}$ be the random update. The update is independent of time $t$ and occurs as follows:
\begin{equation}
 \label{eq:random_update_stoch_vertex_model}
 g_i' = g_i+h_{i-1}-h_i,
 \qquad i=1,2,\ldots ,
\end{equation}
so that $h_i \in \left\{ 0,1,\ldots,J \right\}$, with $i\in \mathbb{Z}_{\ge0}$, are random variables that 
are sampled sequentially
using the stochastic vertex weights $L^{(J)}_{u_i,s_i}$ for $i=0,1,2,\ldots $.
Namely, $h_0$ is sampled from the probability distribution
$L^{(J)}_{u_0,s_0}(\infty,0;\infty,h_0)$
\eqref{eq:L_J_infinity_weights}.
Then sequentially for $i=1,2,\ldots $, given $h_{i-1}$ and $g_i$,
we sample
the pair $(g_i',h_i)$ with $g_i'+h_i=g_i+h_{i-1}$
from the probability distribution
$L^{(J)}_{u_i,s_i}(g_i,h_{i-1};g_i',h_{i})$
\eqref{eq:LJ_definition}. Below, in \Cref{lemma:higher_spin_model_is_well_defined}, we show that eventually the update terminates, making it well-defined.

\begin{lemma}
 \label{lemma:higher_spin_model_is_well_defined}
 We have $h_i=0$ for all $i$ large enough for the update \eqref{eq:random_update_stoch_vertex_model} with probability $1$.
\end{lemma}
\begin{proof}
 We know that $g_i=0$ for all $i$ large enough. 
 Since $L^{(J)}_{u,s}(0,0;0,0)=1$, it suffices to note that
 all the probabilities of the form
 $L^{(J)}_{u_i,s_i}(0,j;0,j)=\frac{s_i^{2j}(u_i/s_i;q)_j}{(s_iu_i;q)_j}$,
 where $j \in \left\{ 0,1,\ldots,J \right\}$,
 are bounded away from $1$ 
 uniformly 
 in $j\ge 1$
 and sufficiently large $i$,
 due to \eqref{eq:uniformly_bounded_propagation_condition}.
 This means that once $i$ becomes large enough so that all further $g_i$'s are zero,
 then with probability $1$ the quantities $h_i$ eventually decrease to zero,
 and the update terminates.
\end{proof}

After the sequential update over $i=0,1,\ldots $ terminates
according to \Cref{lemma:higher_spin_model_is_well_defined},
we have reached the next state $\mathbf{g}'=\mathbf{g}(t+1)$.

\medskip

The trajectory of the 
Markov process $\left\{ \mathbf{g}(t) \right\}_{t\in \mathbb{Z}_{\ge0}}$
may be viewed as a random path ensemble
in the quadrant 
$\mathbb{Z}_{\ge1}^{2}$. Namely, 
the initial condition $\mathbf{g}(0)$ corresponds to the 
paths entering the quadrant from below,
and the quantities $h_0$ sampled at each time moment $t\ge1$
determine the paths entering from the left. 
The configuration $\mathbf{g}(t)\in \mathscr{G}$
describes the paths crossing 
the horizontal line at height $t+\frac{1}{2}$.
The random update from $\mathbf{g}(t)$ to $\mathbf{g}(t+1)$
determines the horizontal path counts at height $t+1$.
See \Cref{fig:stochastic_vertex_model_exclusion_process} for an illustration.

Denote by $T_{\mathbf{u},\mathbf{s}}$ the one-step Markov transition 
operator for the FS6V process $\{\mathbf{g}(t)\}_{t\in \mathbb{Z}_{\ge0}}$ on $\mathscr{G}$.
This operator also depends on the parameters $q$ and $J$, but we suppress this 
in the notation. In the literature on solvable lattice models
(for instance, \cite{baxter2007exactly}),
$T_{\mathbf{u},\mathbf{s}}$ is often referred to as the 
transfer matrix. 
Our transfer matrix is a Markov transition operator 
since the model is stochastic. Similar stochasticity of transfer matrices
was first observed in \cite{GwaSpohn1992}.

\begin{figure}[htb]
 \centering
 \includegraphics[width=.7\textwidth]{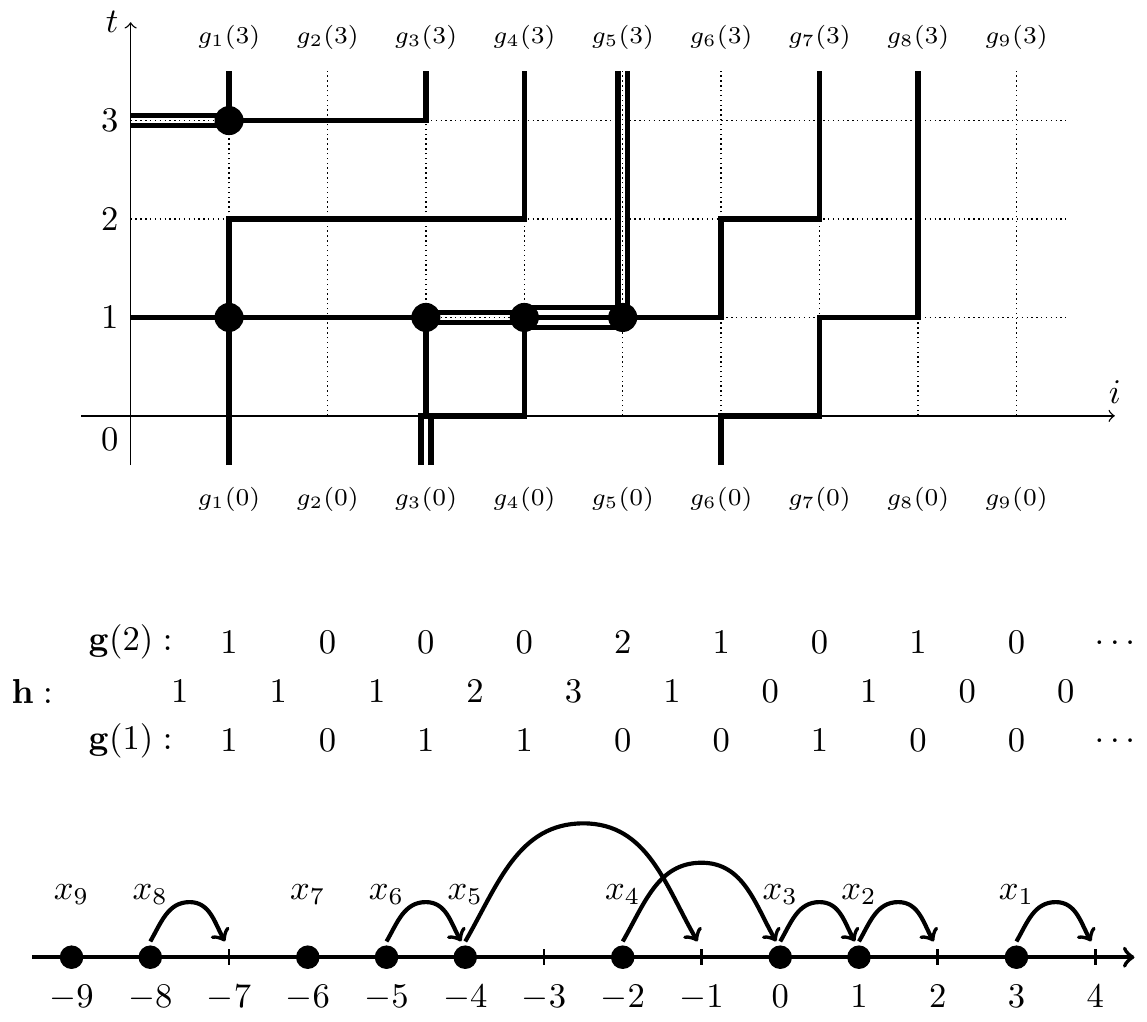}
 \caption{Top: A path ensemble corresponding to the time 
 evolution of $\mathbf{g}(t)$ for $t\in \left\{ 0,1,2,3 \right\}$.
 Middle: Update $\mathbf{g}(1)\to\mathbf{g}(2)$
 and the corresponding quantities $h_i$, $i\ge0$. 
 Bottom: 
 the corresponding exclusion process 
 update $\mathbf{x}(1)\to \mathbf{x}(2)$,
 where the particle configuration $\mathbf{x}(1)$ is shown,
 and the arrows represent sequential particle jumps.}
 \label{fig:stochastic_vertex_model_exclusion_process}
\end{figure}

\medskip

Finally, let us describe the Markov process 
$\{\mathbf{x}(t)\}_{t\in \mathbb{Z}_{\ge0}}$
on the space $\mathscr{X}$ induced by the FS6V process $\left\{ \mathbf{g}(t) \right\}$ via the gap-particle transformation \eqref{eq:gap_particle_transform}.
The random update from 
$\mathbf{x}(t)$
to 
$\mathbf{x}(t+1)$
is described as follows. First, 
the rightmost particle $x_1$ jumps to the right by a random distance
$h_0$ sampled from
$L^{(J)}_{u_0,s_0}(\infty,0;\infty,h_0)$
\eqref{eq:L_J_infinity_weights}.
Then, sequentially for $i=1,2,\ldots $, the particle $x_{i+1}$ jumps to the right
by a random distance $h_i$ sampled from
\begin{equation*}
 L^{(J)}_{u_i,s_i}(x_i(t)-x_{i+1}(t)-1,h_{i-1};x_{i}(t+1)-x_{i+1}(t+1)-1,h_i),
\end{equation*}
given that $x_{i}$ has jumped by the distance $h_{i-1}$. Note that the upper bound for the distance of each particle's jump is equal to the parameter $J$.

Due to the path conservation, we see that the dynamics of $\mathbf{x}(t)$
satisfies the exclusion rule,
that is, a particle may only jump to an unoccupied location. Consequently, the strict order of the particles $x_1>x_2>\ldots $
is preserved throughout the dynamics. 
We also note that the 
jump of each particle $x_i$ at time $t$ is governed by the 
parameters $(u_{i-1},s_{i-1})$ as well as the 
locations $x_i(t)$, $x_{i-1}(t)$, and $x_{i-1}(t+1)$.

We denote the one-step Markov transition operator
of the process $\{\mathbf{x}(t)\}_{t\in \mathbb{Z}_{\ge0}}$
on $\mathscr{X}$ by $\tilde{T}_{\mathbf{u},\mathbf{s}}$.
It is the image of the FS6V model 
operator $T_{\mathbf{u},\mathbf{s}}$ on $\mathscr{G}$
under the gap-particle transformation \eqref{eq:gap_particle_transform}.
Throughout the paper we adopt the same convention for all
Markov transition operators: 
$A$ on $\mathscr{G}$ corresponds to $\tilde{A}$ on $\mathscr{X}$.

\section{Yang-Baxter equation and cross vertex weights}
\label{sec:YBE}

The vertex weights $L^{(J)}_{u,s}$ \eqref{eq:LJ_definition} satisfy the Yang-Baxter equation, and this makes the stochastic processes from \Cref{sec:stoch_vert_models} very special, i.e.,~integrable. In short, the Yang-Baxter equation determines the (local) action of swapping two consecutive vertex weights in the FS6V model. This swapping action is represented by introducing a cross-vertex that is dragged across the vertex weights; see \Cref{fig:YBE}.

\subsection{Vertical Yang-Baxter equation}
\label{sub:YBE}

There are several possible Yang-Baxter equations that the vertex weights $L^{(J)}_{u,s}$'s may satisfy. For our purposes, we focus on the Yang-Baxter equation that may be represented by vertically dragging a cross vertex through two consecutive horizontal vertex weights.

Let $s_1,s_2\in (-1,0]$ and $z\ge 0$ be three parameters.
Define the cross vertex weights as follows:
\begin{equation}
 \label{eq:vertical_R_matrix_for_LJ}
 R_{z,s_1,s_2}(i_1,i_2;j_1,j_2)
 \coloneqq
 L^{(I_1)}_{s_1z,s_2}(j_1,i_2;i_1,j_2),
\end{equation}
with the right side given by \eqref{eq:LJ_definition} so that 
$I_1$ is determined by the identity$s_1=q^{-I_1/2}$. Here,
we treat $q^{-I_1/2}$ as an independent parameter
which enters $R_{z,s_1,s_2}$ in a rational manner, according to \Cref{rmk:LJ_rational_in_qJ}. 
Explicitly, we have
\begin{equation}
 \label{eq:vertical_R_matrix_for_LJ_explicit}
 \begin{split}
 &
 R_{z,s_1,s_2}(i_1,i_2;j_1,j_2)=
 \mathbf{1}_{j_1+i_2=i_1+j_2}
 \frac{(-s_1z)^{j_1}q^{
 \frac12 j_1(j_1+2i_2-1)}
 s_2^{i_2+j_2-i_1}(zs_1 s_2^{-1};q)_{j_2- j_1}}
 {(q;q)_{i_1} (zs_1s_2 ;q)_{i_1+j_2}
 (s_1^{-2}q^{1-i_2};q)_{i_2-j_2}}
 \\&\hspace{200pt}
 \times{}_{4}\bar{\phi}_3\left(\begin{matrix} q^{-i_1};q^{-j_1},zs_1^{-1}s_2
 ,q z^{-1}s_1^{-1}s_2\\
 s_2^{2},q^{1+j_2-j_1},s_1^{-2}q^{1-i_1-j_2}\end{matrix}
 \bigg|\, q,q\right),
 \end{split}
\end{equation}
where $i_1,i_2,j_1,j_2$ are arbitrary nonnegative integers.
The path conservation property in
\eqref{eq:vertical_R_matrix_for_LJ}
means that 
$R_{z,s_1,s_2}(i_1,i_2;j_1,j_2)$
vanishes unless $i_1+j_2=i_2+j_1$.
See \Cref{fig:R_matrix} for an illustration.

\begin{figure}[htb]
 \centering
 \includegraphics[width=.15\textwidth]{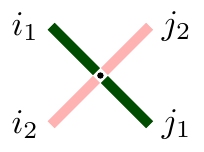}
 \caption{The weights $R_{z,s_1,s_2}$ 
 \eqref{eq:vertical_R_matrix_for_LJ} are attached to the cross vertices, i.e., vertices drawn on the lattice rotated by 
 $45^\circ$. The path counts $i_1,i_2,j_1,j_2$ are as shown in the figure.}
 \label{fig:R_matrix}
\end{figure}

The vertex weights $L^{(J)}_{u,s}$ and the cross vertex weights $R_{z,s_1,s_2}$ satisfy the following Yang-Baxter equation:

\begin{proposition}[Yang-Baxter equation]
 \label{prop:YBE}
 Fix the path counts $i_1,j_1 \in\left\{ 0,1,\ldots,J \right\}$,
 $i_2,i_3,j_2,j_3\in \mathbb{Z}_{\ge0}$ 
 and the parameters $u_1,u_2,s_1,s_2$. Then, we have
 \begin{equation}
 \label{eq:YBE_for_LJ}
 \begin{split}
 &
 \sum_{k_1,k_2,k_3}
 R_{\frac{u_2}{u_1},s_1,s_2}(j_3,k_2;k_3,j_2)
 \ssp
 L^{(J)}_{u_1,s_1}(i_2,i_1;k_2,k_1)
 \ssp
 L^{(J)}_{u_2,s_2}(i_3,k_1;k_3,j_1)
 \\
 &\hspace{60pt}
 =
 \sum_{k_1,k_2,k_3}
 L^{(J)}_{u_2,s_2}(k_3,i_1;j_3,k_1)\ssp
 L^{(J)}_{u_1,s_1}(k_2,k_1;j_2,j_1)\ssp
 R_{\frac{u_2}{u_1},s_1,s_2}(k_3,i_2;i_3,k_2).
 \end{split}
 \end{equation}
 See \Cref{fig:YBE} for an illustration.
 The sums in \eqref{eq:YBE_for_LJ} are over $k_1 \in\left\{ 0,1,\ldots,J \right\}$
 and $k_2,k_3\in \mathbb{Z}_{\ge0}$.
 However, both sum are finite due to the path conservation properties, making the Yang-Baxter equation \eqref{eq:YBE_for_LJ} an identity between rational 
 functions.
\end{proposition}
Observe that the cross vertex weights $R_{\frac{u_2}{u_1},s_1,s_2}$ 
entering the Yang-Baxter equation
\eqref{eq:YBE_for_LJ}
do not depend on the parameter $J$ and only depend on the parameters
$u_1,u_2$ through their ratio.
\begin{proof}[Proof of \Cref{prop:YBE}]
 The Yang-Baxter equation $\eqref{eq:YBE_for_LJ}$ 
 follows by fusion from the simpler case when the paramters are $s_1=s_2=q^{-1/2}$ and $I_1=1$. 
 This simpler case of the Yang-Baxter equation may be checked through direct computations. Moreover, the latter equation 
 essentially coincides with
 \cite[Proposition A.1]{BufetovMucciconiPetrov2018},
 up to 
 the specialization of their parameter $s$ into
 $q^{-J/2}$
 and
 a gauge transformation
 making the vertex weights $w_{u,s}$ stochastic. Note that the
 cross vertex weights $r_{u/v}$ in 
 \cite[Proposition A.1]{BufetovMucciconiPetrov2018}
 are already stochastic (i.e.~satisfy the sum to one property).
 Then, our fused
 Yang-Baxter equation
 \eqref{eq:YBE_for_LJ} 
 follows from 
 \cite[Proposition~A.3]{BufetovMucciconiPetrov2018}
 (which is essentially a fusion of 
 \cite[Proposition A.1]{BufetovMucciconiPetrov2018}),
 up to a gauge transformation and 
 path complementation $i\mapsto I-i$.

 Alternatively, the fused Yang-Baxter 
 equation with stochastic vertex weights
 implying \eqref{eq:YBE_for_LJ}
 is a color-blind case of the 
 master Yang-Baxter equation coming from
 $U_q(\widehat{\mathfrak{sl}_n})$
 obtained in
 \cite{bosnjak2016construction}; see
 \cite[(C.1.2)]{borodin_wheeler2018coloured}.
\end{proof}

\begin{figure}[htb]
 \centering
 \includegraphics[width=.6\textwidth]{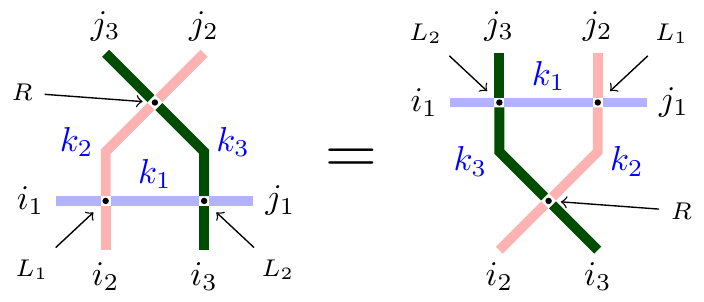}
 \caption{The Yang-Baxter equation
 \eqref{eq:YBE_for_LJ} is the equality of partition functions of the 
 two configurations in the figure, where the boundary 
 path counts $i_1,i_2,i_3,j_1,j_2,j_3$ are fixed, and the summation is 
 over all possible configurations of paths occupying the internal
 edges. These internal path counts are $k_1,k_2,k_3$.
 The vertex weights are $L_i=L^{(J)}_{u_i,s_i}$, and 
 $R=R_{\frac{u_2}{u_1},s_1,s_2}$.}
 \label{fig:YBE}
\end{figure}

\subsection{Nonnegativity}
\label{sub:nonnegativity}

From \Cref{prop:LJ_properties},
we see that the cross vertex weights $R_{z,s_1,s_2}$,
defined by
\eqref{eq:vertical_R_matrix_for_LJ},
satisfy the sum to one property
\begin{equation}
 \label{eq:R_sum_to_one}
 \sum_{i_1,j_2=0}^{\infty}
 R_{z,s_1,s_2}(i_1,i_2;j_1,j_2)=1
\end{equation}
for any fixed $i_2,j_1\in \mathbb{Z}_{\ge0}$.
Moreover, if the vertex weights $R_{z,s_1,s_2}$ are nonnegative, the vertex weights define a probability distribution. This distribution is on the top paths $(i_1, j_2)$ of a cross vertex for any fixed bottom paths $(i_2,j_1)$, see \Cref{fig:R_matrix}. Below, we show that the cross vertex weights are nonnegative under a suitable restriction of the parameters.

\begin{proposition}
 \label{prop:R_nonnegative}
 If 
 $q\in [0,1)$,
 $s_1,s_2\in (-1,0)$,
 and 
 $0\le z\le \min\bigl\{ \frac{s_1}{s_2}, \frac{s_2}{s_1},\frac{q}{s_1s_2} \bigr\}$,
 then 
 \begin{equation*}
 R_{z,s_1,s_2}(i_1,i_2;j_1,j_2)\ge0
 \end{equation*}
 for all 
 $i_1,j_1,i_2,j_2\in \mathbb{Z}_{\ge0}$.
\end{proposition}
\begin{proof}
 First, assume that $i_1\le i_2$,
 which is equivalent to $j_1\le j_2$.
 Rewrite
 \eqref{eq:vertical_R_matrix_for_LJ_explicit}
 via the ordinary 
 $q$-hypergeometric series
 $_4\phi_3$ \eqref{eq:r_phi_r_reg_function}:
 \begin{equation}
 \label{eq:R_positivity_proof_1}
 \begin{split}
 &
 R_{z,s_1,s_2}(i_1,i_2;j_1,j_2)=
 \mathbf{1}_{j_1+i_2=i_1+j_2}
 \frac{(-s_1z)^{j_1}q^{
 \frac12 j_1(j_1+2i_2-1)}
 s_2^{i_2+j_2-i_1}(zs_1 s_2^{-1};q)_{j_2-j_1}}
 {(q;q)_{i_1} (zs_1s_2 ;q)_{i_1+j_2}
 (s_1^{-2}q^{1-i_2};q)_{i_2-j_2}}
 \\&\hspace{20pt}
 \times
 (s_2^2;q)_{i_1}
 (q^{1+j_2-j_1};q)_{i_1}
 (s_1^{-2}q^{1-i_1-j_2};q)_{i_1}
 \cdot
 {}_{4}\phi_3\left(\begin{matrix} q^{-i_1},q^{-j_1},zs_1^{-1}s_2
 ,q z^{-1}s_1^{-1}s_2\\
 s_2^{2},q^{1+j_2-j_1},s_1^{-2}q^{1-i_1-j_2}\end{matrix}
 \bigg|\, q,q\right).
 \end{split}
 \end{equation}
 Note that $s_2^{i_2+j_2-i_1}
 =
 (-1)^{j_1}
 s_2^{2j_2}(-s_2)^{-j_1}$.
 Then, the following part of the prefactor
 \begin{equation*}
 \frac{
 (-s_1z)^{j_1} s_2^{2j_2}(-s_2)^{-j_1}
 q^{\frac{1}{2}j_1(j_1+2i_2-1)}
 (s_2^2;q)_{i_1}
 (q^{1+j_2-j_1};q)_{i_1}
 (zs_1 s_2^{-1};q)_{j_2-j_1}
 }
 {
 (q;q)_{i_1}
 (zs_1s_2;q)_{i_1+j_2}
 }
 \end{equation*}
 is clearly nonnegative
 under our conditions. Additionally, 
 observe that 
 \begin{equation*}
 (-1)^{j_1}\frac{(s_1^{-2}q^{1-i_1-j_2};q)_{i_1}}
 {(s_1^{-2}q^{1-i_2};q)_{i_2-j_2}}
 =(-1)^{j_1}\,(s_1^{-2}q^{1-j_1-i_2};q)_{j_1}\ge0,
 \end{equation*}
 as all factors in the above $q$-Pochhammer symbol 
 are nonpositive, and
 there are $j_1$ of them.
 
 Thus, it remains to establish
 the nonnegativity of the 
 $q$-hypergeometric 
 series $_4\phi_3$ in \eqref{eq:R_positivity_proof_1}.
 We use the nonnegativity result in the proof of
 \cite[Proposition A.8]{BufetovMucciconiPetrov2018}
 which is based on
 Watson's 
 transformation formula
 \cite[(III.19)]{GasperRahman}.
 That proof essentially
 established the nonnegativity of 
 \begin{equation}
 \label{eq:R_positivity_proof_2}
 {}_{4}\phi_3\left(\begin{matrix} 
 q^{-\mathsf{i}_2},q^{-\mathsf{i}_1},- q / (\mathsf{s}\xi), -\mathsf{s}\theta
 \\
 -\mathsf{s}/\xi,
 q^{1+\mathsf{j}_2-\mathsf{i}_2},-\theta q^{1-\mathsf{i}_1-\mathsf{j}_2}
 /\mathsf{s}\end{matrix}
 \bigg|\, q,q\right),
 \end{equation}
 where $\mathsf{i}_2+\mathsf{j}_1=\mathsf{i}_1+\mathsf{j}_2$, $\mathsf{i}_2\le 
 \mathsf{j}_2$ (so $\mathsf{i}_1\le\mathsf{j}_1$), and the parameters satisfy 
 \begin{equation}
 \label{eq:R_positivity_proof_conditions_old}
 q\in(0,1),\qquad 
 \mathsf{s}\in(-\sqrt{q},0), \qquad 
 \xi,\theta\in[-\mathsf{s},-\mathsf{s}^{-1}].
 \end{equation}
 Indeed, 
 one can 
 check that the prefactor 
 \begin{equation}
 \label{eq:R_positivity_proof_3}
 \frac{\mathsf{s}^{\mathsf{j}_2}
 (-\theta q^{1-\mathsf{i}_1-\mathsf{j}_2}/\mathsf{s};q)_{\mathsf{i}_2}}
 {
 (-q / (\mathsf{s}\xi);q)_{\mathsf{i}_1-\mathsf{j}_1}
 }
 \end{equation}
 in front of $_4\phi_3$ in
 \cite[(A.24)]{BufetovMucciconiPetrov2018} is already nonnegative.
 Namely, for $\mathsf{i}_1=\mathsf{j}_2=0$ this prefactor is $1$.
 For all other values of $\mathsf{i}_1,\mathsf{j}_2$ we see
 that $\mathsf{s}<0$, $-\theta q^{1-\mathsf{i}_1-\mathsf{j}_2+l}/\mathsf{s}\ge 1$
 for $0\le l\le \mathsf{i}_2-1$. Using \eqref{eq:qPochhammer_negative} we have
 $(-q / (\mathsf{s}\xi);q)_{\mathsf{i}_1-\mathsf{j}_1}^{-1}
 =(-1 / (\mathsf{s}\xi);1/q)_{\mathsf{j}_1-\mathsf{i}_1}$,
 and note that $-q^{-l}/(\mathsf{s}\xi)\ge1$ for $l\ge 0$.
 Thus, \eqref{eq:R_positivity_proof_3}
 is a product of $\mathsf{j}_2+\mathsf{i}_2+\mathsf{j_1}-\mathsf{i_1}=2\mathsf{j}_2$
 nonpositive factors, and hence is nonnegative.

 Now, \eqref{eq:R_positivity_proof_2} matches 
 the $_4\phi_3$ function in
 \eqref{eq:R_positivity_proof_1}
 when
 $\mathsf{i}_1=i_1$, $\mathsf{i}_2=j_1$,
 $\mathsf{j}_1=i_2$, $\mathsf{j}_2=j_2$,
 and
 \begin{equation*}
 \mathsf{s}=-\sqrt{z s_1s_2},
 \qquad 
 \xi=\sqrt{zs_1s_2^{-3}},\qquad 
 \theta=\sqrt{zs_1^{-3}s_2}.
 \end{equation*}
 Rewriting conditions \eqref{eq:R_positivity_proof_conditions_old}
 on $\mathsf{s},\xi,\theta$ in terms of $z,s_1,s_2$, 
 we arrive at the desired result for $i_1\le i_2$.

 For the remaining case $i_1>i_2$, one can check that $R_{z,s_1,s_2}$
 satisfies 
 \begin{equation}
 \label{eq:R_matrix_symmetry}
 R_{z,s_1,s_2}(i_1,i_2;j_1,j_2)=z^{j_2-i_2}\frac{s_2^{i_2+j_2}}{s_1^{i_1+j_1}}
 \,R_{z,s_2,s_1}(j_2,j_1;i_2,i_1).
 \end{equation}
 Thus, the case $i_1>i_2$ reduces to case $i_1 \leq i_2$ since the prefactor is nonnegative. This establishes the result for all cases.
\end{proof}

\subsection{Specialization at $q=0$}
\label{sub:spec_q0_of_R}

The expression for the cross vertex weights $R_{z,s_1,s_2}$ simplify considerably when $q=0$. Let us denote the specialization of the vertex weight at $q=0$ as follows
\begin{equation}\label{eq:R_at_q0_specialization}
 R^{(0)}_{z,s_1,s_2}(i_1,i_2;j_1,j_2) = R_{z,s_1,s_2}(i_1,i_2;j_1,j_2)\big\vert_{q=0}.
\end{equation}
Additionally, introduce the notation
\begin{equation}
 \label{eq:R_at_q0_notation}
 \begin{split}
 &
 \widehat R^{(0)}_{z,s_1,s_2}(i_1,i_2;j_1,j_2)
 \coloneqq
 \mathbf{1}_{i_1+j_2=i_2+j_1}
 \ssp
 z^{j_1}(s_1s_2)^{-j_1}
 s_2^{2j_2}
 \\&\hspace{120pt}
 \times
 \biggl( 
 \frac{(1-zs_1 s_2^{-1}\mathbf{1}_{j_2>j_1})(1-s_2^2\mathbf{1}_{i_1>0})
 (1-s_1^2\mathbf{1}_{i_2=0}\mathbf{1}_{j_1>0})}
 {1-zs_1s_2\mathbf{1}_{i_1+j_2>0}}
 \\
 &\hspace{255pt}
 -
 \mathbf{1}_{i_1=i_2>0}\mathbf{1}_{j_1>0} 
 \ssp
 \frac{(-s_1)(s_2z-s_1)}{1-zs_1s_2}
 \biggr).
 \end{split}
\end{equation}
We express the specialization at $q=0$ in terms of $\widehat R^{(0)}_{z,s_1,s_2}$:

\begin{proposition}
 \label{prop:R_at_q0}
 We have
 \begin{equation}
 \label{eq:R_at_q0}
 R^{(0)}_{z,s_1,s_2}(i_1,i_2;j_1,j_2)
 =
 \begin{cases}
 \widehat R^{(0)}_{z,s_1,s_2}(i_1,i_2;j_1,j_2),&\textnormal{if $j_1\le j_2$};\\
 z^{j_2-i_2}s_2^{i_2+j_2}s_1^{-i_1-j_1}\widehat R^{(0)}_{z,s_2,s_1}
 (j_2,j_1;i_2,i_1),&\textnormal{if $j_1\ge j_2$},
 \end{cases}
 \end{equation}
 with the weights on the left and right side of the equation given by \eqref{eq:R_at_q0_specialization} and \eqref{eq:R_at_q0_notation}.
\end{proposition}

\begin{proof}[Proof of \Cref{prop:R_at_q0}]
 Throughout the proof we assume that
 $j_1+i_2=i_1+j_2$ due to the path conservation property.
 We have
 \begin{equation}
 \label{eq:R_at_q0_proof1}
 \begin{split}
 &
 R_{z,s_1,s_2}(i_1,i_2;j_1,j_2)
 =
 \sum_{k=0}^{\min(i_1,j_1)}
 \frac{(-s_1z)^{j_1}q^{
 \frac12 j_1(j_1+2i_2-1)}
 s_2^{i_2+j_2-i_1}(zs_1 s_2^{-1};q)_{j_2-j_1}}
 {(q;q)_{i_1} (zs_1s_2 ;q)_{i_1+j_2}
 (s_1^{-2}q^{1-i_2};q)_{i_2-j_2}}
 \frac{q^k}{(q;q)_k}
 (q^{-i_1};q)_k
 \\&
 \hspace{20pt}\times
 (q^{-j_1};q)_k
 \bigl(\tfrac{zs_2}{s_1};q\bigr)_k
 \bigl(\tfrac{qs_2}{zs_1};q\bigr)_k
 (s_2^2q^k;q)_{i_1-k}
 (q^{1+j_2-j_1+k};q)_{i_1-k}
 (s_1^{-2}q^{1-i_1-j_2+k};q)_{i_1-k}.
 \end{split}
 \end{equation}
 Setting $q=0$ eliminates all the terms in the 
 sum containing a positive power of $q$.
 There are no 
 negative powers of $q$, which follows from the fusion construction
 of $R_{z,s_1,s_2}$, see the proof of \Cref{prop:YBE}.
 A positive power of $q$ may only come from the 
 following terms:
 \begin{equation*}
 \frac{
 q^{
 \frac12 j_1(j_1+2i_2-1)+k}
 (zs_1 s_2^{-1};q)_{j_2-j_1}}
 {(s_1^{-2}q^{1-i_2};q)_{i_2-j_2}}
 (q^{-i_1};q)_k
 (q^{-j_1};q)_k
 (s_1^{-2}q^{1-i_1-j_2+k};q)_{i_1-k}.
 \end{equation*}
 For example, $(zs_1 s_2^{-1};q)_{j_2-j_1}$ yields 
 $q^{\frac{1}{2}(j_1-j_2)(j_1-j_2+1)}$
 for $j_2<j_1$, see \eqref{eq:qPochhammer_negative}.
 Overall, one can check that the total power of $q$
 is equal to
 \begin{equation}
 \label{eq:R_at_q0_proof2}
 \mathbf{1}_{j_2<j_1}\tbinom{j_1-j_2+1}{2}
 +
 \tfrac{1}{2} k(k-1) +
 k(i_2-i_1)
 \end{equation}
 If $j_1\le j_2$, then \eqref{eq:R_at_q0_proof2}
 equal to zero in the following cases:
 \begin{enumerate}[$\bullet$]
 \item either $k=0$;
 \item or $i_1=i_2$ and $k=1$. 
 \end{enumerate}
 We obtain \eqref{eq:R_at_q0_notation} by the contribution of the two cases above.
 Setting $q=0$ and $k=0$ in the sum in
 \eqref{eq:R_at_q0_proof1},
 we get the first summand in \eqref{eq:R_at_q0_notation}.
 For $i_1=i_2$, we get the additional
 second summand in \eqref{eq:R_at_q0_notation}
 coming from the term with $k=1$.
 This proves \eqref{eq:R_at_q0}
 for $j_1\le j_2$. We use the symmetry \eqref{eq:R_matrix_symmetry} to obtain the result when $j_1\ge j_2$
 This completes the proof.
\end{proof}

We extend the nonnegativity result of \Cref{prop:R_nonnegative} for the specialization at $q=0$. 
Note that \Cref{prop:R_nonnegative} restricts $z$ to $0$ for $q=0$. Due to this, we need to independently
find a range of parameters $(z,s_1,s_2)$ for which 
the weights $R^{(0)}_{z,s_1,s_2}$ are nonnegative:
\begin{proposition}
 \label{prop:R_at_q0_nonnegative}
 If $q=0$, $s_1,s_2\in (-1,0)$, 
 $0\le z\le \min\{\frac{s_1}{s_2},\frac{s_2}{s_1}\}$,
 and $s_1^2+s_2^2 \le 1+zs_1s_2$, then
 \begin{equation*}
 R^{(0)}_{z,s_1,s_2}(i_1,i_2;j_1,j_2)\ge0
 \end{equation*}
 for all 
 $i_1,j_1,i_2,j_2\in \mathbb{Z}_{\ge0}$.
\end{proposition}
\begin{proof}
 For $i_1\ne i_2$, one can check that 
 conditions 
 $s_1,s_2\in (-1,0)$ and
 $0\le z\le \min\{\frac{s_1}{s_2},\frac{s_2}{s_1}\}$
 are enough for nonnegativity since only the first summand in 
 \eqref{eq:R_at_q0_notation}
 is present. When $i_1=i_2$, 
 in the 
 second summand in 
 \eqref{eq:R_at_q0_notation}
 we have $s_2z-s_1\ge0$, so the second summand is negative. 
 Combining it with the first summand leads to 
 the additional condition
 $s_1^2+s_2^2 \le 1+zs_1s_2$.
\end{proof}

\subsection{Specialization to $q$-beta-binomial cross vertex weights}
\label{sub:R_Hahn_specialization}

For $q>0$, the 
cross vertex weights 
$R_{\frac{u_2}{u_1},s_1,s_2}$ 
\eqref{eq:vertical_R_matrix_for_LJ_explicit} have a complicated 
$q$-hypergeometric expression, even when $J=1$ and the lattice vertex weights 
$L^{(J)}_{u_i,s_i}$ 
entering the Yang-Baxter equation
\eqref{eq:YBE_for_LJ}
are explicit rational functions
(see \eqref{eq:L1_weights_explicit}).
There are several ways to specialize the 
cross vertex
parameters 
to simplify
$R_{\frac{u_2}{u_1},s_1,s_2}$.
For instance, one may take finite spin specializations by setting, in the simplest case,
$s_1=s_2=q^{-\frac{1}{2}}$. However, this specialization
would
bound the gap sizes in the higher spin
exclusion process $\mathbf{x}(t)$
defined in \Cref{sub:particle_systems} and, as a result, 
we will not consider this specialization here.
Instead, we distinguish a specialization 
reducing
$R_{\frac{u_2}{u_1},s_1,s_2}$ 
to the $q$-beta-binomial distribution:
\begin{proposition}
 \label{prop:qHahn_degeneration_of_R}
 If $u_2/u_1=s_2/s_1$,
 then the cross vertex weights 
 in the Yang-Baxter equation in 
 \Cref{prop:YBE} 
 are simplified as
 \begin{equation}
 \label{eq:R_matrix_qHahn_reduction}
 R_{\frac{s_2}{s_1},s_1,s_2}
 (i_1,i_2;j_1,j_2)
 =
 \mathbf{1}_{i_1+j_2=i_2+j_1}
 \cdot
 \mathbf{1}_{j_2\le j_1}
 \cdot
 \varphi_{q,s_2^2/s_1^2,s_2^2}(j_2\mid j_1),
 \end{equation}
 where $\varphi_{q,s_2^2/s_1^2,s_2^2}$
 is the 
 $q$-beta-binomial distribution
 \eqref{eq:phi_finite}. 
\end{proposition}
\begin{proof}
 This is a combination of
 \eqref{eq:vertical_R_matrix_for_LJ}
 and the reduction of the weights $L^{(J)}_{s,s}$
 \eqref{eq:qHahn_reduction_of_LJ}. Here we set
 $q^{I_1}=s_1^{-2}$.
\end{proof}

Recall \eqref{eq:phi_cases_nonnegative_1}
and note that 
$R_{\frac{s_2}{s_1},s_1,s_2}
(i_1,i_2;j_1,j_2)\ge0$ for all $i_1,i_2,j_1,j_2\in \mathbb{Z}_{\ge0}$
if
\begin{equation}
 \label{eq:R_matrix_qHahn_reduction_nonnegative}
 s_1,s_2 \in (-1,0],\qquad |s_2|\le |s_1|.
\end{equation}
We see that the specialization $z=s_2/s_1$
extends the range of nonnegativity of the cross vertex weights compared to the conditions of 
\Cref{prop:R_nonnegative}.
In particular, the condition $s_2^2\le q$ is dropped.

\begin{definition}
 \label{def:parameter_range_R}
 Let $\mathcal{R}$ be the range of parameters so that $(z,s_1,s_2)\in \mathcal{R}$ if
 \begin{enumerate}[$\bullet$]
 \item 
 either $q\in (0,1)$,
 $s_1,s_2\in (-1,0)$,
 and 
 $0\le z\le \min\bigl\{ \frac{s_1}{s_2}, \frac{s_2}{s_1},\frac{q}{s_1s_2} \bigr\}$
 as in \Cref{prop:R_nonnegative};
 \item 
 or $q=0$,
 $s_1,s_2\in (-1,0)$, 
 $0\le z\le \min\{\frac{s_1}{s_2},\frac{s_2}{s_1}\}$,
 and $s_1^2+s_2^2 \le 1+zs_1s_2$
 as in \Cref{prop:R_at_q0_nonnegative};
 \item 
 or 
 $z=s_2/s_1$ and $s_1,s_2\in (-1,0]$ with
 $|s_2|\le |s_1|$ as in \eqref{eq:R_matrix_qHahn_reduction_nonnegative}. 
 \end{enumerate}
 Note that the cross vertex weights $R_{z,s_1,s_2}(i_1,i_2;j_1,j_2)$ are nonnegative
 for all $(z,s_1,s_2)\in \mathcal{R}$.
\end{definition}

Let us describe the probabilistic interpretation of the specialization
$R_{\frac{s_2}{s_1},s_1,s_2}(i_1,i_2;j_1,j_2)$
viewed as the distribution of the top paths $(i_1,j_2)$
given the bottom paths $(i_2,j_1)$, see \Cref{fig:R_matrix}
for an illustration.
From \eqref{eq:R_matrix_qHahn_reduction}
we see that 
$j_1$ paths coming from southeast
randomly split into $j_2$ and $j_1-j_2$
according to the $q$-beta-binomial distribution 
$\varphi_{q,s_2^2/s_1^2,s_2^2}(j_2\mid j_1)$. Then
$j_2$ paths continue in the northeast direction,
while $j_1-j_2$ paths turn in the northwest direction. All the southwest $i_2$ paths simply continue in the northwest direction, so that
$i_1=i_2+j_1-j_2$.

\subsection{Limit to infinitely many paths}
\label{sub:R_infinite_limit}

We will also need a limit of the cross vertex weights 
$R_{z,s_1,s_2}(i_1,i_2;j_1,j_2)$ as the 
number of southwest incoming paths $i_2$ grows to infinity:
\begin{proposition}
 \label{prop:R_infinity_limit}
 Let $j_1,j_2\in \mathbb{Z}_{\ge0}$. We have
 $\lim\limits_{L\to+\infty}R_{z,s_1,s_2}(L,L+j_2-j_1;j_1,j_2)
 =
 R^{\mathrm{bdry}}_{z,s_1,s_2}(j_1,j_2)$,
 where
 $R^{\mathrm{bdry}}_{z,s_1,s_2}(j_1,j_2)$ is, by definition, equal to
 \begin{equation}
 \label{eq:R_circ_as_the_infinity_limit_def}
 \frac{(-1)^{j_1}q^{\frac{1}{2}j_1(j_1-1)}s_2^{2j_2}(s_1z/s_2;q)_{j_2-j_1}}
 {(q;q)_{j_2}(s_2^2;q)_{j_1}}\ssp
 \frac{(s_2^2;q)_{\infty}}{(zs_1s_2;q)_{\infty}}
 \,{}_{3}\bar{\phi}_2
 \left(
 \begin{matrix} 
 q^{-j_1};zs_2/s_1, q^{j_2-j_1}zs_1/s_2\\
 q^{j_2-j_1+1},q^{1-j_1}z/(s_1s_2)
 \end{matrix}
 \bigg|\, q,q
 \right).
 \end{equation}
 Here ${}_{3}\bar{\phi}_2$ is the regularized 
 (terminating)
 $q$-hypergeometric series 
 \eqref{eq:r_phi_r_reg_function}.
 Moreover, for any fixed $j_1\in \mathbb{Z}_{\ge0}$ we have
 \begin{equation}
 \label{eq:R_circ_sum_to_one}
 \sum_{j_2=0}^{\infty} R^{\mathrm{bdry}}_{z,s_1,s_2}(j_1,j_2)=1.
 \end{equation}
\end{proposition}
\begin{proof}
 We apply the Sears' transformation formula 
 \cite[(III.15)]{GasperRahman}
 to 
 ${}_{4}{\phi}_3$
 in $R_{z,s_1,s_2}$
 \eqref{eq:vertical_R_matrix_for_LJ_explicit}.
 After necessary simplifications,
 we obtain a terminating $q$-hypergeometric series
 ${}_{4}{\phi}_3$:
 \begin{equation}
 \label{eq:R_circ_as_the_infinity_limit_def_proof}
 \begin{split}
 &
 R_{z,s_1,s_2}(i_1,i_2;j_1,j_2)=
 \frac{
 s_2^{2j_2-j_1}
 z^{j_1}s_1^{-j_1}(s_1s_2/z;q)_{j_1}
 (zs_1 s_2^{-1};q)_{j_2-j_1}
 (s_2^2;q)_{i_1}
 (q^{1+j_2-j_1};q)_{i_1}}
 {(q;q)_{i_1} 
 (s_2^2;q)_{j_1}
 (z s_1s_2;q)_{i_2}
 }
 \\&\hspace{90pt}\times
 \mathbf{1}_{j_1+i_2=i_1+j_2}\cdot
 {}_{4}\phi_3\left(\begin{matrix} q^{-j_1},zs_2s_1^{-1},q^{j_2-j_1}zs_1s_2^{-1},
 q^{i_2+1}\\
 q^{1+j_2-j_1},q^{1-j_1}zs_1^{-1}s_2^{-1},q^{i_2}zs_1s_2\end{matrix}
 \bigg|\, q,q\right).
 \end{split}
 \end{equation}
 Note that in ${}_{4}\phi_3$ there is one upper and one lower parameter that each contain $q^{i_2}$
 as a factor. Then, sending $i_2$ to infinity eliminates these upper and lower parameters
 in ${}_{4}\phi_3$, producing
 $_3\phi_2$ with the remaining parameters.
 The prefactor in front of $_3\phi_2$
 readily leads to that in
 \eqref{eq:R_circ_as_the_infinity_limit_def}; recall the regularization \eqref{eq:r_phi_r_reg_function}.
 This completes the proof of the first claim.

 For the second claim,
 recall that
 the quantities
 $R_{z,s_1,s_2}
 (i_2+j_1-j_2,i_2;j_1,j_2)$
 \eqref{eq:R_circ_as_the_infinity_limit_def_proof} 
 sum to one over $j_2\ge0$ for any fixed $(i_2,j_1)$, 
 see \eqref{eq:R_sum_to_one}. 
 Due to the presence of the factor $s_2^{2j_2}$ in 
 \eqref{eq:R_circ_as_the_infinity_limit_def_proof},
 one can check that the tail 
 \begin{equation*}
 \sum_{j_2\ge M}R_{z,s_1,s_2}
 (i_2+j_1-j_2,i_2;j_1,j_2)
 \end{equation*}
 is bounded above
 uniformly in $i_2$
 by $\mathrm{const}\cdot
 (1-\varepsilon)^{M}$ for some $\varepsilon>0$.
 This implies that we can take the 
 limit $i_2\to+\infty$ inside the sum, resulting in 
 \eqref{eq:R_circ_sum_to_one}.
\end{proof}

Let us write down the specializations
of 
$R^{\mathrm{bdry}}_{z,s_1,s_2}(j_1,j_2)$
to 
$q=0$ and to the
the $q$-beta-binomial distribution, 
similar to \Cref{sub:spec_q0_of_R,sub:R_Hahn_specialization}. 
For $q=0$, we have the specialization \eqref{eq:R_at_q0_notation}--\eqref{eq:R_at_q0} for finite $i_1,i_2$. Then, by taking the limit as $i_1,i_2$ increase arbitrarily large, we obtain the following specialization:
\begin{equation}
 \label{eq:R_circ_at_q0}
 R^{\mathrm{bdry},(0)}_{z,s_1,s_2}(j_1,j_2)
 \coloneqq
 R^{\mathrm{bdry}}_{z,s_1,s_2}(j_1,j_2)
 \big\vert_{q=0}
 =
 \begin{cases}
 \widehat R^{\mathrm{bdry},(0)}_{z,s_1,s_2}(j_1,j_2),&\textnormal{if 
 $j_1\le j_2$};
 \\[4pt]
 \dfrac{s_2^{2j_2}s_1^{-2j_1}}{1-s_1^2\mathbf{1}_{j_2=0}}
 \widehat R^{\mathrm{bdry},(0)}_{z,s_2,s_1}(j_2,j_1),&\textnormal{if 
 $j_1> j_2$},
 \end{cases}
\end{equation}
where
\begin{equation}
 \label{eq:R_circ_at_q0_notation}
 \widehat R^{\mathrm{bdry},(0)}_{z,s_1,s_2}(j_1,j_2)
 \coloneqq
 \frac{z^{j_1}(s_1s_2)^{-j_1}
 s_2^{2j_2}}
 {1-zs_1s_2}
 \bigl(
 (1 - zs_1s_2^{-1} \mathbf{1}_{j_2>j_1})(1- s_2^2)+(z s_1 s_2 -s_1^2) \mathbf{1}_{j_1=j_2 >0}
 \bigr).
\end{equation}

In the $q$-beta-binomial specialization $z=s_2/s_1$,
the limiting weights
$R^{\mathrm{bdry}}_{z,s_1,s_2}(j_1,j_2)$
are exactly the same as 
pre-limit ones (see \Cref{prop:qHahn_degeneration_of_R}):
\begin{equation}
 \label{eq:qHahn_reduction_of_R_circ}
 R^{\mathrm{bdry}}_{\frac{s_2}{s_1},s_1,s_2}(j_1,j_2)=
 \mathbf{1}_{j_2\le j_1}
 \cdot
 \varphi_{q,s_2^2/s_1^2,s_2^2}(j_2\mid j_1).
\end{equation}
Indeed, setting $z=s_2/s_1$ eliminates the dependence of 
$R_{z,s_1,s_2}(i_1,i_2;j_1,j_2)$
on $i_1,i_2$. Then, one can immediately take the limit
$i_1,i_2\to+\infty$
as in \Cref{prop:R_infinity_limit}.

\medskip

We also observe that the limiting weights
$R^{\mathrm{bdry}}_{z,s_1,s_2}$ are nonnegative if $(z,s_1,s_2)\in \mathcal{R}$, see \Cref{def:parameter_range_R}.

\section{Intertwining relations}
\label{sec:comm_rel}

In this section we present our first main result, the 
intertwining
(or quasi-commutation) relations for the Markov transition 
operator of the stochastic higher spin six vertex model.
Here we discuss the result at the level of Markov operators
based on vertex weights. Then, in 
\Cref{sec:comm_rel_specializations,sec:schur_vertex_model}
below, we present its specializations
to exclusion processes on the line, such as $q$-TASEP and TASEP, and 
to the Schur vertex model.

\subsection{Swap operators}
\label{sub:swapping_operator}

Recall the state spaces
$\mathscr{G}$ and $\mathscr{X}$,
from \Cref{def:state_spaces}, and the 
Markov operators $T_{\mathbf{u},\mathbf{s}}$ 
and $\tilde{T}_{\mathbf{u},\mathbf{s}}$
on $\mathscr{G}$ and $\mathscr{X}$, respectively and described in \Cref{sub:particle_systems},
coming from the stochastic higher spin six vertex model
$\mathbf{g}(t)$
and its exclusion process counterpart $\mathbf{x}(t)$.
These Markov operators depend
on two sequences of parameters 
$(\mathbf{u},\mathbf{s})\in \mathcal{T}$
(i.e.~parameters satisfying the conditions given by
\eqref{eq:parameters_for_higher_spin_vertex_model}--
\eqref{eq:uniformly_bounded_propagation_condition}).

Here we define new Markov operators 
on $\mathscr{G}$
based on the stochastic cross vertex weights
$R_{z,s_1,s_2}$ and $R^{\mathrm{bdry}}_{z,s_1,s_2}$. 
Via the gap-particle correspondence,
these operators also define the corresponding Markov operators on $\mathscr{X}$.
\begin{definition}[Markov swap operators]
 \label{def:Pn_operator}
 For $n\ge1$
 and $(z,s_1,s_2)\in \mathcal{R}$
 (the range of parameters given in \Cref{def:parameter_range_R}),
 let $P^{(n)}_{z,s_1,s_2}$
 be the Markov operator 
 acting on $\mathscr{G}$ by 
 randomly changing the 
 coordinates $(g_{n-1},g_n)$ into
 $(g_{n-1}',g_n')$
 sampled from the cross vertex weights
 \begin{equation}
 \label{eq:Pn_operator_definition}
 \begin{cases}
 R_{z,s_1,s_2}(g_{n-1}',g_{n-1};g_n,g_n'),&n\ge2;\\
 R^{\mathrm{bdry}}_{z,s_1,s_2}(g_n,g_n'),&n=1.
 \end{cases}
 \end{equation}
 In particular, we always have $g_{n-1}'+g_n'=g_{n-1}+g_n$.
 The boundary case $n=1$ is consistent with our usual 
 agreement $g_0=g_0'=+\infty$.
 By definition, the operator
 $P^{(n)}_{z,s_1,s_2}$ 
 does not change all other coordinates $g_j$, where $j\ne n-1,n$.

 Additionally, let 
 $\tilde P^{(n)}_{z,s_1,s_2}$
 be the corresponding operator on $\mathcal{X}$ induced via the gap-particle duality.
 This operator randomly
 moves the particle $x_n$ to a new location $x_n'$
 based on the locations of the 
 neighboring particles $x_{n-1}$ and $x_{n+1}$,
 with probabilities coming from \eqref{eq:Pn_operator_definition}
 via the gap-particle transformation 
 (\Cref{def:gap_particle_transform}).
\end{definition}

The Yang-Baxter equation 
implies an intertwining relation
between $P^{(n)}$ and $T$.
This relation may also be called a (quasi-)commutation between the operators.
The following statement 
is an immediate consequence of 
\Cref{prop:YBE}:
\begin{proposition}
 \label{prop:Pn_T_intertw}
 Fix $n\ge1$.
 Let $(\mathbf{u},\mathbf{s})\in \mathcal{T}$
 be such that 
 $\bigl( \frac{u_n}{u_{n-1}},s_{n-1},s_n \bigr)\in \mathcal{R}$.
 Then, we have
 \begin{equation}
 \label{eq:Pn_T_intertw}
 T_{\mathbf{u},\mathbf{s}}
 \ssp
 P^{(n)}_{u_n/u_{n-1},s_{n-1},s_n}
 =
 P^{(n)}_{u_n/u_{n-1},s_{n-1},s_n}
 T_{\sigma_{n-1}\mathbf{u},\sigma_{n-1}\mathbf{s}},
 \end{equation}
 where $\sigma_{n-1}=(n-1,n)$ is the $n$-th elementary transposition
 in the symmetric group acting on $\mathbb{Z}_{\ge0}$.
 The same identity holds if all the operators 
 in \eqref{eq:Pn_T_intertw}
 are replaced by their counterparts acting in the space $\mathscr{X}$.
 See \Cref{fig:Pn_T_intertw} for an illustration.
\end{proposition}
Thus, we have established \Cref{prop:swap_intro} from the Introduction.

\begin{figure}[htb]
 \centering
 \includegraphics[width=.98\textwidth]{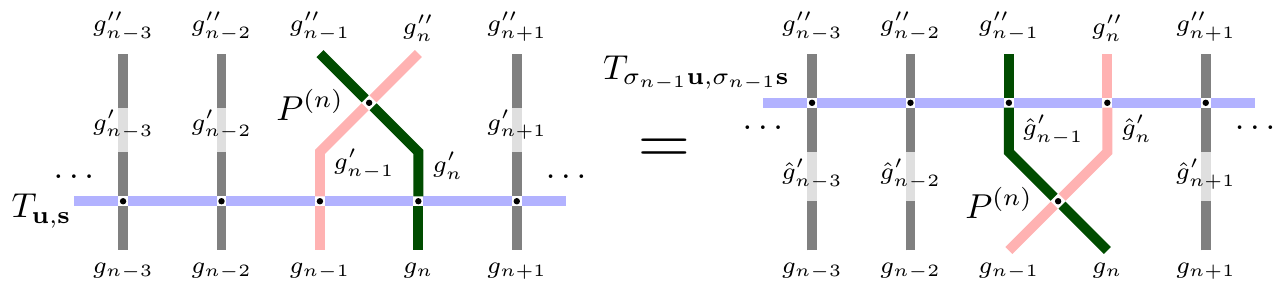}
 \caption{Relation \eqref{eq:Pn_T_intertw}--\eqref{eq:product_of_Markov_operators}
 between Markov operators 
 $T_{\mathbf{u},\mathbf{s}}$ and
 $P^{(n)}=P^{(n)}_{u_n/u_{n-1},s_{n-1},s_n}$.}
 \label{fig:Pn_T_intertw}
\end{figure}

\begin{remark}
 \label{rmk:product_of_Markov_operators}
 In \eqref{eq:Pn_T_intertw}
 and throughout the paper, we adopt 
 the convention that the product of Markov operators
 follows the order of their action on measures.
 In particular, on arbitrary
 delta measures $\delta_{\mathbf{g}}$, where $\mathbf{g}\in \mathscr{G}$
 is fixed,
 the identity \eqref{eq:Pn_T_intertw}
 is expanded as follows
 \begin{equation}
 \label{eq:product_of_Markov_operators}
 \sum_{\mathbf{g}'\in \mathscr{G}}
 T_{\mathbf{u},\mathbf{s}}
 (\mathbf{g},\mathbf{g}')
 \ssp
 P^{(n)}_{u_n/u_{n-1},s_{n-1},s_n}
 (\mathbf{g}',\mathbf{g}'')
 =
 \sum_{\hat{\mathbf{g}}'\in \mathscr{G}}
 P^{(n)}_{u_n/u_{n-1},s_{n-1},s_n}
 (\mathbf{g},\hat{\mathbf{g}}')
 \ssp
 T_{\sigma_{n-1}\mathbf{u},\sigma_{n-1}\mathbf{s}}
 (\hat{\mathbf{g}}',\mathbf{g}''),
 \end{equation}
 for any fixed $\mathbf{g},\mathbf{g}''\in \mathscr{G}$;
 see \Cref{fig:Pn_T_intertw}
 for an illustration.
 Note that both sums in \eqref{eq:product_of_Markov_operators} are finite due to the path conservation property, which is built into the operator $P^{(n)}$.
\end{remark}

\subsection{Shift operator}
\label{sub:shifting_operator}

Let us now consider a product of the operators $P^{(n)}$ over 
all $n\ge1$ with parameters chosen in such a way that the 
iterated intertwining relations
\eqref{eq:Pn_T_intertw}
lead to the shifting in
$\mathbf{u}, \mathbf{s}$:
\begin{equation}
 \label{eq:shifting_action}
 \mathsf{sh}\coloneqq\ldots \sigma_2 \sigma_1 \sigma_0 ,
 \qquad 
 \mathsf{sh}(u_0,u_1,u_2,\ldots )=
 (u_1,u_2,\ldots ),
 \qquad 
 \mathsf{sh}(s_0,s_1,s_2,\ldots )=
 (s_1,s_2,\ldots ).
\end{equation}
That is, we swap the parameters 
$(u_0,s_0)$ first with $(u_1,s_1)$,
then with 
$(u_2,s_2)$, and so on all the way to infinity. 
As a result, the parameters 
$(u_0,s_0)$ disappear, leading to the shift \eqref{eq:shifting_action}.
First, we need certain assumptions on the parameters:

\begin{definition}
 \label{def:B_operator_uniform_conditions_on_parameters}
 Denote by 
 $\mathcal{B}$ the
 space of sequences $(\mathbf{u},\mathbf{s})$
 as in \eqref{eq:parameters_for_higher_spin_vertex_model}
 such that:
 \begin{enumerate}[$\bullet$]
 \item 
 $\bigl(\frac{u_n}{u_0},s_0,s_n\bigr)\in \mathcal{R}$ for all $n\ge1$;
 \item
 There exists $\varepsilon>0$
 such that
 \begin{equation}
 \label{eq:B_operator_uniform_conditions_on_parameters}
 \frac{(-s_n)(u_ns_0-u_0s_n)}{u_0-s_0s_nu_n}<1-\varepsilon<1
 \end{equation}
 for all sufficiently large $n$.
 \end{enumerate}
\end{definition}

Similarly to \Cref{rmk:epsilon_condition}, the condition
$\bigl(\frac{u_n}{u_0},s_0,s_n\bigr)\in \mathcal{R}$
implies that the ratios in \eqref{eq:B_operator_uniform_conditions_on_parameters}
are already $\le 1$. 
However, these ratios must be bounded away from $1$ as $n$ grows.

We are now in a position to 
define the 
Markov operator 
on the space $\mathscr{G}$
which acts 
on the stochastic higher spin six vertex model
$T_{\mathbf{u},\mathbf{s}}$
by shifting the parameter sequences.

\begin{definition}[Markov shift operator]
 \label{def:B_operator}
 Let $(\mathbf{u},\mathbf{s})\in \mathcal{B}$. 
 We define the operator 
 $B_{\mathbf{u},\mathbf{s}}$ on $\mathscr{G}$
 by
 \begin{equation}
 \label{eq:B_operator}
 B_{\mathbf{u},\mathbf{s}}
 \coloneqq 
 P^{(1)}_{\frac{u_1}{u_0},s_0,s_1}
 P^{(2)}_{\frac{u_2}{u_0},s_0,s_2}
 P^{(3)}_{\frac{u_3}{u_0},s_0,s_3}
 \ldots 
 \end{equation}
 (the order follows the action on measures, cf.
 \Cref{rmk:product_of_Markov_operators}).
 See \Cref{fig:B_operator} for an 
 illustration.
 By means of the gap-particle transformation
 (\Cref{def:gap_particle_transform}),
 we also obtain a corresponding operator $\tilde B_{\mathbf{u},\mathbf{s}}$
 acting in the space $\mathscr{X}$ of particle configurations.
\end{definition}

\begin{figure}[htb]
 \centering
 \includegraphics[width=.45\textwidth]{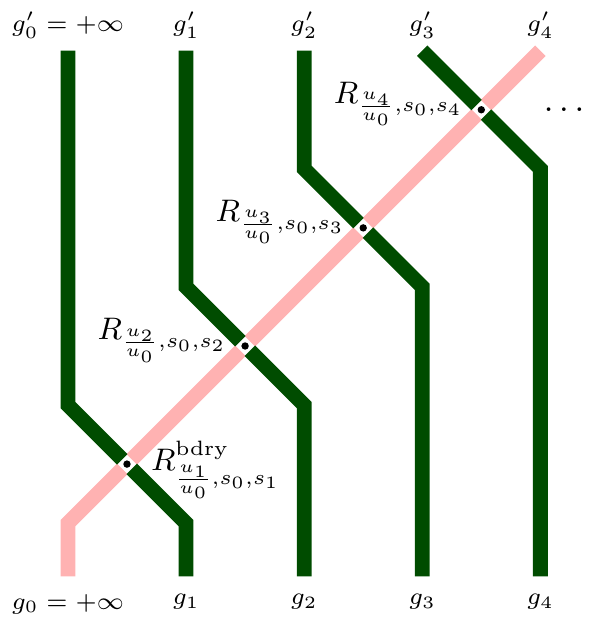}
 \caption{The path configuration 
 whose weight is the 
 matrix element
 $B_{\mathbf{u},\mathbf{s}}(\mathbf{g},\mathbf{g}')$
 of the Markov
 shift operator from
 \Cref{def:B_operator}.%
 }
 \label{fig:B_operator}
\end{figure}

\begin{lemma}
 \label{lemma:B_operator_well_defined}
 If $(\mathbf{u},\mathbf{s})\in \mathcal{B}$,
 then the shift operator 
 $B_{\mathbf{u},\mathbf{s}}$ is well-defined.
\end{lemma}
\begin{proof}
 Since
 $(u_n/u_0,s_0,s_n)\in \mathcal{R}$ for all $n\ge1$,
 all vertex weights involved in the operators
 $P^{(n)}_{u_n/u_0,s_0,s_n}$
 in the product \eqref{eq:B_operator}
 are nonnegative. 
 Next, the second condition in
 \eqref{eq:B_operator_uniform_conditions_on_parameters}
 implies that no path escapes to infinity 
 under the action of $B_{\mathbf{u},\mathbf{s}}$,
 similarly to the proof of
 \Cref{lemma:higher_spin_model_is_well_defined}. 
 Indeed, we have
 \begin{equation*}
 R_{\frac{u_n}{u_0},s_0,s_n}(0,j;0,j)=\frac{s_n^{2j}(u_ns_0/(u_0s_n);q)_{j}}
 {(u_ns_0s_n/u_0;q)_{j}},
 \qquad j \in \mathbb{Z}_{\ge0},
 \end{equation*}
 and due to 
 \eqref{eq:B_operator_uniform_conditions_on_parameters}, these quantities
 are bounded away from $1$ uniformly in $j\ge1$ for all $n$ sufficiently large.
 This completes the proof.
\end{proof}

The first main structural result is the next \Cref{thm:B_T_intertw_relation}.
It concerns the stochastic higher spin six vertex model
$T_{\mathbf{u},\mathbf{s}}$ with general horizontal
spin $J$, and shows how the operator $B_{\mathbf{u},\mathbf{s}}$
acts on it by shifting the parameter sequence. 
Below in \Cref{sec:comm_rel_specializations}, we consider specializations of \Cref{thm:B_T_intertw_relation} to 
simpler particle systems, which leads to known and 
new results. Most importantly, whereas the known results only applied to step initial conditions, the new results apply to stochastic
particle systems started from a arbitrary initial condition.

\begin{theorem}
 \label{thm:B_T_intertw_relation}
 Let $(\mathbf{u},\mathbf{s})\in \mathcal{T}\cap \mathcal{B}$.
 Then
 \begin{equation}
 \label{eq:B_T_intertw_relation}
 T_{\mathbf{u},\mathbf{s}}
 \ssp
 B_{\mathbf{u},\mathbf{s}}
 =
 B_{\mathbf{u},\mathbf{s}}
 \ssp
 T_{\mathsf{sh}(\mathbf{u}),\mathsf{sh}(\mathbf{s})},
 \end{equation}
 where $\mathsf{sh}$ is the shift \eqref{eq:shifting_action}.
 The same identity holds if all the operators 
 in \eqref{eq:B_T_intertw_relation}
 are replaced
 (via the gap-particle transformation of \Cref{def:gap_particle_transform})
 by their corresponding counterparts acting on the space $\mathscr{X}$.
\end{theorem}
\begin{proof}
 This
 is simply an iteration of
 \Cref{prop:Pn_T_intertw}.
\end{proof}

Note that the intermediate summations, as in \Cref{rmk:product_of_Markov_operators}, arising
in the products 
in both sides of 
\eqref{eq:B_T_intertw_relation}
are finite. Indeed,
for any fixed $\mathbf{g},\mathbf{g}''\in \mathscr{G}$
there are only finitely many $\mathbf{g}'\in \mathscr{G}$
such that 
$T_{\mathbf{u},\mathbf{s}}(\mathbf{g},\mathbf{g}')
\ssp
B_{\mathbf{u},\mathbf{s}}(\mathbf{g}',\mathbf{g}'')>0$.
Moreover, one can check similarly to the proofs of \Cref{lemma:higher_spin_model_is_well_defined,lemma:B_operator_well_defined} that under the condition $(\mathbf{u},\mathbf{s})\in \mathcal{T}\cap \mathcal{B}$, both products in \eqref{eq:B_T_intertw_relation} are well-defined as Markov operators in $\mathscr{G}$. That is, 
with probability $1$, the product of the Markov operators does not allow paths to run off to infinity if $(\mathbf{u},\mathbf{s})\in \mathcal{T}\cap \mathcal{B}$.

\section{Application to $q$-Hahn TASEP and its specializations}
\label{sec:comm_rel_specializations}

In this section we consider specializations
of \Cref{thm:B_T_intertw_relation}
to the $q$-Hahn TASEP \cite{Povolotsky2013}, \cite{Corwin2014qmunu}, 
various $q$-TASEPs 
\cite{SasamotoWadati1998},
\cite{BorodinCorwin2011Macdonald}, \cite{BorodinCorwin2013discrete},
and
the usual TASEP.
We recover the previously known results 
on parameter symmetry
obtained in
\cite{PetrovSaenz2019backTASEP},
\cite{petrov2019qhahn}, and extend them to 
arbitrary initial data.
All Poisson-type limit transitions involved 
in this section are the same as in \cite{petrov2019qhahn},
and therefore we only sketch the details of how discrete time Markov 
chains become continuous time Markov jump processes.

\subsection{$q$-Hahn Boson and $q$-Hahn TASEP}
\label{sub:qHahn}

Let us first recall 
\cite[Section 6.6]{BorodinPetrov2016inhom}
how the stochastic higher spin six vertex
model $T_{\mathbf{u},\mathbf{s}}$
from \Cref{sub:particle_systems}
specializes to the $q$-Hahn Boson system
from
\cite{Povolotsky2013}, \cite{Corwin2014qmunu}. 
This is achieved by setting
\begin{equation}
	\label{eq:qHahn_Boson_parameters}
	u_i=s_i\in (-1,0],\quad i\in \mathbb{Z}_{\ge0},\qquad 
	q^J=\gamma\in\bigl[1,\sup\nolimits_{i\ge0} s_i^{-2}\bigr).
\end{equation}
The stochastic vertex weights
$L^{(J)}_{s_i,s_i}$ for all $i\ge0$ 
are the $q$-beta-binomial
weights
$\varphi_{q,\gamma s_i^2,s_i^2}$ 
\eqref{eq:phi_finite},
see
\eqref{eq:qHahn_reduction_of_LJ}--\eqref{eq:L_J_infinity_weights}.
These weights are nonnegative under the parameter assumptions 
\eqref{eq:qHahn_Boson_parameters}.
Thus, the resulting
stochastic vertex model is
called the (\emph{stochastic}) \emph{$q$-Hahn Boson system}.
Denote its one-step Markov transition operator
acting on $\mathscr{G}$ by
$T_{\gamma,\mathbf{s}}^{\mathrm{qHahn}}$,
and the corresponding Markov transition operator
acting on $\mathscr{X}$
(via the gap-particle transformation, see \Cref{def:gap_particle_transform})
by
$\tilde T_{\gamma,\mathbf{s}}^{\mathrm{qHahn}}$.
Throughout this section it is convenient to 
work in the exclusion process state space $\mathscr{X}$.

The stochastic particle system on $\mathscr{X}$
with Markov transition operator 
$\tilde T_{\gamma,\mathbf{s}}^{\mathrm{qHahn}}$
is called the
\emph{$q$-Hahn TASEP}.
In $q$-Hahn TASEP, the updates are performed \emph{in parallel},
as opposed to the sequential update
in the general case. That is,
each particle $x_n$ jumps to the right independently
of other particles by a random distance $h_{n-1}$, 
where $0\le h_{n-1}\le x_{n-1}-x_n-1 $,
with probability
\begin{equation}
	\label{eq:qHahn_jumping_probability}
	\varphi_{q,\gamma s_{n-1}^2,s_{n-1}^2}(h_{n-1}\mid x_{n-1}-x_n-1),
	\qquad 
	n=1,2,\ldots .
\end{equation}
Here, we use the notation $h_{n-1}$ in agreement with
\Cref{sub:particle_systems}.
For $n=1$, we have $x_0=+\infty$, so the jumping distribution
is given by
\eqref{eq:phi_infinite}.

\medskip

Let us denote the 
$q$-Hahn specialization of the 
swap operator
(\Cref{def:Pn_operator})
acting on the space $\mathscr{G}$
by $P^{(n),\mathrm{qHahn}}_{s_{n-1},s_n}$,
and its corresponding counterpart acting on the space
$\mathscr{X}$ by $\tilde P^{(n),\mathrm{qHahn}}_{s_{n-1},s_n}$. 
These operators involve the 
cross vertex weights
\begin{equation}
	\label{eq:qHahn_Boson_R_weights_for_swap}
	\begin{split}
		R^{\mathrm{bdry}}_{\frac{s_1}{s_0},s_0,s_1}(g_1,g_1')
		&=\mathbf{1}_{g_1'\le g_1}\cdot \varphi_{q,s_1^2/s_0^2,s_1^2}(g_1'\mid g_1)
		,\\
		R_{\frac{s_n}{s_{n-1}},s_{n-1},s_n}(g_{n-1}',g_{n-1};g_n,g_n')
		&=
		\mathbf{1}_{g_{n-1}'+g_n'=g_{n-1}+g_n}
		\cdot
		\mathbf{1}_{g_n'\le g_n}\cdot
		\varphi_{q,s_n^2/s_{n-1}^2,s_n^2}(g_n'\mid g_n),
	\end{split}
\end{equation}
where $n\ge2$,
which specialize to the $q$-beta-binomial 
distribution (\Cref{prop:qHahn_degeneration_of_R}).
By \eqref{eq:R_matrix_qHahn_reduction_nonnegative},
the operators
$P^{(n),\mathrm{qHahn}}_{s_{n-1},s_n}$ and
$\tilde P^{(n),\mathrm{qHahn}}_{s_{n-1},s_n}$ have nonnegative matrix elements
if $s_{n-1},s_n\in(-1,0]$ and $|s_n|\le |s_{n-1}|$.

\begin{remark}
	\label{rmk:qHahn_not_only_process_see_S8}
	Setting $s_i=u_i$ is not the only way of making the cross vertex weights 
	to take the 
	simpler
	$q$-beta-binomial form.
	For instance, one could take $J\in \mathbb{Z}_{\ge 1}$ 
	and require that $s_i/s_j=u_i/u_j$ for all $i,j$.
	We do not focus on this case in the current \Cref{sec:comm_rel_specializations}
	since this section is
	devoted to extending existing results 
	from \cite{PetrovSaenz2019backTASEP}, \cite{petrov2019qhahn}
	on $q$-Hahn TASEP and 
	its specializations.
	Results very similar to
	the ones below in \Cref{sec:comm_rel_specializations}
	hold for the subfamily of stochastic higher spin six vertex models
	with 
	$s_i/s_j=u_i/u_j$ for all $i,j$.
	We return to this subfamily in \Cref{sec:discrete_time_bijectivisation} below.
\end{remark}

\Cref{prop:Pn_T_intertw} extends the action of the $q$-Hahn swap operators from
\cite{petrov2019qhahn}
to general initial data.
Let us fix some notation to formulate the result. 
Fix a discrete time $t\in \mathbb{Z}_{\ge0}$ and a particle configuration $\mathbf{y}\in \mathscr{X}$.
Let $\mathbf{x}(t)$
be the particle configuration of the $q$-Hahn TASEP at time $t$
started from the initial particle configuration $\mathbf{x}(0)=\mathbf{y}$,
with parameters $\mathbf{s}=(s_0,s_1,s_2,\ldots )$, 
$-1<s_i\le 0$.
Additionally, let $n\in \mathbb{Z}_{\ge1}$ and assume that $|s_n|\le |s_{n-1}|$.
Applying the swap operator $\tilde P^{(n),\mathrm{qHahn}}_{s_{n-1},s_n}$
to the configuration 
$\mathbf{x}(t)$ at time $t$ moves a single particle
$x_n(t)$ to a random new location $x_n'(t)$ with probability
\begin{equation}
	\label{eq:qHahn_swap_operator_explicit}
	\varphi_{q,s_n^2/s_{n-1}^2,s_{n}^2}(x_n'(t)-x_{n+1}(t)-1\mid 
	x_n(t)-x_{n+1}(t)-1).
\end{equation}
Denote the resulting configuration by $\mathbf{x}'(t)$.

\begin{proposition}
	[Extension of {\cite[Theorem 3.8]{petrov2019qhahn} to general initial data}]
	\label{prop:qHahn_Pn_T_intertw}
	Take the notation above. Then, the random configuration
	$\mathbf{x}'(t)$ coincides in distribution
	with 
	the configuration of the $q$-Hahn TASEP at time $t$
	started from a \emph{random} initial configuration 
	$\delta_{\mathbf{y}}\tilde P^{(n),\mathrm{qHahn}}_{s_{n-1},s_n}$
	that evolves with the swapped parameters $\sigma_{n-1}\mathbf{s}$,
	where $\sigma_{n-1}=(n-1,n)$ is the $n$-th elementary transposition.
\end{proposition}
\begin{proof}
	This is the $q$-Hahn specialization of 
	the intertwining relation
	\eqref{eq:Pn_T_intertw}
	of \Cref{prop:Pn_T_intertw}
	between the swap operator and the time evolution operator
	(applied $t$ times). Note that this result is formulated for the
	space $\mathscr{X}$ of particle configurations in $\mathbb{Z}$.
\end{proof}

The swap operator 
$\tilde P^{(n),\mathrm{qHahn}}_{s_{n-1},s_n}$ preserves
the step initial configuration
$\mathbf{y}=\mathbf{x}_{step}$
due to the presence of the indicator $\mathbf{1}_{g_n'\le g_n}$
in \eqref{eq:qHahn_Boson_R_weights_for_swap}.
Thus, the swap operator does not 
randomize the initial configuration
for the $q$-Hahn TASEP started with $\mathbf{x}_{step}$.
In the case of step initial data,
\Cref{prop:qHahn_Pn_T_intertw} was proven in
\cite{petrov2019qhahn} (up to matching notation
$\nu_i=s_{i-1}^2$)
using exact formulas which are not readily available 
for general initial data.
Finally, observe that
$\tilde P^{(n),\mathrm{qHahn}}_{s_{n-1},s_n}$
becomes the identity operator
when $s_n=s_{n-1}$,
and the statement of \Cref{prop:qHahn_Pn_T_intertw} in this case
is trivial (while still true).

\medskip

One may also specialize 
\Cref{thm:B_T_intertw_relation}
to the case of $q$-Hahn TASEP. The result is a shift operator
which acts by removing the parameter
$s_0$ if $|s_0|\ge |s_n|$ for all $n\ge1$.
In a continuous time limit, this leads to an extension of
\cite[Theorem 4.7]{petrov2019qhahn}
to general initial data.
The original result \cite[Theorem 4.7]{petrov2019qhahn}
for the step initial data 
follows by observing that
the
$q$-Hahn specialization of the 
shift operator preserves $\mathbf{x}_{step}$.
To avoid cumbersome notation, in this paper we only consider the
continuous limit for the case of $q$-TASEP, see
\Cref{sub:QTASEP_cont,sub:QTASEP_cont_back} below.

\subsection{Intertwining relation for geometric $q$-TASEP}
\label{sub:QTASEP_discr}

Let us now consider the limit of the $q$-Hahn TASEP leading to 
the discrete time $q$-TASEP: 
\begin{equation}
	\label{eq:from_qH_to_qTASEP}
	s_n^2\to 0,\qquad \gamma s_n^2\to a_n\in(0,1),\qquad n=0,1,\ldots. 
\end{equation}
This also implies that $s_n^2/s_{k}^2\to a_n/a_k$.
Under this limit, the $q$-Hahn TASEP turns into the
\emph{discrete time geometric $q$-TASEP} introduced
in \cite{BorodinCorwin2013discrete}.
During each time step in
this process, each particle $x_n$, $n\in \mathbb{Z}_{\ge1}$, jumps to the right
independently of other particles 
by a random distance $h_{n-1}$ with probability
(see \eqref{eq:qHahn_jumping_probability})
\begin{equation}
	\label{eq:geometric_qTASEP_jumping_probability}
	\varphi_{q,a_{n-1},0}(h \mid g)=
	a_{n-1}^h
	(a_{n-1};q)_{g-h}
	\ssp
	\frac{(q;q)_g}{(q;q)_h(q;q)_{g-h}}
	,
	\qquad h=h_{n-1},\ g=x_{n-1}-x_n-1.
\end{equation}
When $n=1$, we have $g=+\infty$, by agreement.
Each $a_{n-1}$ may be viewed as the 
speed parameter attached to the particle $x_{n}$ in the $q$-TASEP.
When $q=0$,
the jumping distance \eqref{eq:geometric_qTASEP_jumping_probability} becomes
$h_{n-1}=\min(\eta,x_{n-1}-x_n-1)$, where $\eta\in \mathbb{Z}_{\ge0}$
is a geometric random variable with $\mathbb{P}(\eta=k)=a_{n-1}^k(1-a_{n-1})$. Hence, the name ``geometric''.

\medskip

Observe that the geometric $q$-TASEP swap operator
depends only on the ratio of the speed parameters,
and involves the vertex weights
$\varphi_{q,a_n/a_{n-1},0}$ similarly to 
\eqref{eq:qHahn_Boson_R_weights_for_swap}--\eqref{eq:qHahn_swap_operator_explicit}.
The swap operator and the $q$-TASEP evolution satisfy
a relation similar to \Cref{prop:qHahn_Pn_T_intertw}.

\medskip

Let us further specialize the speed parameters in the $q$-TASEP
by setting 
\begin{equation}
	\label{eq:qTASEP_regular_parameters_specialization}
	a_k=\alpha r^{k},\qquad  k\in \mathbb{Z}_{\ge0}, 
\end{equation}
where $\alpha,r\in (0,1)$ 
are fixed.
Denote the 
Markov transition operator of this $q$-TASEP acting on $\mathscr{X}$
by $\tilde T^{\mathrm{qT}}_{\alpha,r}$.
Using \Cref{def:B_operator},
let us also denote by 
$\tilde B^{\mathrm{qT}}_{r}$ the 
corresponding
shift operator.
Note that it does not depend on $\alpha$
and involves the vertex weights
$\varphi_{q,r^n,0}$, where $r^n=a_n/a_0$.
\begin{proposition}
	\label{prop:gqT_shift}
	Fix $t\in \mathbb{Z}_{\ge0}$ and $\mathbf{y}\in \mathscr{X}$.
	Let $\mathbf{x}(t)$ be the configuration of 
	the geometric $q$-TASEP 
	$\tilde T^{\mathrm{qT}}_{\alpha,r}$ at time $t$,
	started from $\mathbf{y}=\mathbf{x}(0)$. Also, let $\mathbf{x}'(t)$
	be the configuration resulting from applying the shift operator $\tilde B^{\mathrm{qT}}_{r}$
	to $\mathbf{x}(t)$.
    Then,
	$\mathbf{x}'(t)$ coincides in distribution with the $q$-TASEP
	$\tilde T^{\mathrm{qT}}_{\alpha r,r}$
	at time $t$ started from a random initial configuration
	$\delta_{\mathbf{y}}\tilde B^{\mathrm{qT}}_{r}$
	and evolving with the modified 
	parameters $a_k'=\alpha r^{k+1}$, $k\in \mathbb{Z}_{\ge0}$.
\end{proposition}
In terms of the operators, the statement is equivalent to 
the following intertwining relation:
\begin{equation}
	\label{eq:geo_qTASEP_shift_operator}
	\bigl(\tilde T^{\mathrm{qT}}_{\alpha,r}\bigr)^t
	\tilde B^{\mathrm{qT}}_{r}
	=
	\tilde B^{\mathrm{qT}}_{r}
	\bigl(\tilde T^{\mathrm{qT}}_{\alpha r,r}\bigr)^t,
\end{equation}
where the order of the operators is understood 
as in \Cref{rmk:product_of_Markov_operators},
and $(\cdots)^t$ means raising to the nonnegative integer power $t$.
\begin{proof}[Proof of \Cref{prop:gqT_shift}]
	This is a specialization of 
	\Cref{thm:B_T_intertw_relation}
	formulated in terms of particle systems.
	Note that the operator
	$\tilde T^{\mathrm{qT}}_{\alpha,r}$
	is well-defined for $a_k=\alpha r^k$, 
	since $0<a_k<1$ for all $k$.
	Also, the shift operator 
	$\tilde B^{\mathrm{qT}}_{r}$
	is well-defined by \Cref{lemma:B_operator_well_defined}
	since $a_k>a_{k+1}$ for all $k$, and the second condition
	in \Cref{def:B_operator_uniform_conditions_on_parameters}
	holds trivially if $u_i = s_i$ for all $i$.
\end{proof}

\subsection{Limit to continuous time $q$-TASEP}
\label{sub:QTASEP_cont}

Let us now take a Poisson-type limit to continuous time for the geometric
$q$-TASEP $\tilde T^{\mathrm{qT}}_{\alpha,r}$.
This is achieved by letting
\begin{equation}
	\label{eq:alpha_t_limit_to_cont_time_qTASEP}
	\alpha\to 0,
	\qquad 
	t=\lfloor 
	(1-q)\mathsf{t}/\alpha
	\rfloor ,
\end{equation}
where 
$\mathsf{t}\in \mathbb{R}_{\ge0}$ is the new continuous time variable.
Indeed, observe the expansion
\begin{equation}
	\label{eq:alpha_t_limit_to_cont_time_qTASEP_1}
	\varphi_{q,\mu,0}(h\mid g)=
	\begin{cases}
		1+O(\mu),& h=0;\\[4pt]
		\dfrac{1-q^g}{1-q}\ssp\mu+O(\mu^2),& h=1;\\[7pt]
		O(\mu^2),&h\ge1,
	\end{cases}
	\qquad \qquad 
	\mu\to 0.
\end{equation}
This means that particles jump very rarely for small $\alpha$ in discrete time. 
Moreover, when a particle jumps, 
it jumps by one with much higher probability than any other distance greater than one. 
Then, by speeding up the time, the
discrete jumping distributions 
$\varphi_{q,\alpha r^{n-1},0}$ \eqref{eq:geometric_qTASEP_jumping_probability}
lead to independent exponential clocks. Therefore,
under the resulting 
\emph{continuous time $q$-TASEP} \cite{BorodinCorwin2011Macdonald}, 
each particle $x_n$ has an independent exponential clock
of rate $r^{n-1}\left( 1-q^{g_n} \right)$,
where $g_n=x_n-x_{n+1}-1$ (the factor $1-q$ in the rate in \eqref{eq:alpha_t_limit_to_cont_time_qTASEP_1}
is removed by the time scaling \eqref{eq:alpha_t_limit_to_cont_time_qTASEP}). 
When the clock attached to the particle $x_n$ rings, this 
particle jumps by $1$ to the right. Note that when $g_n=0$, the jump rate of $x_n$
is zero, which means that a particle cannot jump into an occupied location.

\begin{remark}[TASEP specialization, $q=0$]
	\label{rmk:TASEP_definition}
	The continuous time $q$-TASEP
	with the sequence of speeds $(1,r,r^2,\ldots )$
	turns into the \emph{TASEP} with these speeds when $q=0$.
	Under TASEP, each particle $x_n$ has an independent
	exponential clock with rate $r^{n-1}$. When 
	a clock rings, the corresponding particle jumps
	to the right by one, provided that the destination is unoccupied.
	Otherwise, the jump of the particle is blocked.

	Moreover, in the case $r=1$, we recover the well-known 
	\emph{homogeneous continuous time TASEP} in which the speeds 
	of all particles are equal to $1$.
\end{remark}

Let us denote the continuous time Markov semigroup on $\mathscr{X}$
corresponding to the continuous time $q$-TASEP
with 
particle speeds $(1,r,r^2,\ldots )$
by
$\{\tilde T^{\mathrm{qT}}_r(\mathsf{t})\}_{\mathsf{t}\in \mathbb{R}_{\ge0}}$.
In the case $r=1$, the process given by the semigroup
$\tilde T^{\mathrm{qT}}_1(\mathsf{t})$ 
(which we will denote simply by $\tilde T^{\mathrm{qT}}(\mathsf{t})$)
is the \emph{homogeneous $q$-TASEP}, where all particles have speeds equal to $1$.

Taking the continuous time limit
\eqref{eq:alpha_t_limit_to_cont_time_qTASEP}
in \eqref{eq:geo_qTASEP_shift_operator}, we get the following intertwining
relations for any $m\in \mathbb{Z}_{\ge1}$:
\begin{equation}
	\label{eq:B_T_intertw_relation_cont_qTASEP}
	\tilde T^{\mathrm{qT}}_r(\mathsf{t})
	\bigl(  
	\tilde B^{\mathrm{qT}}_{r}
	\bigr)^{m}
	=
	\bigl(  
	\tilde B^{\mathrm{qT}}_{r}
	\bigr)^{m}\ssp
	\tilde T^{\mathrm{qT}}_r(r^{m}\ssp \mathsf{t}).
\end{equation}
Indeed, shifting
the sequence of speed parameters as
$(1,r,r^2,\ldots )\mapsto (r^m,r^{m+1},r^{m+2},\ldots )$
means slowing down all the particles by 
the factor $r^{m}$, which is equivalent to looking at the 
$q$-TASEP
distribution
at an earlier time $r^{m}\ssp\mathsf{t}$.

\subsection{Mapping $q$-TASEP back in time}
\label{sub:QTASEP_cont_back}

We now aim to take one more Poisson-type limit 
in the intertwining relation 
\eqref{eq:B_T_intertw_relation_cont_qTASEP}. Let
\begin{equation}
	\label{eq:qTASEP_back_Poisson_limit}
	r= 1-\varepsilon,\qquad m=\lfloor \tau/\varepsilon \rfloor ,\qquad 
	\varepsilon\searrow 0,
\end{equation}
where $\tau\in \mathbb{R}_{\ge0}$ is a new continuous time parameter.
Under \eqref{eq:qTASEP_back_Poisson_limit}, one readily sees that
the $q$-TASEP Markov operators in both sides of
\eqref{eq:B_T_intertw_relation_cont_qTASEP}
turn into the operators
$\tilde T^{\mathrm{qT}}(\mathsf{t})$
and 
$\tilde T^{\mathrm{qT}}(e^{-\tau} \ssp \mathsf{t})$, respectively.
Recall that the latter two operators
correspond the homogeneous $q$-TASEP where all particles move with homogeneous speed one.

Let us consider the limit of 
$\bigl(  
	\tilde B^{\mathrm{qT}}_{r}
\bigr)^{m}$.
In particular, consider the cross
vertex weights~\eqref{eq:qHahn_Boson_R_weights_for_swap}
in the limit $\varepsilon \to 0$
under the specialization \eqref{eq:from_qH_to_qTASEP} and  \eqref{eq:qTASEP_regular_parameters_specialization}.
For any fixed $n\in \mathbb{Z}_{\ge0}$, we have:
\begin{equation}
	\label{eq:qHahn_distribution_expansion}
	\varphi_{q,r^n,0}(g' \mid g)=
	\begin{cases}
		\displaystyle\frac{n}{1-q^{g-g'}}\frac{(q;q)_g}{(q;q)_{g'}}
		\ssp \varepsilon+O(\varepsilon^2),& 0\le g'\le g-1;
		\\[8pt]
		1-ng\ssp\varepsilon+O(\varepsilon^2),& g'=g.
	\end{cases}
\end{equation}
Note that the quantity $r^n$ arises as $a_n/a_0$, see \eqref{eq:qTASEP_regular_parameters_specialization}.

We have the following interpretation for the expansion \eqref{eq:qHahn_distribution_expansion}.
For small $\varepsilon$,
the action of a single shift operator
$\tilde B^{\mathrm{qT}}_{r}$
does not change the particle configuration with high probability.
Speeding up the time leads to exponential particle jumps
with rates coming from the coefficients of the $\varepsilon$-terms 
in \eqref{eq:qHahn_distribution_expansion}. In the limit regime
\eqref{eq:qTASEP_back_Poisson_limit}
the operators 
$\bigl(  
\tilde B^{\mathrm{qT}}_{r}
\bigr)^{m}$ on $\mathscr{X}$
converge into a continuous time Markov semigroup
$\{\tilde B^{\mathrm{qT}}(\tau)\}_{\tau\in \mathbb{R}_{\ge0}}$
on $\mathscr{X}$, where the 
convergence is in the sense of matrix elements of Markov operators
on $\mathscr{X}$.

We call the Markov semigroup
$\tilde B^{\mathrm{qT}}(\tau)$
on $\mathscr{X}$ the \emph{backwards $q$-TASEP dynamics}.
Under this dynamics,
each particle $x_n$ has an independent
exponential clock with rate $ng_n=n(x_n-x_{n+1}-1)$.
When a clock rings, 
the corresponding particle $x_n$ instantaneously jumps backwards to a new location
$x_n'<x_n$ with probability
\begin{equation}
	\label{eq:backwards_jumping_probability}
	\frac{1}{g_n(1-q^{x_n-x_n'})}\frac{(q;q)_{g_n}}{(q;q)_{g_n'}},
	\qquad \textnormal{where}\quad
	g_n=x_n-x_{n+1}-1,\quad
	g_n'=x_n'-x_{n+1}-1.
\end{equation}
Observe that for any configuration in $\mathscr{X}$,
the sum of the jump rates of all possible particle jumps is finite, meaning that 
the backwards $q$-TASEP on $\mathscr{X}$ is well-defined.

\begin{remark}
	\label{rmk:TASEP_backwards_dynamics}
	In the TASEP specialization, when $q=0$ (cf. \Cref{rmk:TASEP_definition}),
	the probabilities \eqref{eq:backwards_jumping_probability}
	define a uniform distribution. Therefore, under the backwards
	dynamics, when the clock of the particle $x_n$
	rings (with rate $n(x_n-x_{n+1}-1)$), 
	this particle selects one of the following locations 
	\begin{equation*}
		\left\{ x_{n+1}+1,x_{n+1}+2,\ldots,x_n-2,x_n-1  \right\}
	\end{equation*}
	uniformly at random, and instantaneously jumps into the 
	selected location.
	Thus, setting $q=0$ turns the backwards $q$-TASEP dynamics
	$\tilde B^{\mathrm{qT}}(\tau)$
	into the (\emph{inhomogeneous})
	\emph{backwards Hammersley process} introduced
	in \cite{PetrovSaenz2019backTASEP} (see \Cref{fig:intro_BHP} for an illustration).
\end{remark}

Taking the Poisson-type limit 
\eqref{eq:qTASEP_back_Poisson_limit}
of the intertwining relation \eqref{eq:B_T_intertw_relation_cont_qTASEP},
we immediately obtain the main result of \Cref{sec:comm_rel_specializations}
(this is \Cref{thm:cont_intertwining_intro} from the Introduction):

\begin{theorem}
	\label{thm:mapping_qTASEP_TASEP_back_general_IC}
	Let 
	$\{\tilde T^{\mathrm{qT}}(\mathsf{t})\}_{\mathsf{t}\in \mathbb{R}_{\ge0}}$
	and
	$\{\tilde B^{\mathrm{qT}}(\tau)\}_{\tau\in \mathbb{R}_{\ge0}}$
	be the
	Markov semigroups of the homogeneous $q$-TASEP 
	and the backwards $q$-TASEP
	on $\mathscr{X}$, respectively. Then
	\begin{equation}
		\label{eq:mapping_qTASEP_TASEP_back_general_IC}
		\tilde T^{\mathrm{qT}}(\mathsf{t})
		\ssp
		\tilde B^{\mathrm{qT}}(\tau)
		=
		\tilde B^{\mathrm{qT}}(\tau)
		\ssp
		\tilde T^{\mathrm{qT}}\bigl(e^{-\tau}\mathsf{t}\bigr)
		\qquad 
		\textnormal{for all $\mathsf{t}, \tau \in \mathbb{R}_{\ge0}$}.
	\end{equation}
	The same identity holds if all the operators 
	are replaced
	(via the gap-particle transformation of \Cref{def:gap_particle_transform})
	by their counterparts acting in the vertex model space $\mathscr{G}$.
\end{theorem}

\Cref{thm:mapping_qTASEP_TASEP_back_general_IC}
may be reformulated equivalently
in terms 
of stochastic particle systems on $\mathbb{Z}$.
Fix $\mathbf{y}\in \mathscr{X}$,
and let $\mathbf{x}(\mathsf{t})$ denote
the configuration of the homogeneous 
$q$-TASEP at time $\mathsf{t}$ started with initial condition $\mathbf{x}(0)=\mathbf{y}$.
Fix $\tau$, and run 
the backwards $q$-TASEP dynamics from the configuration $\mathbf{x}(\mathsf{t})$
for time $\tau$. Then, the
distribution of the resulting configuration
is the same as the distribution of the 
$q$-TASEP at time $e^{-\tau}\mathsf{t}$ with \emph{random} initial configuration
$\delta_{\mathbf{y}}\tilde B^{\mathrm{qT}}(\tau)$.

We recover the $\nu=0$ case\footnote{An
intertwining relation for general $\nu$ is also readily obtained in a continuous
time limit from the shift operator for the $q$-Hahn TASEP, but 
in the present paper
we omit 
this statement, as well as its limit to the 
beta polymer as in \cite[Section 6]{petrov2019qhahn}.} of
\cite[Theorem 4.7]{petrov2019qhahn} by setting the initial configuration $\mathbf{y}$ is $\mathbf{x}_{step}$. In particular, note that the configuration $\mathbf{x}_{step}$ is fixed by $\tilde B^{\mathrm{qT}}(\tau)$. The Theorem in \cite[Theorem 4.7]{petrov2019qhahn} states that the backwards dynamics maps the distribution of 
$q$-TASEP with step initial data backwards in time, from $\mathsf{t}$ to $e^{-\tau}\mathsf{t}$, by applying the backwards $q$-TASEP for time $\tau$. Moreover, setting $q=0$
recovers 
\cite[Theorem 1]{PetrovSaenz2019backTASEP}
for the homogeneous TASEP.


\subsection{Lax equation for $q$-TASEP and TASEP}
\label{sub:Lax_equations}

We obtain a Lax type equation for the 
$q$-TASEP (and TASEP in the special case $q=0$), arising from 
identity 
\eqref{eq:mapping_qTASEP_TASEP_back_general_IC}
established in \Cref{thm:mapping_qTASEP_TASEP_back_general_IC}.
Our computations in this subsection are informal,
though we believe that
the end results
\eqref{eq:Lax_equation_qTASEP}
and 
\eqref{eq:Lax_equation_qTASEP_expectations}
become rigorous 
in appropriate spaces of functions.
We believe that our Lax equation could be employed to study
multipoint asymptotics of the $q$-TASEP and, in a scaling
limit, lead to Kadomtsev–Petviashvili (KP) or Korteweg–de
Vries (KdV) type equations recently derived in
\cite{quastel2019kp} for the KPZ fixed point process
\cite{matetski2017kpz}. We leave the asymptotic analysis of
the Lax equation to future work.

\medskip

Let 
$\tilde{\mathsf{T}}$
and
$\tilde{\mathsf{B}}$
denote the infinitesimal generators of the $q$-TASEP and the 
backwards $q$-TASEP, respectively.
Multiply both sides of \eqref{eq:mapping_qTASEP_TASEP_back_general_IC}
by $\tilde T^{\mathrm{qT}}(\mathsf{t}-e^{-\tau}\mathsf{t})$
from the right. Using the semigroup property of $\tilde T^{\mathrm{qT}}(\mathsf{t})$, 
we obtain
\begin{equation*}
	\tilde T^{\mathrm{qT}}(\mathsf{t})
	\ssp
	\tilde B^{\mathrm{qT}}(\tau)
	\ssp
	\tilde T^{\mathrm{qT}}(\mathsf{t}-e^{-\tau}\mathsf{t})
	=
	\tilde B^{\mathrm{qT}}(\tau)
	\ssp
	\tilde T^{\mathrm{qT}}\bigl(\mathsf{t}\bigr).
\end{equation*}
Fix $\mathsf{t}>0$, and differentiate this identity 
in $\tau$ at $\tau=0$. We obtain
\begin{equation*}
	\tilde T^{\mathrm{qT}}(\mathsf{t})
	\bigl( \tilde{\mathsf{B}}+\mathsf{t}\cdot 
	\tilde{\mathsf{T}} \bigr)
	=\tilde{\mathsf{B}}\ssp
	\tilde T^{\mathrm{qT}}(\mathsf{t}).
\end{equation*}
Dividing by $\mathsf{t}$, rewrite this as
\begin{equation}
	\label{eq:Lax_equation_qTASEP_proof}
	\tilde T^{\mathrm{qT}}(\mathsf{t})
	\ssp
	\tilde{\mathsf{T}}
	=
	\left[ \tfrac{1}{\mathsf{t}} \tilde{\mathsf{B}}, 
	\tilde T^{\mathrm{qT}}(\mathsf{t}) \right],
\end{equation}
where $[\cdot,\cdot]$ is the commutator of operators.
Using Kolmogorov (also called Fokker--Planck) equation,
we can express the left-hand side as a derivative in $\mathsf{t}$.
Thus, we obtain
\begin{equation}
	\label{eq:Lax_equation_qTASEP}
	\frac{d}{d\ssp \mathsf{t}}
	\tilde{\mathsf{T}}^{\mathrm{qT}}(\mathsf{t})
	=
	\left[ \tfrac{1}{\mathsf{t}} \tilde{\mathsf{B}}, 
	\tilde T^{\mathrm{qT}}(\mathsf{t}) \right],
\end{equation}
a differential equation for
the $q$-TASEP semigroup in the Lax form.

Let us apply the Lax equation
to an 
arbitrary (sufficiently nice)
function
$F$ on the space~$\mathscr{X}$.
Note that we have
$\bigl(\tilde{\mathsf{T}}^{\mathrm{qT}}(\mathsf{t})
F\bigr)(\mathbf{y})=	
\mathbb{E}_{\mathbf{y}}\left[ F(\mathbf{x}(\mathsf{t})) \right]$,
where the expectation is with respect to the $q$-TASEP at time $\mathsf{t}$
started from $\mathbf{y}$, since $\tilde{\mathsf{T}}^{\mathrm{qT}}(\mathsf{t})$ is a Markov semigroup.
Then,
from \eqref{eq:Lax_equation_qTASEP_proof}, we obtain the following:
\begin{equation}
	\label{eq:Lax_equation_qTASEP_expectations}
	\mathsf{t}\,
	\mathbb{E}_{\mathbf{y}}\bigl[ \bigl(\tilde{\mathsf{T}}\ssp
	F\bigr)(\mathbf{x}(\mathsf{t})) \bigr]
	=
	\tilde{\mathsf{B}}
	\ssp
	\mathbb{E}_{\mathbf{y}}
	\left[ F(\mathbf{x}(\mathsf{t})) \right]
	-
	\mathbb{E}_{\mathbf{y}}
	\bigl[ \bigl(\tilde{\mathsf{B}}\ssp F\bigr)(\mathbf{x}(\mathsf{t})) \bigr],
\end{equation}
where the operator $\tilde{\mathsf{B}}$ on the right side acts on 
the expectation as a function in~$\mathbf{y}$, for the first term, and
on the function $F$, for the second term.

\medskip

Identity \eqref{eq:Lax_equation_qTASEP_expectations}
generalizes
\cite[Proposition 5.3]{petrov2019qhahn}
(and also \cite[Proposition 7.1]{PetrovSaenz2019backTASEP} when $q=0$)
by allowing an arbitrary initial condition~$\mathbf{y}$.
Indeed, if $\mathbf{y}=\mathbf{x}_{step}$, then
$\tilde{\mathsf{B}}
\ssp\mathbb{E}_{\mathbf{y}}
\left[ F(\mathbf{x}(\mathsf{t})) \right]=0$
because $\mathbf{x}_{step}$ is an absorbing state for $\tilde{\mathsf{B}}$.
Thus, the combined generator 
$\mathsf{t}\ssp\tilde{\mathsf{T}}+\tilde{\mathsf{B}}$
satisfies 
\begin{equation*}
	\mathbb{E}_{\mathbf{x}_{step}}
	\bigl[ 
		\bigl(
			\mathsf{t}\ssp\tilde{\mathsf{T}}F+\tilde{\mathsf{B}}F
		\bigr)(\mathbf{x}(\mathsf{t}))
	\bigr]=0,
\end{equation*}
so the process with this combined generator
preserves
the time $\mathsf{t}$ distribution 
of the $q$-TASEP
started 
from the step initial
configuration.
This preservation of measure
was proven in \cite{petrov2019qhahn}
using contour integral formulas available for the $q$-TASEP distribution with the
step initial configuration,
and for $q=0$ in \cite{PetrovSaenz2019backTASEP}
using a different approach.
Moreover, using duality, in \cite{petrov2019qhahn} it was shown that the 
process with the combined generator converges to its stationary distribution
when started from an arbitrary initial configuration in $\mathscr{X}$.

\section{Application to Schur vertex model}
\label{sec:schur_vertex_model}

The Schur vertex model
studied in
\cite{SaenzKnizelPetrov2018}
is the $J=1$, $q=0$ specialization of the
stochastic higher spin six vertex model 
$\mathbf{g}(t)$ defined in 
\Cref{sub:particle_systems}.
The name ``Schur'' comes from the fact that
some joint distributions in $\mathbf{g}(t)$
are expressed through the Schur processes 
\cite[Theorem~3.5]{SaenzKnizelPetrov2018}.
This model can be equivalently reformulated as a certain corner growth
\cite[Section~1.2]{SaenzKnizelPetrov2018}, and is also equivalent
to the generalized TASEP of
\cite{derbyshev2012totally} and
\cite{Povolotsky2013},
which appeared (in the form of tandem
queues and first passage percolation models)
already in \cite{woelki2005steady} and \cite{Matrin-batch-2009}.
Here we outline the specialization of the general
shift operator from \Cref{sub:shifting_operator} to this model. We also
observe that in contrast with the 
$q$-Hahn TASEP and its specializations,
in the Schur vertex model the shift operator \emph{does not} preserve the 
distinguished initial configuration $\mathbf{g}_{step}$.

The Schur vertex model scales to
a version of the TASEP in continuous inhomogeneous space 
\cite[Theorem 2.7]{SaenzKnizelPetrov2018}.
It would be interesting to see how the shift operators 
behave under this scaling, but 
we do not pursue this analysis here.

\subsection{Schur vertex model}
\label{sub:Schur_vertex_model_def}

The \emph{Schur vertex model}
depends on the parameters $u_i,s_i$
as in \eqref{eq:parameters_for_higher_spin_vertex_model}.
For simplicity, here we can take the parameters $s_i$ to be homogeneous, 
$s_i\equiv s\in(-1,0)$.
Denote $\nu=s^2\in(0,1)$ and $-su_i=a_i\ge0$.
In term of these parameters,
condition \eqref{eq:uniformly_bounded_propagation_condition}
means that the $a_i$'s should be uniformly bounded from above.

The transition probabilities in the Schur vertex
model are the $q=0$ specializations of
\eqref{eq:L1_weights_explicit}. They are given by
\begin{equation}
	\label{eq:Schur_vertex_model_weights}
	\begin{split}
		&L^{\mathrm{Schur}}_{a_i,\nu}(0,0;0,0)=1
		,
		\qquad 
		L^{\mathrm{Schur}}_{a_i,\nu}(0,1;0,1)=\frac{\nu+a_i}{1+a_i}
		,
		\qquad 
		L^{\mathrm{Schur}}_{a_i,\nu}(0,1;1,0)=\frac{1-\nu}{1+a_i};
		\\
		&
		L^{\mathrm{Schur}}_{a_i,\nu}(g,0;g,0)=
		L^{\mathrm{Schur}}_{a_i,\nu}(g,1;g+1,0)=
		\frac{1}{1+a_i},\qquad g\ge1;
		\\[4pt]
		&
		L^{\mathrm{Schur}}_{a_i,\nu}(g,0;g-1,1)=
		L^{\mathrm{Schur}}_{a_i,\nu}(g,1;g,1)=
		\frac{a_i}{1+a_i},\qquad g\ge1
		.
	\end{split}
\end{equation}
Throughout this section it is convenient to work in the vertex model state space
$\mathscr{G}$ (\Cref{def:state_spaces}).
We interpret $\mathbf{g}_i$ for each $i\in \mathbb{Z}_{\ge1}$ as 
the number of particles at location $i$, 
where multiple particles per site are allowed.
Let $T^{\mathrm{Schur}}_{\nu,\mathbf{a}}$
denote the Markov transition operator
for the Schur vertex model
acting in $\mathscr{G}$.

Let us describe the dynamics for the Markov operator $T^{\mathrm{Schur}}_{\nu,\mathbf{a}}$. At each time step, the stacks of particles are updated in parallel. First, 
each nonempty stack of particles $\mathbf{g}_i(t)>0$
emits a single particle with probability
$a_i / (1+a_i)$. Then, the emitted particle instantaneously 
travels to the right by a 
random distance $\min(\eta,k+1)$, where $\eta$ is a random variable 
in $\mathbb{Z}_{\ge1}$ with distribution
\begin{equation*}
	\mathbb{P}\left( \eta=j \right)
	=
	\frac{1-\nu}{1+a_{i+j}}
	\prod_{m=1}^{j-1}\frac{\nu+a_{i+m}}{1+a_{i+m}},
	\qquad 
	j\ge1,
\end{equation*}
and $k\ge0$ 
is the number of empty stacks after 
$\mathbf{g}_i(t)$, i.e.~$\mathbf{g}_{i+1}(t)=\ldots=\mathbf{g}_{i+k}(t)=0 $
and $\mathbf{g}_{i+k+1}(t)>0$.
If $\mathbf{g}_i(t)$ is the rightmost nonempty stack, then $k=+\infty$.

\subsection{Shift operator for the Schur vertex model}
\label{sub:Schur_vertex_model_comm_rel}

Let $B^{\mathrm{Schur}}_{\nu,\mathbf{a}}$ denote the Markov shift operator
(\Cref{def:B_operator})
for the Schur
vertex model.
It 
acts on the space $\mathscr{G}$ and 
involves the 
cross vertex weights 
$R^{\mathrm{bdry},(0)}_{a_1/a_0,-\sqrt\nu,-\sqrt\nu}$ 
\eqref{eq:R_circ_at_q0}
and
$R^{(0)}_{a_n/a_0,-\sqrt\nu,-\sqrt\nu}$ 
\eqref{eq:R_at_q0}
for $n\ge2$.
For the nonnegativity of these weights, the parameters must 
satisfy the conditions in \Cref{prop:R_at_q0_nonnegative},
which means 
\begin{equation}
	\label{eq:B_operator_uniform_conditions_on_parameters_Schur_case}
	2-\frac{1}{\nu}\le \frac{a_n}{a_0}\le 1,\qquad n\ge1.
\end{equation}
Note that the lower bound on $a_n/a_0$
is restrictive only for $\nu> \frac{1}{2}$.
The operator $B^{\mathrm{Schur}}_{\nu,\mathbf{a}}$ is well-defined
due to \Cref{lemma:B_operator_well_defined},
since the condition \eqref{eq:B_operator_uniform_conditions_on_parameters}
is automatic for our specialization of parameters.
The next statement readily follows from 
\Cref{thm:B_T_intertw_relation}:
\begin{proposition}
	\label{prop:B_Schur_shift}
	Let $\nu\in(0,1)$ and the parameters
	$a_n\ge0$ be uniformly bounded from above and satisfy
	\eqref{eq:B_operator_uniform_conditions_on_parameters_Schur_case}.
	Then
	\begin{equation}
		\label{eq:B_Schur_shift}
		T^{\mathrm{Schur}}_{\nu,\mathbf{a}}\ssp
		B^{\mathrm{Schur}}_{\nu,\mathbf{a}}
		=
		B^{\mathrm{Schur}}_{\nu,\mathbf{a}}\ssp
		T^{\mathrm{Schur}}_{\nu,\mathsf{sh}(\mathbf{a})},
	\end{equation}
	where $\mathsf{sh}$ is the shift of the sequence
	$\mathbf{a}=(a_0,a_1,a_2,\ldots )$ as in \eqref{eq:shifting_action}.
\end{proposition}

In contrast with the 
$q$-Hahn TASEP and its specializations 
considered in \Cref{sec:comm_rel_specializations}
and in the previous papers
\cite{PetrovSaenz2019backTASEP} and \cite{petrov2019qhahn},
the shift operator 
$B^{\mathrm{Schur}}_{\nu,\mathbf{a}}$,
for the Schur vertex model, 
\emph{does not preserve} the distinguished 
empty 
configuration $\mathbf{g}_{step}\in \mathscr{G}$:

\begin{proposition}
	\label{prop:B_Schur_0_at_step_nonabsorbing}
	Let the parameters $\{a_n \}$ satisfy 
	\eqref{eq:B_operator_uniform_conditions_on_parameters_Schur_case}.
	Then the
	action of $B^{\mathrm{Schur}}_{\nu,\mathbf{a}}$
	on $\mathbf{g}_{step}$ changes 
	$\mathbf{g}_{step}$ with positive probability.
\end{proposition}
\begin{proof}
	From \eqref{eq:R_circ_at_q0}--\eqref{eq:R_circ_at_q0_notation}
	we have
	\begin{equation}
		\label{eq:R_bdry_0_at_step_nonabsorbing}
		R^{\mathrm{bdry},(0)}_{z,-\sqrt\nu,-\sqrt\nu}
		(0,j)=\frac{\nu^j(1-\nu)(1-z\mathbf{1}_{j>0})}{1-\nu z},
		\qquad j\in \mathbb{Z}_{\ge0}.
	\end{equation}
	This means that applying the first operator
	$P^{(1)}$ (see
	\eqref{eq:B_operator})
	to $\mathbf{g}_{step}$
	introduces a random number of paths 
	according to the distribution \eqref{eq:R_bdry_0_at_step_nonabsorbing}
	with $z=a_1/a_0$.
	These paths do not disappear after the application of the further 
	operators $P^{(2)},P^{(3)},\ldots $ in 
	\eqref{eq:B_operator} due to path conservation.
	Moreover, from \eqref{eq:R_at_q0_notation}--\eqref{eq:R_at_q0}
	we have
	\begin{equation}
		\label{eq:R_0_at_step_nonabsorbing}
		R^{(0)}_{z,-\sqrt\nu,-\sqrt\nu}(i,j;0,j-i)
		=\frac{\nu^{j-i}(1-\nu\mathbf{1}_{i>0})(1-z\mathbf{1}_{i<j})}
		{1-\nu z\mathbf{1}_{j>0}},
		\qquad i\in {0,1,\ldots,j },
	\end{equation}
	where $z=a_n/a_0$ for $n\geq 2$.
	This implies that the operator
	$B^{\mathrm{Schur}}_{\nu,\mathbf{a}}$
	indeed does not preserve the distinguished empty configuration
	$\mathbf{g}_{step}$.
\end{proof}

\newpage
\part{Bijectivisation and Rewriting History}

In the second part, we describe how the 
intertwining relations obtained in the first part lead to 
couplings between trajectories of the stochastic vertex model
(and the corresponding exclusion process)
with different sequences of parameters. 
The passage from intertwining relations to
couplings, a ``bijectivisation'', 
is by now a well-known technique that 
originated in \cite{DiaconisFill1990}
and was later
developed in the context of integrable stochastic particle systems
in 
\cite{BorFerr2008DF}, 
\cite{BorodinGorin2008}, 
\cite{BorodinPetrov2013NN}, 
\cite{BufetovPetrovYB2017}.
Here, we apply a bijectivisation in a new setting leading to couplings of probability measures on trajectories
under time evolution of integrable stochastic systems.

\section[Bijectivisation and coupling of trajectories. General constructions]{Bijectivisation and coupling of trajectories.\\General constructions}
\label{sec:bijectivisation}

In this section, we return to the general
setup of the fused stochastic higher spin six vertex model
as in \Cref{sec:stoch_vert_models,sec:YBE,sec:comm_rel}.
We construct couplings
between measures on trajectories of the stochastic higher spin six vertex model
with different sequences of parameters
by applying a bijectivisation
\cite{BufetovPetrovYB2017}
(also called a ``probabilistic bijection'', e.g., see \cite{aigner2020_Macdonald_RSK})
to the Yang-Baxter equation and iterating it.
This section focuses on a general discussion which does not
rely on any particular choice of a bijectivisation of the Yang-Baxter equation. 
In further sections, we consider 
the simplest, i.e.~\emph{independent}, bijectivisation
in a subfamily of vertex models.
This subfamily is still quite general and, in particular, 
includes
$q$-TASEP and TASEP.

\subsection{Bijectivisation of summation identities}
\label{sub:bij_basics}

We begin by recalling the basic notion of a bijectivisation for a summation identity with finitely many terms, see \cite[Section 2]{BufetovPetrovYB2017}. Let $A, B$ be two disjoint finite sets. Also, introduce a positive weight function $w(x)$ for each element $x \in A \cup B$ so that the following identity holds:
\begin{equation}
	\label{eq:bijectivisation_of_identity}
	\sum_{a\in A}w(a)
	=
	\sum_{b\in B}w(b).
\end{equation}
In particular, identity \eqref{eq:bijectivisation_of_identity} defines probability distributions on the sets $A$ and $B$ with probability weights proportional to $\{w(a)\}_{a\in A}$ and $\{w(b)\}_{b\in B}$, respectively. A bijectivisation is a coupling between these two probability distributions, expressed via conditional probabilities. 

More precisely, a \emph{bijectivisation} is a family of forward and backward transition probabilities $p^{\mathrm{fwd}}(a \rightarrow b)\ge0$, $p^{\mathrm{bwd}}(b \rightarrow a)\ge0$, for $a\in A$, $b\in B$, satisfying the following 
\emph{stochasticity} and \emph{detailed balance equation}:
\begin{equation}
	\label{eq:basic_bij_properties}
	\begin{split}
		&\sum_{b \in B} p^{\mathrm{fwd}}(a \rightarrow b) = 1 \quad   \forall a \in A,
		\qquad 
		\qquad 
		\sum_{a \in A} p^{\mathrm{bwd}}(b \rightarrow a) = 1 \quad  \forall b \in B .
		\\
		&\hspace{35pt}
		w(a)\, p^{\mathrm{fwd}}(a \rightarrow b) = 
		w(b)\, p^{\mathrm{bwd}}(b \rightarrow a),
		\qquad 
		\forall a\in A,\ b\in B.
	\end{split}
\end{equation}

For general sets $A$ and
$B$, a bijectivisation exists and it is not unique.
However, in the special case when the cardinality of the sets $A$ or $B$, i.e.~$|A|$ or $|B|$, is equal to $1$,
a bijectivisation is unique. For instance, when
$|A|=1$ and $A=\left\{ a_0 \right\}$, we have
\begin{equation*}
	p^{\mathrm{fwd}}(a_0 \rightarrow b)=\frac{w(b)}{w(a_0)},
	\qquad 
	p^{\mathrm{bwd}}(b \rightarrow a_0)=1, \qquad 
	\forall b\in B.
\end{equation*}
In the case when $|A|=|B|=2$, the dimension of the space of all possible
solutions to the linear equations
\eqref{eq:basic_bij_properties}
is equal to one, meaning that there is a one-parameter family of bijectivizations for this case.

\subsection{Bijectivisation of the vertical Yang-Baxter equation}
\label{sub:YBE_bijectivisation}

Let us determine a bijectivisation to the vertical Yang-Baxter
equation from \Cref{prop:YBE}. 
The equation depends on four parameters $u_1,u_2\ge0$ and $s_1,s_2\in(-1,0]$.
Recall that the path conservation implies that the sums in both sides of the 
Yang-Baxter equation \eqref{eq:YBE_for_LJ} are actually finite.
Additionally, all terms in the sums for the Yang-Baxter equation are nonnegative
if $(u_2/u_1,s_1,s_2)\in \mathcal{R}$; see \Cref{def:parameter_range_R}.
For fixed
$i_1,j_1 \in \left\{ 0,1,\ldots,J  \right\}$,
$i_2,i_3,j_2,j_3\in \mathbb{Z}_{\ge0}$,
we denote the terms on the left and right side of the Yang-Baxter equation, respectively, by the following weight functions:
\begin{equation}
	\label{eq:YBE_terms_w_LHS_RHS}
	\begin{split}
		&
		w_{i_1,j_1}^{\mathrm{LHS}}(k_2,k_3\mid i_2,i_3;j_2,j_3)=
		R_{\frac{u_2}{u_1},s_1,s_2}(j_3,k_2;k_3,j_2)
		\ssp
		L^{(J)}_{u_1,s_1}(i_2,i_1;k_2,k_1)
		\ssp
		L^{(J)}_{u_2,s_2}(i_3,k_1;k_3,j_1),
		\\&
		w_{i_1,j_1}^{\mathrm{RHS}}(k_3',k_2'\mid i_2,i_3;j_2,j_3)=
		L^{(J)}_{u_2,s_2}(k_3',i_1;j_3,k_1')\ssp
		L^{(J)}_{u_1,s_1}(k_2',k_1';j_2,j_1)\ssp
		R_{\frac{u_2}{u_1},s_1,s_2}(k_3',i_2;i_3,k_2')
		,
	\end{split}
\end{equation}
where $k_1,k_1'$ are omitted in the notation of the weight functions since they may be determined 
through the path conservation:
\begin{equation*}
	k_1=i_1+i_2-k_2,
	\qquad 
	k_1'=j_1+j_2-k_2'.
\end{equation*}

Throughout the current \Cref{sec:bijectivisation},
we denote any choice of transition probabilities
coming from a bijectivisation of the
Yang-Baxter equation \eqref{eq:YBE_for_LJ}
by 
\begin{equation}
	\label{eq:bij_p_up_p_down}
	p^{\downarrow}_{i_1,j_1}[(k_2,k_3)\to
	(k_3',k_2')\mid i_2,i_3,j_2,j_3]
	\qquad \textnormal{and}\qquad 
	p^{\uparrow}_{i_1,j_1}[(k_3',k_2')\to
	(k_2,k_3)\mid i_2,i_3,j_2,j_3].
\end{equation}
Here, the down and up arrows indicate the direction in which the 
cross vertex is moved.
See \Cref{fig:YBE_bijectivisation} for an illustration.

\begin{figure}[htb]
	\centering
	\includegraphics[width=.89\textwidth]{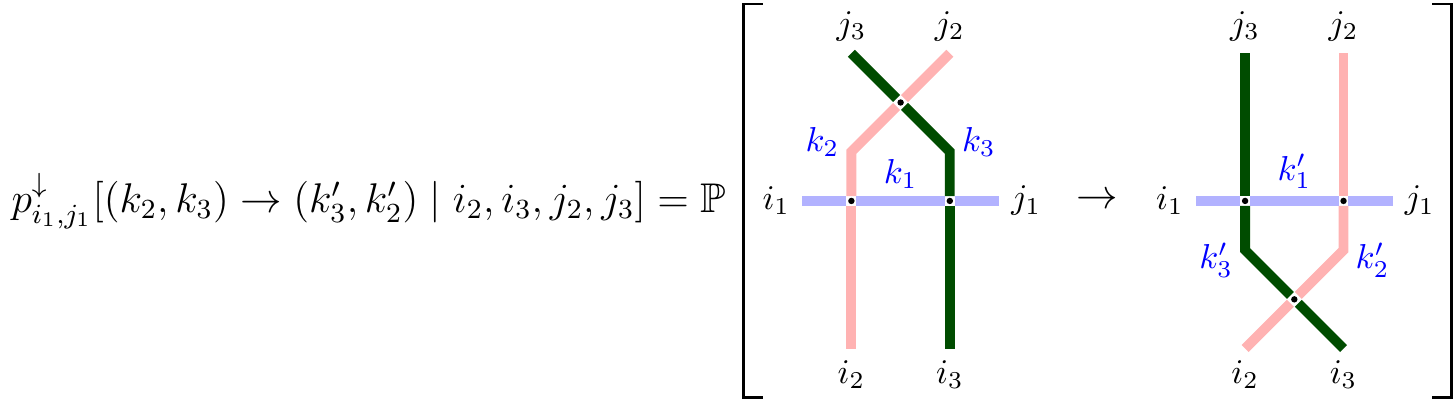}
	\caption{Graphical illustration of the 
	down transition probability coming from 
	a bijectivisation of the Yang-Baxter equation.
	The up transition is similar, with the cross vertex moving upwards instead.}
	\label{fig:YBE_bijectivisation}
\end{figure}

\begin{remark}[Bijectivisation with infinitely many paths]
	\label{rmk:left_boundary_up_bijectivisation}
	We tacitly included the case of infinitely many paths
	$i_2=j_3=+\infty$
	(arising at the left boundary of the stochastic
	higher spin six vertex model, see \Cref{sub:R_infinite_limit}) 
	into the notation \eqref{eq:bij_p_up_p_down}.
	In this case, the range 
	of the down transition probability
	(i.e.~the set of possible values of $(\infty,k_2')$)
	is always finite.
	However, the range of the up transition probability
	\begin{equation*}
		p^{\uparrow}_{*,j_1}[(\infty,k_2')\to
		(\infty,k_3)\mid \infty, i_3, j_2, \infty]
	\end{equation*}
	(``$*$'' means that there is no dependence on $i_1$)
	may be infinite, since both $k_1$ and $k_3$ may be arbitrarily 
	large, provided that $i_3+k_1=k_3+j_1$, because the stochastic higher
	spin six vertex model allows for an unbounded
	number of paths per horizontal edge by letting $q^J$ be an independent parameter
	with $J\notin \mathbb{Z}_{\ge1}$.

	Throughout the rest of this section, we
	assume that a well-defined bijectivisation
	at the left boundary
	exists. 
	In considering explicit bijectivisations for the present paper,
	we restrict our attention to models with $J=1$
	and, thus, the issue of an infinite range of the up transition
	probability does not arise.
\end{remark}

In the current \Cref{sec:bijectivisation}, 
we explain the general framework 
for a bijectivisation and couplings of measures on trajectories of stochastic vertex models.
We do not pursue an explicit computation of possible transition probabilities 
$p^{\downarrow}_{i_1,j_1}$
and 
$p^{\uparrow}_{i_1,j_1}$
in the fully general case when all three vertex weights entering the 
Yang-Baxter equation have a $q$-hypergeometric form.
Below in \Cref{sec:discrete_time_bijectivisation}, we focus on a, still rather general,
subfamily of vertex models
for which the cross vertex weights
factorize and become $q$-beta-binomial as in \Cref{sub:R_Hahn_specialization}.
Moreover, we set $J=1$ in the weights $L^{(J)}_{u_i,s_i}$, 
which forces $i_1,j_1$ to be either zero or one. 
For this subfamily, an explicit treatment of a bijectivisation is accessible.

\subsection{Down and up transitions on vertex model configurations}
\label{sub:up_down_on_configurations}

Recall the space $\mathscr{G}$ 
whose elements encode vertical paths crossing a given
horizontal slice in the stochastic higher spin six vertex model,
see \Cref{sub:particle_systems}.
Recall the transition operator of the stochastic vertex 
model $T_{\mathbf{u},\mathbf{s}}$ (\Cref{sub:particle_systems})
and
the swap operator $P^{(n)}_{z,s_1,s_2}$
(\Cref{sub:swapping_operator}).
These operators satisfy a (quasi-)computation relation
(\Cref{prop:Pn_T_intertw})
which follows from the Yang-Baxter equation.
Here we employ bijectivisation of this intertwining relation
to 
define up and down transitions 
on vertex model configurations.

Fix
$n\in \mathbb{Z}_{\ge1}$ and abbreviate throughout the rest of this section:
\begin{equation}
	\label{eq:comm_bij_abbreviation_Pn}
	P^{(n)}=P^{(n)}_{u_n/u_{n-1},s_{n-1},s_n},
	\qquad 
	T=T_{\mathbf{u},\mathbf{s}},
	\qquad 
	T_{\sigma}=
	T_{\sigma\mathbf{u},\sigma\mathbf{s}},
\end{equation}
where $\sigma$ is an arbitrary permutation from the 
infinite symmetric group (that is, $\sigma$ acts on $\mathbb{Z}_{\ge0}$ by 
fixing all but finitely many points).
The intertwining relation from \Cref{prop:Pn_T_intertw}
is
\begin{equation}
	\label{eq:TP_PT_comm_rel_simpler}
	TP^{(n)}=P^{(n)}T_{\sigma_{n-1}},
\end{equation}
see \Cref{fig:comm_rel}, left, for an illustration.
Here $\sigma_{n-1}=(n-1,n)$ is an elementary transposition.
Recall that we are writing products of Markov operators
as acting on measures, cf. \Cref{rmk:product_of_Markov_operators}.

Relation \eqref{eq:TP_PT_comm_rel_simpler} follows from a single Yang-Baxter
equation illustrated in \Cref{fig:Pn_T_intertw}. 
Note that all the terms of this Yang-Baxter equation are nonnegative
if we assume that the parameters $\mathbf{u},\mathbf{s}$
of the operators \eqref{eq:comm_bij_abbreviation_Pn} 
satisfy the conditions of \Cref{prop:Pn_T_intertw}.
Taking a bijectivisation \eqref{eq:bij_p_up_p_down}
of this Yang-Baxter equation, we arrive at the following down and up
Markov operators.

\begin{figure}[htb]
	\centering
	\includegraphics[width=.8\textwidth]{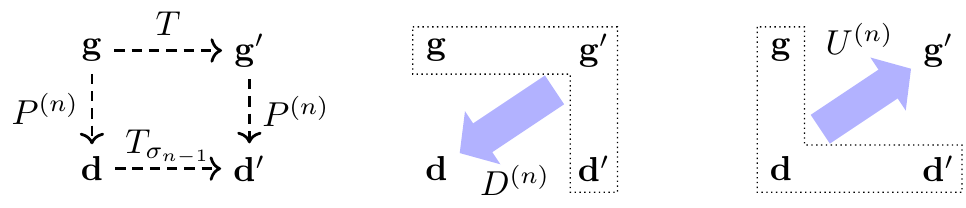}
	\caption{Left: A commuting diagram of Markov operators,
		where $\mathbf{g},\mathbf{g}',\mathbf{d},\mathbf{d}'\in \mathscr{G}$
		with the notation from \eqref{eq:comm_bij_abbreviation_Pn}.
		The element $\mathbf{g}$ is fixed, and all other elements are random. 
		Intertwining means that the distributions of
		$\mathbf{d}'$ obtained along both paths (right-down and down-right)
		coincide.
		Center and right: Markov operators $D^{(n)}$
		and $U^{(n)}$ constructed from bijectivisation.}
	\label{fig:comm_rel}
\end{figure}

\begin{definition}[Down Markov operators for swaps]
	\label{def:Dn_Markov_operator_from_bij}
	Fix $n\in \mathbb{Z}_{\ge1}$ and
	$\mathbf{g},\mathbf{g}',\mathbf{d}'\in \mathscr{G}$
	such that $T(\mathbf{g},\mathbf{g}')\ssp P^{(n)}(\mathbf{g}',\mathbf{d}')\ne 0$.
	Define a random 
	element $\mathbf{d}\in \mathscr{G}$ 
	such that
	$d_l=g_l$ for all $l\ne n-1,n$,
	and
	$d_{n-1},d_n$ are random and chosen from the distribution
	\begin{equation*}
		p^{\downarrow}_{i_1,j_1}[(g_{n-1}',g_n')\to
		(d_{n-1},d_n)\mid g_{n-1},g_n,d_{n-1}',d_n'],
	\end{equation*}
	with $i_1,j_1$ given by the numbers of horizontal paths determined from the 
	configurations of vertical paths
	\begin{equation}
		\label{eq:down_up_Markov_operators_definition_of_i1_j1}
		\textstyle i_1 =
		\sum_{l \ge n-1}d_l'
		-
		\sum_{l \ge n-1}g_l
		,
		\qquad 
		j_1 =
		\sum_{l \ge n+1}d_l'
		-
		\sum_{l \ge n+1}g_l.
	\end{equation}
	Let $D^{(n)}(\mathbf{g}'\to \mathbf{d}\mid \mathbf{g},\mathbf{d}')$
	denote the probability weight of $\mathbf{d}$,
	and call $D^{(n)}=D^{(n)}_{u_{n-1},s_{n-1};u_n,s_n}$ 
	the \emph{down Markov operator} corresponding
	to the swap operator $P^{(n)}$ at sites $n-1,n$. 
	See \Cref{fig:comm_rel}, center, for an illustration.
\end{definition}

\begin{definition}[Up Markov operators for swaps]
	\label{def:Un_Markov_operator_from_bij}
	Fix $n\in \mathbb{Z}_{\ge1}$ 
	and 
	$\mathbf{g},\mathbf{d},\mathbf{d}'\in \mathscr{G}$
	such that $P^{(n)}(\mathbf{g},\mathbf{d})
		\ssp
		T_{\sigma_{n-1}}(\mathbf{d},\mathbf{d}')\ne 0$.
	Define a random element $\mathbf{g}'\in \mathscr{G}$
	such that $g'_l=d'_l$ for all $l\ne n-1,n$,
	and $g'_{n-1},g'_n$ are random and chosen from the distribution
	\begin{equation*}
		p^{\uparrow}_{i_1,j_1}
		[(d_{n-1},d_n)\to
		(g'_{n-1},g'_n)\mid g_{n-1},g_n,d_{n-1}',d_n'],
	\end{equation*}
	with $i_1,j_1$ given by \eqref{eq:down_up_Markov_operators_definition_of_i1_j1}.
	Let $U^{(n)}(\mathbf{d}\to\mathbf{g}'\mid \mathbf{g},\mathbf{d}')$
	denote
	the probability weight of $\mathbf{g}'$,
	and call $U^{(n)}=U^{(n)}_{u_{n-1},s_{n-1};u_n,s_n}$ the \emph{up Markov operator}
	corresponding to $P^{(n)}$.
	See \Cref{fig:comm_rel}, right, for an illustration.
\end{definition}

The operators $D^{(n)}, U^{(n)}$
depend
on the parameters $u_{n-1},s_{n-1},u_n,s_n$
(which we often omit from the notation)
and on the choice of bijectivisation
which typically is not unique. 
For any choice of bijectivisation, 
the down and up operators
satisfy the stochasticity
\begin{equation*}
	\sum_{\mathbf{d}\in \mathscr{G}}D^{(n)}(\mathbf{g}'\to \mathbf{d}\mid
	\mathbf{g},\mathbf{d}')=1 \quad \forall \mathbf{g},\mathbf{g}',\mathbf{d'};
	\qquad 
	\sum_{\mathbf{g}'\in \mathscr{G}}U^{(n)}(\mathbf{d}\to \mathbf{g}'\mid
	\mathbf{g},\mathbf{d}')=1 \quad \forall \mathbf{g},\mathbf{d},\mathbf{d'},
\end{equation*}
and the detailed balance equation
\begin{equation}\label{eq:Dn_Un_detailed_balance_equation}
	T(\mathbf{g},\mathbf{g}')
	\ssp
	P^{(n)}(\mathbf{g}',\mathbf{d}')
	\ssp
	D^{(n)}(\mathbf{g}'\to \mathbf{d}\mid \mathbf{g},\mathbf{d}')
	=
	P^{(n)}(\mathbf{g},\mathbf{d})
	\ssp
	T_{\sigma_{n-1}}(\mathbf{d},\mathbf{d}')
	\ssp
	U^{(n)}(\mathbf{d}\to\mathbf{g}'\mid \mathbf{g},\mathbf{d}')
\end{equation}
for any quadruple $\mathbf{g},\mathbf{g}',\mathbf{d},\mathbf{d}'\in \mathscr{G}$.
Note that when, say, 
$T(\mathbf{g},\mathbf{g}')
\ssp
P^{(n)}(\mathbf{g}',\mathbf{d}')=0$
(in contradiction with the assumption in \Cref{def:Dn_Markov_operator_from_bij}),
the value of $D^{(n)}$ is irrelevant in \eqref{eq:Dn_Un_detailed_balance_equation}
and, additionally, the corresponding value of
$U^{(n)}$ must be zero to satisfy the detailed balance.
Moreover, observe that summing \eqref{eq:Dn_Un_detailed_balance_equation}
over $\mathbf{d}$ and $\mathbf{g}'$
results in the intertwining relation
\eqref{eq:TP_PT_comm_rel_simpler}.

\subsection{Down and up transitions related to the Markov shift operator}
\label{sub:bij_of_shift_operator}

Throughout the rest of this section we continue to use 
abbreviations \eqref{eq:comm_bij_abbreviation_Pn},
and also introduce the following abbreviations
\begin{equation}
	\label{eq:comm_bij_abbreviation_Pn0_further}
	P^{(0,n)}=
	P^{(n)}_{u_n/u_{0},s_{0},s_n}
	,
	\qquad 
	B=B_{\mathbf{u},\mathbf{s}}
	,
	\qquad 
	T_{\mathsf{sh}}=T_{\mathsf{sh}(\mathbf{u}),\mathsf{sh}(\mathbf{s})},
\end{equation}
where $\mathsf{sh}$ is the shift \eqref{eq:shifting_action}.
Recall that the Markov shift operator $B$ is obtained by 
iterating the swap operators
$P^{(0,n)}$ over all $n\in \mathbb{Z}_{\ge1}$, see
\eqref{eq:B_operator}.
Iterating the down or up operators
in a similar manner would result in Markov operators on $\mathscr{G}$
denoted by 
$D^{\bullet}$ and $U^{\bullet}$
which satisfy the following detailed balance equation:
\begin{equation}
	\label{eq:shift_ops_detailed_balance_equation}
	T(\mathbf{g},\mathbf{g}')
	\ssp
	B(\mathbf{g}',\mathbf{d}')
	\ssp
	D^{\bullet}(\mathbf{g}'\to \mathbf{d}\mid \mathbf{g},\mathbf{d}')
	=
	B(\mathbf{g},\mathbf{d})
	\ssp
	T_{\mathsf{sh}}(\mathbf{d},\mathbf{d}')
	\ssp
	U^{\bullet}(\mathbf{d}\to\mathbf{g}'\mid \mathbf{g},\mathbf{d}')
\end{equation}
for any 
$\mathbf{g},\mathbf{g}',\mathbf{d},\mathbf{d}'\in \mathscr{G}$.
Graphically, one can extract the definition
of $D^{\bullet}$ and $U^{\bullet}$ from the tower of intertwining
relations in \Cref{fig:comm_rel_shifting}.
However, to describe these operators in full detail we need some notation
and observations.

\begin{figure}[htb]
	\centering
	\includegraphics[width=.45\textwidth]{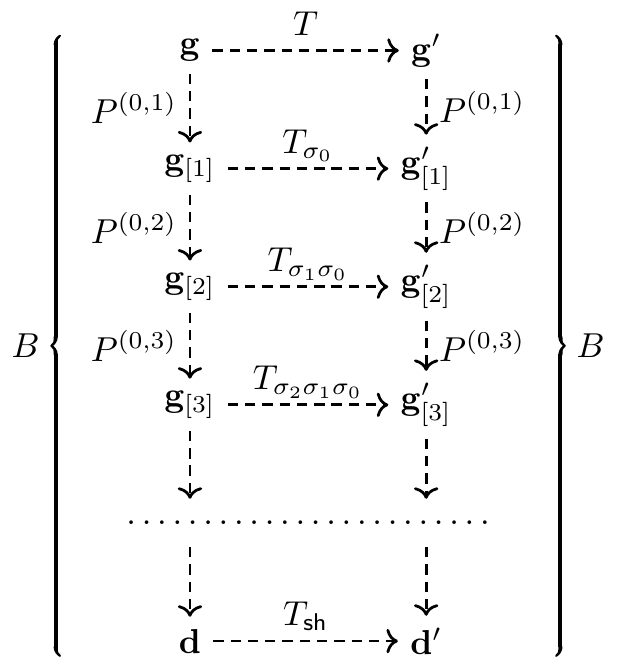}
	\caption{
		Relation $TB=BT_{\mathsf{sh}}$
		is a consequence of the tower of intertwining relations displayed
		in the figure. 
		The random configuration $\mathbf{d}\in \mathscr{G}$ 
		distributed according to $D^{\bullet}(\mathbf{g}'\to \mathbf{d}\mid 
		\mathbf{g},\mathbf{d}')$ is constructed as follows.
		Fix $\mathbf{g},\mathbf{g}',\mathbf{d}'\in \mathscr{G}$. 
		This completely determines
		$\mathbf{g}'_{[m]}$
		for all $m\ge1$,
		see \eqref{eq:h_m_for_tower}--\eqref{eq:g_m_for_tower}. 
		Apply the down Markov operator for swaps
		to each square of the tower from the top of the diagram to the bottom; see \Cref{def:Dn_Markov_operator_from_bij}.
		First, use 
		$D^{(0,1)}$ 
		to sample $\mathbf{g}_{[1]}$ given $\mathbf{g},\mathbf{g}',\mathbf{g}'_{[1]}$,
		and continue consecutively using $D^{(0,m)}$ to sample
		to sample $\mathbf{g}_{[m]}$ given $\mathbf{g}_{[m-1]},\mathbf{g}'_{[m]},\mathbf{g}'_{[m+1]}$ for $m\geq2$. Once the update
		terminates, we get the desired random element $\mathbf{d}\in \mathscr{G}$. The fact that the sequence terminates is due to \Cref{lemma:B_update_terminates}.
		The up operator $U^{\bullet}$ is defined similarly
		by applying the up Markov operator for swaps to each square of the tower from
		bottom of the diagram to the top; see \Cref{def:Un_Markov_operator_from_bij}}
	\label{fig:comm_rel_shifting}
\end{figure}

Denote
\begin{equation}
	\label{eq:D0n_U0n_abbreviation_third_time}
	D^{(0,n)}=D^{(n)}_{u_0,s_0;u_n,s_n},
	\qquad 
	U^{(0,n)}=U^{(n)}_{u_0,s_0;u_n,s_n}.
\end{equation}
That is, $D^{(0,n)}$ randomly changes $(g_{n-1}',g_n')$ to $(d_{n-1},d_n)$,
but uses the parameters
$(u_0,s_0)$ and $(u_n,s_n)$
instead of the ones in \Cref{def:Dn_Markov_operator_from_bij},
and similarly for $U^{(0,n)}$.
Let the parameters
of the vertex model
satisfy 
$(\mathbf{u},\mathbf{s})\in \mathcal{T}\cap \mathcal{B}$,
so that the operators $T,T_{\mathsf{sh}}$, and $B$ are well-defined
(see \eqref{eq:parameters_for_higher_spin_vertex_model}--\eqref{eq:uniformly_bounded_propagation_condition} 
and \Cref{def:B_operator_uniform_conditions_on_parameters}).
Moreover, assume that the up transition 
$U^{(0,1)}$ at the left boundary is also well-defined,
cf.
\Cref{rmk:left_boundary_up_bijectivisation}.

Now, let us encode the path configurations 
at intermediate horizontal slices in
the vertex model for
$B(\mathbf{g},\mathbf{d})$
given in 
\Cref{fig:B_operator}.
For $\mathbf{g},\mathbf{d}\in \mathscr{G}$ with $B(\mathbf{g},\mathbf{d})\ne 0$, 
we denote the $m^{th}$ horizontal slice 
by $\mathbf{g}_{[m]}$ and it is given by
\begin{equation}
	\label{eq:g_m_for_tower}
	(\mathbf{g}_{[m]})_l=\begin{cases}
		d_l,&l<m;\\
		h_m,&l=m;\\
		g_l,&l>m;
	\end{cases}
\end{equation}
for $m\in \mathbb{Z}_{\ge0}$ so that
the number of vertical arrows at the $m^{th}$ position is given by
\begin{equation}
	\label{eq:h_m_for_tower}
	h_m\coloneqq \sum_{l\ge m}d_l-\sum_{l\ge m+1}g_l,\qquad m\in \mathbb{Z}_{\ge1}.
\end{equation}
Note that $h_m=0$
for all sufficiently large $m$
since $\mathbf{d},\mathbf{g}\in \mathscr{G}$.
Also, note that $\mathbf{g}_{[0]}$ coincides with $\mathbf{g}$.

\begin{lemma}
	\label{lemma:B_update_terminates}
	Let $\mathbf{g},\mathbf{d}\in \mathscr{G}$ and
	$M>1+\max\left\{ l\in \mathbb{Z}_{\ge1}\colon g_l>0 \ \textnormal{or}\ d_l>0 \right\}$.
	Then
	\begin{equation}
		\label{eq:B_as_a_tower_of_P0n_operators}
		B(\mathbf{g},\mathbf{d})
		=
		\prod_{m=1}^{M}
		P^{(0,m)}(\mathbf{g}_{[m-1]},\mathbf{g}_{[m]})
	\end{equation}
	and $\mathbf{d}=\mathbf{g}_{[M]}$.
\end{lemma}
\Cref{lemma:B_operator_well_defined}
essentially
shows that for fixed $\mathbf{g}$, 
the sum of \eqref{eq:B_as_a_tower_of_P0n_operators}
over all $\mathbf{d}$ is equal to $1$.
Since
$M$ depends
on $\mathbf{d}$ in \eqref{eq:B_as_a_tower_of_P0n_operators},
\Cref{lemma:B_update_terminates} does not imply
\Cref{lemma:B_operator_well_defined}.
\begin{proof}[Proof of \Cref{lemma:B_update_terminates}]
	For all $m\ge M$,
	we have $h_m=0$ by the lower bound on $M$.
	Hence,
	$\mathbf{g}_{[m-1]}=\mathbf{g}_{[m]}$ for $m\geq M$.
	In particular, for $m\ge M$,
	we have
	$P^{(0,m)}(\mathbf{g}_{[m-1]},\mathbf{g}_{[m]})
	=
	\mathbf{1}_{\mathbf{g}_{[m]}=\mathbf{g}_{[m-1]}}$
	since the path configuration, as
	in \Cref{fig:B_operator},
	is empty to the right of location $M$.
	Thus, we may truncate the infinite product of 
	swap operators $P^{(0,m)}$ that define
	the operator $B$, see \eqref{eq:B_operator}, 
	to the product of the swap operators $P^{(0,m)}$ 
	in \eqref{eq:B_as_a_tower_of_P0n_operators}.
	Additionally, note that the
	$\mathbf{g}_{[m]}$'s stabilize to $\mathbf{d}$.
	This completes the proof.
\end{proof}

For $\mathbf{g}',\mathbf{d}'$ with
$B(\mathbf{g}',\mathbf{d}')\ne 0$
and $m\in \mathbb{Z}_{\ge0}$,
define $\mathbf{g}'_{[m]}$ 
in the same way as in
\eqref{eq:h_m_for_tower}--\eqref{eq:g_m_for_tower}.
The construction of the tower of intertwining relations 
given
in \Cref{fig:comm_rel_shifting}
follows from 
the representation \eqref{eq:B_as_a_tower_of_P0n_operators} of the shift operators
together with \eqref{eq:TP_PT_comm_rel_simpler}.
Note, in particular, that the tower is finite since the $\mathbf{g}_{[m]}$ and $\mathbf{g}'_{[m]}$ stabilize as is shown below.

\begin{lemma}
	\label{lemma:bijectivisation_for_large_M}
	Let $\mathbf{g},\mathbf{g}',\mathbf{d},\mathbf{d}'\in \mathscr{G}$,
	and $M>1+\max\left\{ l\in \mathbb{Z}_{\ge1}\colon
	\max (g_l,g_l',d_l,d_l')>0 \right\}$.
	Then, for any $m\ge M$, we have
	\begin{equation}
		\label{eq:bijectivisation_large_M_1}
		T_{\sigma_{m-2}\sigma_{m-3}\ldots\sigma_1\sigma_0 }
		(\mathbf{g}_{[m-1]},\mathbf{g}'_{[m-1]})=
		T_{\sigma_{m-1}\sigma_{m-2}\sigma_{m-3}\ldots\sigma_1\sigma_0 }
		(\mathbf{g}_{[m]},\mathbf{g}'_{[m]}),
	\end{equation}
	and
	\begin{equation}
		\label{eq:bijectivisation_large_M_2}
		\begin{split}
			D^{(0,m)}(\mathbf{g}'_{[m-1]}\to \mathbf{g}_{[m]}\mid 
			\mathbf{g}_{[m-1]},\mathbf{g}_{[m]}')
			&=
			\mathbf{1}_{\mathbf{g}_{[m]}=\mathbf{g}_{[m-1]}}
			,
			\\
			U^{(0,m)}(\mathbf{g}_{[m]}\to \mathbf{g}'_{[m-1]}\mid 
			\mathbf{g}_{[m-1]},\mathbf{g}_{[m]}')&=
			\mathbf{1}_{\mathbf{g}'_{[m-1]}=\mathbf{g}_{[m]}'}.
		\end{split}
	\end{equation}
\end{lemma}
\begin{proof}
	Observe that the
	transfer matrices 
	$T_{\sigma_{m-2}\ldots\sigma_1\sigma_0 }$
	and
	$T_{\sigma_{m-1}\sigma_{m-2}\ldots\sigma_1\sigma_0 }$
	differ only by the location of the parameter $s_0$. Moreover, for $m\ge M$, the action of the these transfer matrices on the configuration $\mathbf{g}_{[m]}$ does not depend on $s_0$ since the configuration is empty to the right of $M$. Therefore, the action is the same. This proves \eqref{eq:bijectivisation_large_M_1}.
	
	Next, notice that 
	$P^{(0,m)}$ acts as identity 
	on our elements
	for $m\ge M$, see the proof of \Cref{lemma:B_update_terminates}.
	Then, along with \eqref{eq:bijectivisation_large_M_1},
	this implies that the detailed balance equation for $D^{(0,m)},U^{(0,m)}$
	has a unique solution given by \eqref{eq:bijectivisation_large_M_2}.
	This completes the proof.
\end{proof}

\begin{definition}
	[Down operator for shift]
	\label{def:D_bullet_operator}
	Let $\mathbf{g},\mathbf{g}',\mathbf{d}'\in \mathscr{G}$
	be such that $T(\mathbf{g},\mathbf{g}')B(\mathbf{g}',\mathbf{d}')\ne 0$.
	The \emph{down Markov operator} corresponding to the shift operator $B$ is defined as follows:
	\begin{equation}
		\label{eq:D_bullet_operator}
		D^{\bullet}(\mathbf{g}'\to \mathbf{d}\mid \mathbf{g},\mathbf{d}')
		\coloneqq
		\prod_{m=1}^{\infty}
		D^{(0,m)}(\mathbf{g}'_{[m-1]}\to \mathbf{g}_{[m]}\mid 
		\mathbf{g}_{[m-1]},\mathbf{g}_{[m]}'),
		\qquad \mathbf{d}=\lim_{m\to+\infty}\mathbf{g}_{[m]}\in \mathscr{G}.
	\end{equation}
	Due to \Cref{lemma:bijectivisation_for_large_M},
	the product is actually finite and the limit stabilizes. See \Cref{fig:comm_rel,fig:comm_rel_shifting} for an illustration
\end{definition}
Note that for any $\mathbf{g},\mathbf{g}',\mathbf{d}'\in \mathscr{G}$,
there are only finitely many $\mathbf{d}\in \mathscr{G}$
for which \eqref{eq:D_bullet_operator} is nonzero.
This is due to the fact that there are only finitely many $\mathbf{d}$
for which $T_{\mathsf{sh}}(\mathbf{d},\mathbf{d}')\ne 0$, see
the desired detailed balance equation \eqref{eq:shift_ops_detailed_balance_equation}.

\begin{definition}
	[Up operator for shift]
	\label{def:U_bullet_operator}
	Let $\mathbf{g},\mathbf{d},\mathbf{d}'\in \mathscr{G}$
	be such that $B(\mathbf{g},\mathbf{d}) T_\mathsf{sh}(\mathbf{d},\mathbf{d}')\ne 0$.
	The \emph{up Markov operator} corresponding to the shift operator $B$ is defined as follows:
	\begin{equation}
		\label{eq:U_bullet_operator}
		U^{\bullet}(\mathbf{d}\to \mathbf{g}'\mid \mathbf{g},\mathbf{d}')
		\coloneqq
		\prod_{m=1}^{\infty}
		U^{(0,m)}(\mathbf{g}_{[m]}\to \mathbf{g}'_{[m-1]}\mid 
		\mathbf{g}_{[m-1]},\mathbf{g}_{[m]}'),
		\qquad \mathbf{g}'=\mathbf{g}'_{[0]}\in \mathscr{G}.
	\end{equation}
	Due to \Cref{lemma:bijectivisation_for_large_M},
	the product is actually finite. See \Cref{fig:comm_rel,fig:comm_rel_shifting} for an illustration.
\end{definition}

We assume that the bijectivisation at the left boundary is well-defined, so that the 
whole up operator $U^{\bullet}$ is also well-defined.
In contrast with the down operator $D^{\bullet}$, in 
\eqref{eq:U_bullet_operator}, the number of possible
outcomes $\mathbf{g}'$ for any fixed $\mathbf{g},\mathbf{d},\mathbf{d}'$
may be infinite. These infinitely many choices 
arise at the left boundary, for $m=1$, as explained in \Cref{rmk:left_boundary_up_bijectivisation}.

One readily sees that the
operators $D^{\bullet}$ and $U^{\bullet}$
from \Cref{def:D_bullet_operator,def:U_bullet_operator}
satisfy the detailed balance
equation
\eqref{eq:shift_ops_detailed_balance_equation}
involving the shift operator $B$ and the stochastic higher
spin six vertex model transfer matrices $T$ and $T_{\mathsf{sh}}$. This follows directly from the detailed balance equations for the Markov swap operators $D^{(0,n)}$ and $U^{(0,n)}$ given by \eqref{eq:D0n_U0n_abbreviation_third_time}.

\begin{remark}
	\label{rmk:D_U_operators_in_space_X}
	Here, and in 
	\Cref{sub:up_down_on_configurations}, 
	we argued in terms of
	the space $\mathscr{G}$ of vertex model
	configurations.
	Note that we may define corresponding
	down and up Markov operators
	to act on the space $\mathscr{X}$ of
	exclusion process configurations
	via the gap-particle transformation
	(\Cref{def:gap_particle_transform}).
	Following our convention so far, 
	we use the notation $\tilde D^{(n)},\tilde D^{(0,n)},\tilde D^{\bullet}$, and so on,
	to denote these operators acting on $\mathscr{X}$.
\end{remark}

\subsection{Coupling of measures on trajectories via rewriting history}
\label{sub:coupling_measures_on_trajectories}

We couple together trajectories of two instances of
the stochastic higher spin six vertex model with different parameters
through the use of
the down and up Markov operators 
defined in 
\Cref{sub:up_down_on_configurations,sub:bij_of_shift_operator}.
Here, we only consider this construction for the swap operator
$P^{(n)}$
and the operators $D^{(n)},U^{(n)}$ from \Cref{sub:up_down_on_configurations}.
The couplings involving the shift operator $B$
work very similarly, and they
will be discussed in a continuous time limit
in \Cref{sec:limit_equal_speeds_new} below.

Fix $n\in \mathbb{Z}_{\ge1}$ 
and
take parameters 
$(\mathbf{u},\mathbf{s})\in \mathcal{T}$ such that 
$\bigl( \frac{u_n}{u_{n-1}},s_{n-1},s_n \bigr)\in \mathcal{R}$.
That is, we assume that the parameters satisfy the conditions of 
\Cref{prop:Pn_T_intertw},
so that the Markov operators
$T,T_{\sigma_{n-1}},P^{(n)}$ 
\eqref{eq:comm_bij_abbreviation_Pn}
are well-defined.
Let $D^{(n)}$ and $U^{(n)}$ be the operators
from \Cref{sub:up_down_on_configurations}
providing a bijectivisation of the intertwining relation
$TP^{(n)}=P^{(n)}T_{\sigma_{n-1}}$
from \Cref{prop:Pn_T_intertw}.

Fix $M\ge 1$
and an initial configuration $\hat{\mathbf{g}}
\in \mathscr{G}$.
Denote by 
$\{\mathbf{g}(t)\}_{0\le t\le M}$ the 
stochastic higher spin six vertex model
with parameters $(\mathbf{u},\mathbf{s})$
started from $\hat{\mathbf{g}}$.
Also, denote by 
$\{\mathbf{d}(t)\}_{0\le t\le M}$
the vertex model with parameters
$(\sigma_{n-1}\mathbf{u},\sigma_{n-1}\mathbf{s})$
started from a random initial configuration
$\delta_{\hat{\mathbf{g}}}P^{(n)}$.
Let $\mathfrak{T}$ and $\mathfrak{T}^{\sigma_{n-1}}$, respectively, 
denote the measures on trajectories of these processes 
on $\mathscr{G}$.
Then, in particular, the probability weights for these measures are given by
\begin{equation}
	\label{eq:measures_on_trajectories_for_Pn}
	\begin{split}
		&
		\mathfrak{T}
		(\mathbf{g}(0),\mathbf{g}(1),\ldots,\mathbf{g}(M) )=
		\mathbf{1}_{\mathbf{g}(0)=\hat{\mathbf{g}}}
		\ssp
		T(\mathbf{g}(0),\mathbf{g}(1))\ssp
		T(\mathbf{g}(1),\mathbf{g}(2))
		\ldots 
		T(\mathbf{g}(M-1),\mathbf{g}(M))
		,\\
		&
		\mathfrak{T}^{\sigma_{n-1}}
		(\mathbf{d}(0),\mathbf{d}(1),\ldots,\mathbf{d}(M) )
		=
		P^{(n)}(\hat{\mathbf{g}},\mathbf{d}(0))
		\ssp
		T_{\sigma_{n-1}}(\mathbf{d}(0),\mathbf{d}(1))
		\ldots
		T_{\sigma_{n-1}}(\mathbf{d}(M-1),\mathbf{d}(M)).
	\end{split}
\end{equation}
The iterated intertwining relation
$T^M P^{(n)}=P^{(n)}(T_{\sigma_{n-1}})^{M}$
implies that the distribution of the final state
$\mathbf{d}(M)$ 
of $\mathfrak{T}^{\sigma_{n-1}}$
is the same as the distribution
of $\delta_{\mathbf{g}(M)}P^{(n)}$, obtained by applying
$P^{(n)}$ to the final state of $\mathfrak{T}$
(see \Cref{fig:history_rewriting_Pn} for an illustration).
The next statement 
extends this 
identity in distribution
to couplings between
\emph{joint distributions in time};
this is the main result of the current \Cref{sec:bijectivisation}.
These couplings have a 
sequential nature 
(where the time $t$ runs through
$t\in \left\{ 0,1,\ldots,M  \right\}$),
and may be thought of as ``rewriting the history'' of a vertex model.
There are two 
distinct sequential
couplings corresponding to the direction in which the time $t$
is varied.

\begin{figure}[htb]
	\centering
	\includegraphics[width=\textwidth]{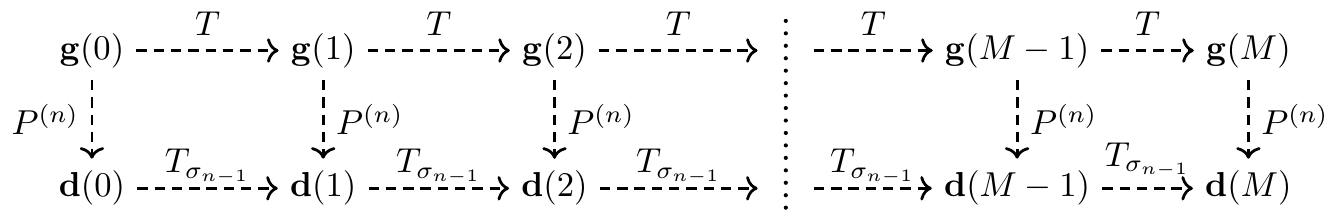}
	\caption{The chain of intertwining
	relations providing the coupling of trajectories 
	(\Cref{thm:coupling_trajectories_with_Pn_swap})
	of the 
	two processes in \eqref{eq:measures_on_trajectories_for_Pn}.}
	\label{fig:history_rewriting_Pn}
\end{figure}

\begin{theorem}
	\label{thm:coupling_trajectories_with_Pn_swap}
	\begin{enumerate}[\bf1.\/]
		\item 
			\textnormal{(Rewriting history from future to past)}
			Fix $\hat{\mathbf{g}}\in \mathscr{G}$.
			Let $\{\mathbf{g}(t)\}_{0\le t\le M}$ be distributed according to $\mathfrak{T}$.
			First, 
			apply $P^{(n)}$ to $\mathbf{g}(M)$, and denote this 
			random configuration by $\mathbf{d}'(M)$.
			Sequentially in the order $t=M-1,M-2,\ldots,1,0 $,
			let $\mathbf{d}'(t)$ be sampled
			from 
			\begin{equation}
				\label{eq:Dn_in_history_rewriting_swap}
				D^{(n)}(\mathbf{g}(t+1)\to \mathbf{d}'(t)\mid 
				\mathbf{g}(t),\mathbf{d}'(t+1)).
			\end{equation}
			Then, the joint distribution 
			of $\{\mathbf{d}'(t)\}_{0\le t\le M}$
			is equal to $\mathfrak{T}^{\sigma_{n-1}}$.
		\item
			\textnormal{(Rewriting history from past to future)}
			Fix $\hat{\mathbf{g}}\in \mathscr{G}$.
			Let
			$\{\mathbf{d}(t)\}_{0\le t\le M}$ 
			be
			distributed according to $\mathfrak{T}^{\sigma_{n-1}}$,
			where now the
			initial condition
			is random and depends on $\hat{\mathbf{g}}$.
			Sequentially in the order $t=1,2,\ldots,M $, 
			let $\mathbf{g}'(t)$ be sampled from
			\begin{equation}
				\label{eq:Un_in_history_rewriting_swap}
				U^{(n)}(\mathbf{d}(t-1)\to \mathbf{g}'(t)\mid \mathbf{g}'(t-1),\mathbf{d}(t)),
			\end{equation}
			where, by agreement, $\mathbf{g}'(0)=\hat{\mathbf{g}}$.
			Then, the joint distribution 
			of $\{\mathbf{g}'(t)\}_{0\le t\le M}$
			is equal to $\mathfrak{T}$.
	\end{enumerate}
\end{theorem}

\begin{proof}
    The results follow by iterating 
	the detailed balance equation \eqref{eq:Dn_Un_detailed_balance_equation}
	involving $T,T_{\sigma_{n-1}}$, and $P^{(n)}$. 
	Moreover, by the Markov property, it suffices to 
	consider joint distributions at adjacent time moments $t,t+1$, 
	and use induction in $t$. This induction is descending or ascending in the 
	first or the second part, respectively.
	See \Cref{fig:history_rewriting_Pn} for an illustration of the 
	notation employed throughout the proof. We give more details below.
	
	Consider the first part. Inductively, we show that the transition probability $ \mathbf{d}'(t) \rightarrow  \mathbf{d}'(t+1)$ is equal to $T_{\sigma_{n-1}} (\mathbf{d}'(t) ,  \mathbf{d}'(t+1))$ if the transition probability 
	for $\mathbf{g}(t) \rightarrow \mathbf{g}(t+1)$ is equal to $T( \mathbf{g}(t), \mathbf{g}(t+1))$. Additionally, for the induction argument, we show that the transition probability for $\mathbf{g}(t) \rightarrow \mathbf{d}'(t)$ is equal to $P^{(n)}(\mathbf{g}(t), \mathbf{d}'(t))$ if the transition probability for $\mathbf{g}(t+1) \rightarrow \mathbf{d}'(t+1)$ is equal to $P^{(n)}(\mathbf{g}(t+1), \mathbf{d}'(t+1))$. We start the induction by noting that the conditions are true for the first step when $t+1 = M$ by assumption. In the following, we carry out the computations for the induction step. 
	
	Let us assume that 
	$\mathbf{g}(t)$ is known. That is, the 
	following computations are conditioned on $\mathbf{g}(t)$. Then, we have the following chain of
	expressions for the 
	joint distribution of 
	$\mathbf{d}'(t),\mathbf{d}'(t+1)$
	conditioned on $\mathbf{g}(t)$:
	\begin{equation*}
	    \begin{split}
	    &
			\mathrm{Prob}\left(
				\mathbf{d}'(t+1),\mathbf{d}'(t) \mid \mathbf{g}(t) 
			\right)
			\\
	    &=  
			\mathrm{Prob}\left(\mathbf{d}'(t) \mid \mathbf{g}(t)\right) \ssp \mathrm{Prob}\left(\mathbf{d}'(t+1) \mid \mathbf{d}'(t) , \mathbf{g}(t) \right)
			\\
	    &= \sum_{\mathbf{g}(t+1)} \mathrm{Prob}(\mathbf{g}(t+1), \mathbf{d}'(t+1),\mathbf{d}'(t) \mid \mathbf{g}(t) )\\
	    &=\sum_{\mathbf{g}(t+1)} \mathrm{Prob}( \mathbf{g}(t+1) \mid \mathbf{g}(t))\ssp \mathrm{Prob}( \mathbf{d}'(t+1) \mid \mathbf{g}(t+1), \mathbf{g}(t))\ssp \mathrm{Prob}( \mathbf{d}'(t) \mid \mathbf{g}(t+1), \mathbf{g}(t), \mathbf{d}'(t+1)) \\
	    &=\sum_{\mathbf{g}(t+1)}T(\mathbf{g}(t),\mathbf{g}(t+1))
		\ssp
		P^{(n)}(\mathbf{g}(t+1),\mathbf{d}'(t+1))
		\ssp
		D^{(n)}(\mathbf{g}(t+1)\to \mathbf{d}'(t)\mid 
		\mathbf{g}(t),\mathbf{d}'(t+1))\\
	    &=\sum_{\mathbf{g}(t+1)}P^{(n)}(\mathbf{g}(t),\mathbf{d}'(t))
		\ssp
		T_{\sigma_{n-1}}(\mathbf{d}'(t),\mathbf{d}'(t+1)) U^{(n)}(\mathbf{d}'(t) \to\mathbf{g}(t+1)\mid 
		\mathbf{g}(t),\mathbf{d}'(t+1))\\
	    &=P^{(n)}(\mathbf{g}(t),\mathbf{d}'(t))
		\ssp
		T_{\sigma_{n-1}}(\mathbf{d}'(t),\mathbf{d}'(t+1)).
	    \end{split}
	\end{equation*}
	We used the induction hypothesis on the fourth equality,
	detailed balance equation for the fifth equality, and the
	stochasticity for the sixth equality. 
	From the identity above
	we see that
	the conditional distribution
	of $\mathbf{d}'(t+1)$ given $\mathbf{d}'(t)$
	is $T_{\sigma_{n-1}}(\mathbf{d}'(t),\mathbf{d}'(t+1))$, which is independent
	of $\mathbf{g}(t)$.
	Moreover,
	the marginal 
	distribution of $\mathbf{d}'(t)$ is 
	$P^{(n)}(\mathbf{g}(t),\mathbf{d}'(t))$,
	which allows to continue the induction.
	Thus, the result for the first part
	follows.
	
	The second part is proven similarly, with a simplification 
	that we do not need to condition the computations on
	$\mathbf{d}(t)$ 
	due to the other direction of the Markov step $P^{(n)}$.
	This completes the proof.
\end{proof}

\begin{definition}[Markov operators for rewriting history]
	\label{def:rewriting_history_operators}
	Fix a trajectory 
	$\{\mathbf{g}(t)\}_{0\le t\le M}$ 
	of the stochastic higher spin six vertex model
	with some initial data~$\hat{\mathbf{g}}$,
	and also fix an \emph{arbitrary} 
	configuration $\mathbf{d}'(M)$ at the final time such that
	$P^{(n)}(\mathbf{g}(M),\mathbf{d}'(M))\ne 0$.
	Given $\mathbf{d}'(M)$, 
	denote 
	by 
	$H_n^{\leftarrow}$
	the Markov operator that maps the 
	trajectory
	$\{\mathbf{g}(t)\}_{0\le t\le M}$ 
	to the trajectory
	$\{\mathbf{d}'(t)\}_{0\le t\le M}$
	by the
	sequential application of $D^{(n)}$
	as in 
	the first part of \Cref{thm:coupling_trajectories_with_Pn_swap}.
	The operator 
	$H_n^{\leftarrow}$ may be viewed
	as a Markov process with initial
	condition
	$\mathbf{d}'(M)$ and running \emph{backwards in time}, from future to past.

	Similarly, fix a trajectory 
	$\{\mathbf{d}(t)\}_{0\le t\le M}$ with some initial data $\hat{\mathbf{d}}$,
	and fix an \emph{arbitrary} configuration $\hat{\mathbf{g}}$
	such that $P^{(n)}(\hat{\mathbf{g}},\hat{\mathbf{d}})\ne 0$.
	Given $\hat{\mathbf{g}}$, denote by $H^{\rightarrow}_{n}$ the Markov operator
	that
	maps the trajectory 
	$\{\mathbf{d}(t)\}_{0\le t\le M}$
	to the trajectory
	$\{\mathbf{g}'(t)\}_{0\le t\le M}$
	by the sequential application of $U^{(n)}$ as in the second part
	of \Cref{thm:coupling_trajectories_with_Pn_swap}.
	The operator
	$H^{\rightarrow}_{n}$ may be viewed as a Markov process with initial condition $\hat{\mathbf{g}}$ and running \emph{forward in time}, from past to future.

	We call 
	$H_n^{\leftarrow}$ 
	and 
	$H_n^{\rightarrow}$ the \emph{Markov operators for rewriting history}
	corresponding to the swap operator $P^{(n)}$.
\end{definition}

Note that both $H_n^{\leftarrow}$ and $H^{\rightarrow}_n$
act locally and change only the components
$g_{n-1},g_n$ along the trajectory of the stochastic history of the spin six vertex
model. This locality comes from the same feature of the Markov swap operator 
$P^{(n)}$.

We reformulate \Cref{thm:coupling_trajectories_with_Pn_swap},
with this definition. 
Recall that measures
on trajectories are defined by
\eqref{eq:measures_on_trajectories_for_Pn}.
\begin{corollary}
	\label{cor:coupling_trajectories_with_Pn_swap}
	\begin{enumerate}[\bf1.\/]
		\item 
			If a trajectory
			$\{\mathbf{g}(t)\}_{0\le t\le M}$
			has distribution $\mathfrak{T}$ 
			and
			$\mathbf{d}'(M)$
			has distribution $\delta_{\mathbf{g}(M)}P^{(n)}$,
			then the application of 
			$H_n^{\leftarrow}$ 
			(with initial condition
			$\mathbf{d}'(M)$)
			to 
			$\{\mathbf{g}(t)\}_{0\le t\le M}$
			produces a trajectory with distribution $\mathfrak{T}^{\sigma_{n-1}}$.
		\item 
			If 
			a trajectory 
			$\{\mathbf{d}(t)\}_{0\le t\le M}$
			has distribution $\mathfrak{T}^{\sigma_{n-1}}$
			(in particular, its initial condition $\hat{\mathbf{d}}$ 
			has distribution $\delta_{\hat{\mathbf{g}}}P^{(n)}$, where $\hat{\mathbf{g}}$
			is fixed),
			then the application of 
			$H^{\rightarrow}_{n}$ 
			(with initial condition $\hat{\mathbf{g}}$)
			to
			$\{\mathbf{d}(t)\}_{0\le t\le M}$
			produces a trajectory with distributed
			$\mathfrak{T}$.
	\end{enumerate}
\end{corollary}

\section{Application to discrete-time particle systems}
\label{sec:discrete_time_bijectivisation}

We now consider the simplest
bijectivisation for a 
subfamily 
of stochastic higher spin six vertex models
We call this the independent bijectivisation.
The advantage of this subfamily is that the cross vertex weights 
factorize into the $q$-beta-binomial form.
The subfamily of stochastic higher spin six vertex models 
is still quite general and, in particular, 
includes
$q$-TASEP and TASEP.
In the following, we also translate the Markov operators 
$H_n^{\leftarrow}$ and $H^{\rightarrow}_n$
for rewriting history in these vertex models
into the language of particle systems.

\subsection{Notation and independent bijectivisation}
\label{sub:discrete_bij_subsection_8_1}

Consider the setting of \Cref{sub:YBE_bijectivisation}
(bijectivisation of a single Yang-Baxter equation)
and take $J=1$, $u_1=-\beta s_1$, $u_2=-\beta s_2$,
where $s_1,s_2\in(-1,0)$ with 
$|s_2|\le |s_1|$, and $\beta>0$.
Then, the vertex weights 
in the Yang-Baxter equation \eqref{eq:YBE_for_LJ}
become as in \Cref{fig:vertex_weights_for_bij}.
In particular, the cross vertex weights 
factorize into the $q$-beta-binomial form,
see \Cref{prop:qHahn_degeneration_of_R}.
Our conditions on the parameters make all the terms in the 
Yang-Baxter equation, i.e.~the weights
\eqref{eq:YBE_terms_w_LHS_RHS}, nonnegative.

\begin{figure}[htb]
	\centering
	\includegraphics[height=.55\textwidth]{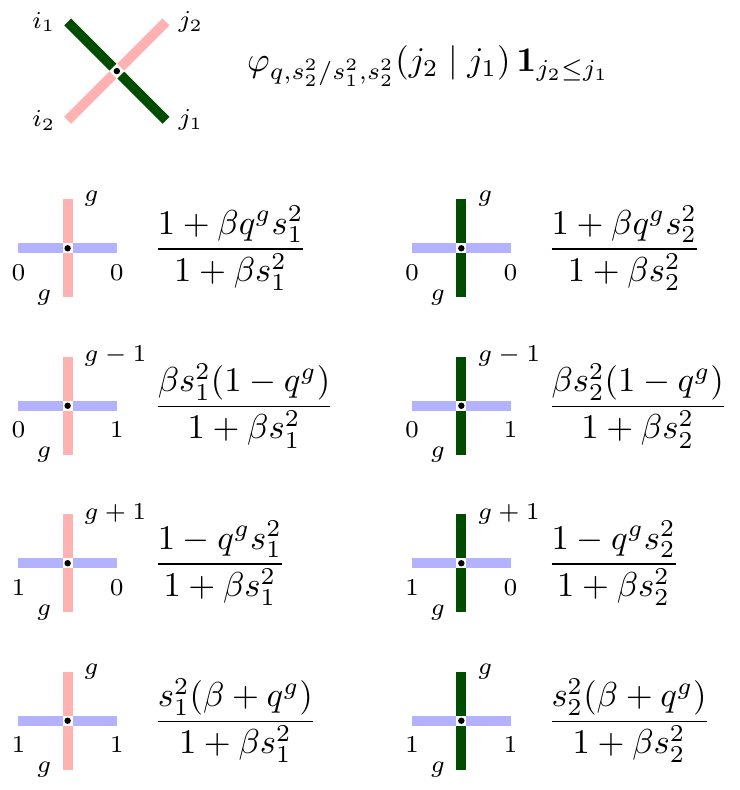}
	\caption{Vertex weights entering the Yang-Baxter equation considered in 
		the current
		\Cref{sec:discrete_time_bijectivisation}.}
	\label{fig:vertex_weights_for_bij}
\end{figure}

Recall that the boundary conditions for the Yang-Baxter
equation are encoded by 
$i_1,j_1 \in\left\{ 0,1,\ldots,J  \right\}$
and $i_2,i_3,j_2,j_3\in \mathbb{Z}_{\ge0}$,
see \Cref{fig:YBE}.
There are only four possible
values for
the pair $(i_1,j_1)\in\left\{ 0,1 \right\}^2$
encoding the horizontal boundary conditions
since we are taking $J=1$.
Fixing $i_1,j_1$, let us
employ the shorthand notation
\begin{equation}
	\label{eq:81_notation_1}
	i_2=a,\qquad i_3=b,\qquad j_2=c.
\end{equation}
We always have $0\le c\le b+1$, otherwise the cross vertex weight vanishes.
The value of $j_3$ is recovered from the path conservation property:
\begin{equation}
	\label{eq:81_notation_2}	
	j_3=a+b-c+i_1-j_1.
\end{equation}
Assuming that $a,b,c$ are also fixed, 
the terms 
\eqref{eq:YBE_terms_w_LHS_RHS}
in both sides of the Yang-Baxter equation,
as well as the transition probabilities \eqref{eq:bij_p_up_p_down},
may all be encoded by the numbers of paths through the internal horizontal edge $k_1,k_1'\in\left\{ 0,1 \right\}$, 
see \Cref{fig:YBE,fig:YBE_bijectivisation}.
Indeed, given $k_1$, we may reconstruct $k_2,k_3$ from $a,b,c,i_1,j_1$,
and similarly for $k_2',k_3'$ given $k_1'$.
We use the following shorthand notation for the corresponding weights and
transition probabilities:
\begin{equation}
	\label{eq:81_notation_3}	
	w^{\mathrm{LHS}}_{i_1,j_1}(k_1),
	\qquad 
	w^{\mathrm{RHS}}_{i_1,j_1}(k_1'),
	\qquad 
	p^{\downarrow}_{i_1,j_1}[k_1\to k_1'],
	\qquad 
	p^{\uparrow}_{i_1,j_1}[k_1'\to k_1].
\end{equation}
The Yang-Baxter equation thus takes the form
\begin{equation}
	\label{eq:YBE_through_shorthand_notation}
	w^{\mathrm{LHS}}_{i_1,j_1}(0)+
	w^{\mathrm{LHS}}_{i_1,j_1}(1)=
	w^{\mathrm{RHS}}_{i_1,j_1}(0)+
	w^{\mathrm{RHS}}_{i_1,j_1}(1).
\end{equation}

To simplify the constructions of our couplings, 
throughout the rest of the paper
we consider the so-called \emph{independent bijectivisation}
so that the transition, say $k_1\to k_1'$,
depends only on the end state $k_1'$ and the boundary conditions, and 
not on $k_1$. More specifically, we give the following
definition:

\begin{definition}
	\label{def:independent_bijectivisation}
	For a fixed set of boundary parameters $a\in \mathbb{Z}_{\ge0}\cup\left\{+\infty \right\}$,
	$b,c\in \mathbb{Z}_{\ge0}$ and $i_1,j_1\in\left\{ 0,1 \right\}$, the transition probabilities for the \emph{independent bijectivisation} are given by:
	\begin{equation}
		\label{eq:bij_independent}
		p^{\downarrow}_{i_1,j_1}[k_1\to k_1']
		\coloneqq
		\frac{w^{\mathrm{RHS}}_{i_1,j_1}(k_1')}
		{w^{\mathrm{RHS}}_{i_1,j_1}(0)+w^{\mathrm{RHS}}_{i_1,j_1}(1)}
		\qquad 
		p^{\uparrow}_{i_1,j_1}[k_1'\to k_1]
		\coloneqq
		\frac{w^{\mathrm{LHS}}_{i_1,j_1}(k_1)}
		{w^{\mathrm{LHS}}_{i_1,j_1}(0)+w^{\mathrm{LHS}}_{i_1,j_1}(1)},
	\end{equation}
	with $k_1,k_1'\in\left\{ 0,1 \right\}$.
\end{definition}
\begin{remark}
	Note that the denominator is nonzero
	in each of the two expressions in \eqref{eq:bij_independent}.
	Otherwise, there
	is no Yang-Baxter equation with the given
	boundary conditions $a,b,c,i_1,j_1$
	and, correspondingly, there is no bijectivisation.
	One readily sees that the
	transition probabilities in \eqref{eq:bij_independent} 
	are always nonnegative and 
	satisfy the detailed balance equation 
	\eqref{eq:basic_bij_properties}.
	Observe that this bijectivisation corresponds to taking the
	coupling of measures on $A$ and $B$ described in \Cref{sub:bij_basics}
	to be simply the product measure.
	For this reason we call \eqref{eq:bij_independent}
	the \emph{independent bijectivisation}.

	The case $a=+\infty$ in \eqref{eq:bij_independent}
	corresponds to having infinitely
	many paths through the leftmost vertical edges, but 
	this does not present an issue since $J=1$; see \Cref{rmk:left_boundary_up_bijectivisation}. 
	In other words, the limits of 
	$p^{\downarrow}_{i_1,j_1}[k_1\to k_1']$
	and 
	$p^{\uparrow}_{i_1,j_1}[k_1'\to k_1]$
	as $a\to+\infty$ exist and give a well-defined bijectivisation of the Yang-Baxter
	equation with $i_2=j_3=+\infty$.
\end{remark}

Formulas arising from 
\eqref{eq:bij_independent}
do not have factorized denominators and, in general, can have a rather
complicated form despite the simplicity of the general definition. For example, we have
\begin{multline*}
	p^{\downarrow}_{0,0}[0\to 1]
	=
	p^{\downarrow}_{0,0}[1\to 1]
	\\=
	\frac{
		(1-q^{c-1}s_1^2)\beta s_2^2(1-q^{a+b-c+1})\ssp\varphi(c-1\mid b)
	}
	{(1+q^c \beta s_1^2)(1+q^{a+b-c}\beta s_2^2)\ssp\varphi(c\mid b)
	+(1-q^{c-1}s_1^2)\beta s_2^2(1-q^{a+b-c+1})\ssp\varphi(c-1\mid b)},
\end{multline*}
where we abbreviated $\varphi=\varphi_{q,s_2^2/s_1^2,s_2^2}$.
Note that this expression admits a straightforward
limit as $a\to+\infty$, making the bijectivisation at the left edge well-defined.

\subsection{Rewriting history in particle systems from future to past}
\label{sub:discrete_bij_particle_systems}

Let us now take the full stochastic higher spin six vertex model
with $J=1$, $u_i=-\beta s_i$, $i\in \mathbb{Z}_{\ge0}$,
where the parameters of the model are 
\begin{equation*}
	\beta>0,\qquad 
	\mathbf{s}=(s_0,s_1,s_2,\ldots ),\quad -1<s_i<0.
\end{equation*}
This vertex model corresponds to a discrete time stochastic 
particle system $\{\mathbf{x}(t)\}_{t\in \mathbb{Z}_{\ge0}}$ 
on the space $\mathscr{X}$ of 
particle configurations on $\mathbb{Z}$,
via the gap-particle transformation
(\Cref{def:gap_particle_transform}).
For any $n\geq 1$, the bijectivisation defined in \Cref{sub:discrete_bij_subsection_8_1}
gives rise
to the
Markov operators for rewriting history
as in \Cref{def:rewriting_history_operators}.
We denote the corresponding operators by
$\tilde H_n^{\leftarrow}$
and
$\tilde H_n^{\rightarrow}$,
following the convention that operators on $\mathscr{X}$ include a tilde.

Fix $n\ge1 $
and assume that
$|s_{n-1}|\ge |s_n|$.
In this subsection, we describe the Markov operator
$\tilde H_n^{\leftarrow}$
and, in the following \Cref{sub:discrete_bij_particle_systems_UP},
we describe the Markov operator
$\tilde H_n^{\rightarrow}$.
We first consider 
the down Markov operator 
$\tilde D^{(n)}$
from \Cref{def:Dn_Markov_operator_from_bij}.
For a fixed time $0\le t\le M-1$, the action of
$\tilde D^{(n)}$ as in \eqref{eq:Dn_in_history_rewriting_swap}
depends on the trajectories of the
neighboring particles around the $n$-th one:
\begin{equation*}
	\mathrm{x}_{n-1}\coloneqq x_{n-1}(t+1)
	>
	\mathrm{x}_{n+1}\coloneqq x_{n+1}(t+1),\qquad 
	\mathrm{x}_{n-1}-i_1=x_{n-1}(t)
	> 
	\mathrm{x}_{n+1}-j_1=x_{n+1}(t),
\end{equation*}
where $i_1,j_1\in \left\{ 0,1 \right\}$.
Given the new location 
$\mathrm{x}_n'$ of the $n$-th particle 
at time $t+1$, 
$\tilde D^{(n)}$
maps the two-time trajectory of the 
$n$-th particle,
\begin{equation*}
	\mathrm{x}_n-k_1= x_n(t)
	\le 
	\mathrm{x}_n = x_n(t+1)
	,
\end{equation*}
where $k_1\in \left\{ 0,1 \right\}$, into 
a random new trajectory
\begin{equation*}
	\mathrm{x}_n'-k_1'
	\le 
	\mathrm{x}_n' 
	,\qquad 
	k_1'\in\left\{ 0,1 \right\}.
\end{equation*}
See \Cref{fig:Dn_Un_for_particles}, left,
for an illustration.
The coordinates introduced above must satisfy
\begin{equation}
	\label{eq:Dn_history_action_inequalities}
	\begin{split}
		\mathrm{x_{n-1}}>\mathrm{x}_n>\mathrm{x}_{n+1}
		,\qquad &
		\mathrm{x_{n-1}}-i_1>\mathrm{x}_n-k_1>\mathrm{x}_{n+1}-j_1
		,
		\\
		\mathrm{x_{n-1}}>\mathrm{x}_n'>\mathrm{x}_{n+1}
		,\qquad &
		\mathrm{x_{n-1}}-i_1>\mathrm{x}_n'-k_1'>\mathrm{x}_{n+1}-j_1,
		\\
		\mathrm{x}_n \ge \mathrm{x}_n' > \mathrm{x}_{n+1},\qquad &
		\mathrm{x}_n-k_1 \ge \mathrm{x}_n'-k_1'>\mathrm{x}_{n+1}-j_1.
	\end{split}
\end{equation}
The inequalities on the last line in \eqref{eq:Dn_history_action_inequalities} come from the
$q$-beta-binomial specialization of the cross vertex weights
$R_{s_n/s_{n-1},s_{n-1},s_n}(\mathsf{i}_1,\mathsf{i}_2;\mathsf{j}_1,\mathsf{j}_2)$,
which contain the indicator $\mathbf{1}_{\mathsf{j}_2\le \mathsf{j}_1}$,
see \Cref{prop:qHahn_degeneration_of_R}.

\begin{figure}[htb]
	\centering
	\includegraphics[height=.45\textwidth]{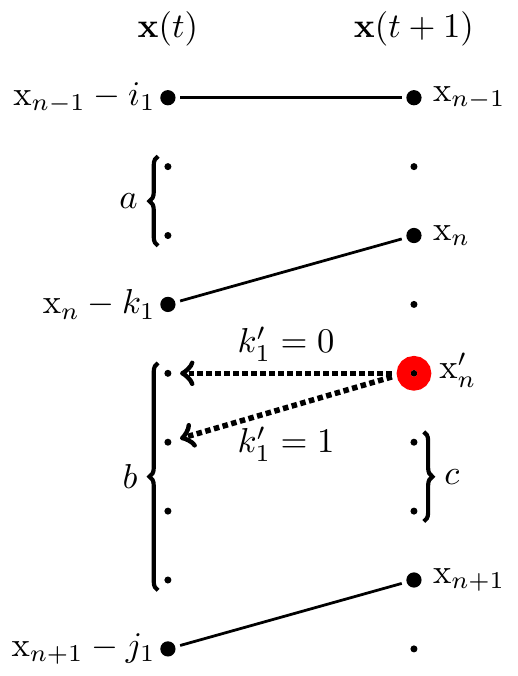}
	\hspace{70pt}
	\includegraphics[height=.45\textwidth]{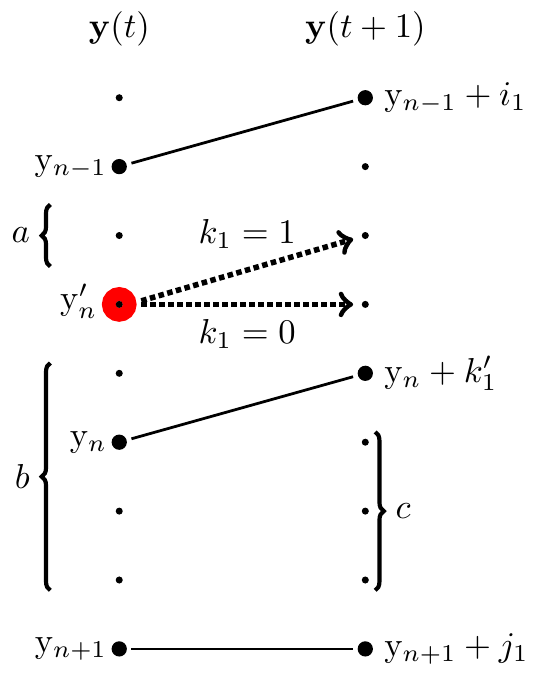}
	\caption{Left: Action of $\tilde D^{(n)}$ with boundary conditions
		$i_1=0$, $j_1=1$. Given the location at time $t+1$, the particle $\mathrm{x}_n'$
		randomly chooses its location at time $t$
		from the possible locations $\{ \mathrm{x}_n',\mathrm{x}_n'-1 \}$ with probabilities determined by
		$p^{\downarrow}_{i_1,j_1}$,
		independent of $k_1$ unless $\mathrm{x}_n'$ is pushed down or blocked.
		\\Right:
		Action of $\tilde U^{(n)}$ with boundary conditions
		$i_1=1$, $j_1=0$.
		Similarly, given the location at time $t$, the particle $\mathrm{y}_n'$ randomly 
		chooses its location at time $t+1$ from the possible locations
		$\{ \mathrm{y}_n', \mathrm{y}_n'+1 \}$ with probabilities
		determined by $p^{\uparrow}_{i_1,j_1}$,
		independent of $k_1'$ 
		unless $\mathrm{y}_n'$ is pushed up or blocked.}
	\label{fig:Dn_Un_for_particles}
\end{figure}

The probability to select 
$k_1'\in\left\{ 0,1 \right\}$ is 
$p^{\downarrow}_{i_1,j_1}[k_1\to k_1']$ from
\eqref{eq:bij_independent}
with parameters $(s_1,s_2)$ replaced by 
$(s_{n-1},s_n)$,
horizontal edge occupation numbers
$i_1,j_1,k_1$ specified above, and
\begin{equation*}
	a=\mathrm{x}_{n-1}-\mathrm{x}_n-1-i_1+k_1
	,\qquad 
	b=\mathrm{x}_{n}-\mathrm{x}_{n+1}-1-k_1+j_1
	,\qquad 
	c=\mathrm{x}_n'-\mathrm{x}_{n+1}-1,
\end{equation*}
as indicated in 
\Cref{fig:Dn_Un_for_particles}, left.

Under the independent bijectivisation, 
the probabilities 
$p^{\downarrow}_{i_1,j_1}[k_1\to k_1']$
are chosen to be independent (as much as possible)
of the old trajectory of the $n$-th particle.
More precisely, they depend on the old state 
$\mathrm{x}_n-k_1$ at time $t$  through $a,b$,
but not on the old state $\mathrm{x}_n$ at time $t+1$.
There are, however, two cases when 
the value of $k_1'$ is deterministically prescribed
by the last inequality in \eqref{eq:Dn_history_action_inequalities}:
\begin{enumerate}[$\bullet$]
	\item (\emph{blocking}) If $c=0$ and $j_1=0$,
		then $k_1'=0$ with probability $1$. This means that $\mathrm{x}_n'$
		is blocked by $\mathrm{x}_{n+1}$ from going down due to close proximity.
	\item (\emph{pushing down})
		If $c=b+1-j_1$,
		then $k_1=1$ and, additionally, $k_1'=1$ with probability $1$.
		This means that $\mathrm{x}_n'$
		is pushed down by $\mathrm{x}_n$ due to close proximity.
\end{enumerate}
We see that in the pushing case, the independent bijectivisation cannot ignore $k_1$ which encoded the old trajectory of the $n$-th particle.

\medskip

Applying the operators
$\tilde D^{(n)}$
sequentially for $t=M-1,M-2,\ldots,1,0 $,
we arrive at the 
operator
$\tilde H_n^{\leftarrow}$ for rewriting history. 
The action of 
$\tilde H_n^{\leftarrow}$ may be viewed as a Markov process
running backwards in time,
which replaces the old trajectory
$\{x_n(t) \}_{0\le t\le M}$
by the new one,
$\{x_n'(t) \}_{0\le t\le M}$. 
The Markov process 
for building
$\{x_n'(t) \}_{0\le t\le M}$
starts from a fixed initial condition 
$x_n'(M)$ such that 
$x_{n+1}(M)<x_n'(M)\le x_n(M)$,
and evolves in the chamber
\begin{equation*}
	x_{n+1}(t)<x_n'(t)\le x_n(t),\qquad  0\le t\le M.
\end{equation*}
The trajectory of the upper neighbor $\{x_{n}(t)\}_{0\le t\le M}$
affects the transition probabilities of 
$x_n'(t)$ due to the push rule described above.
We refer to \Cref{fig:rewriting_history_for_trajectories}, left,
for an illustration. 

Thus, the operator 
$\tilde H_n^{\leftarrow}$
satisfies the first part of \Cref{cor:coupling_trajectories_with_Pn_swap},
where all vertex configurations and operators
are replaced by their exclusion process counterparts
using the gap-particle transformation. Below, in
\Cref{sub:first_particle_general_result},
we explicitly describe the Markov process
$\tilde H_n^{\leftarrow}$ 
for $n=1$ when there is no upper neighbor and 
the transition probabilities are simpler.

\begin{figure}[htbp]
	\centering
	\includegraphics[height=.28\textwidth]{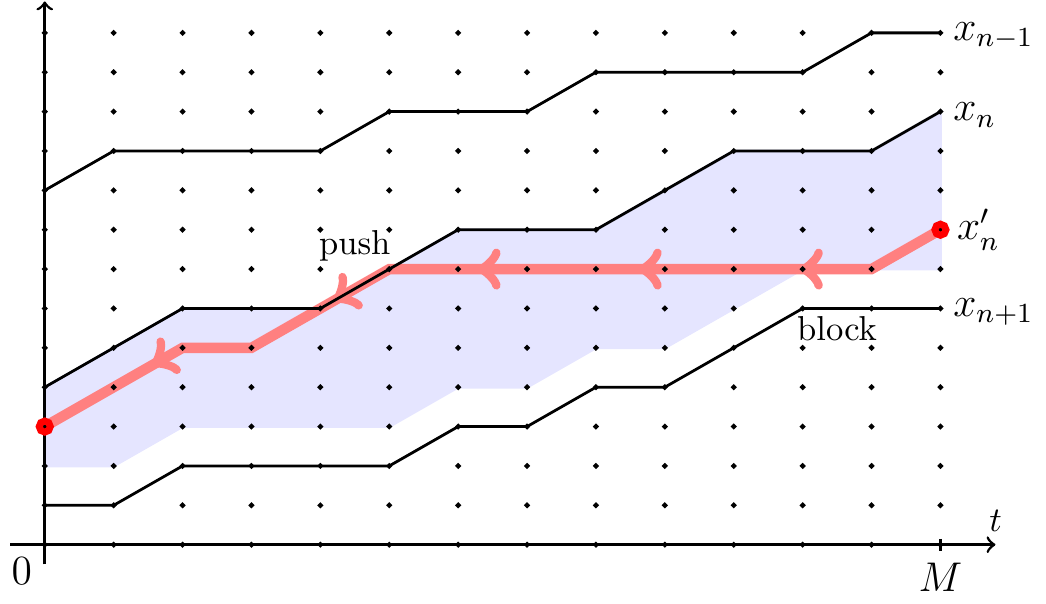}
	\hspace{5pt}
	\includegraphics[height=.28\textwidth]{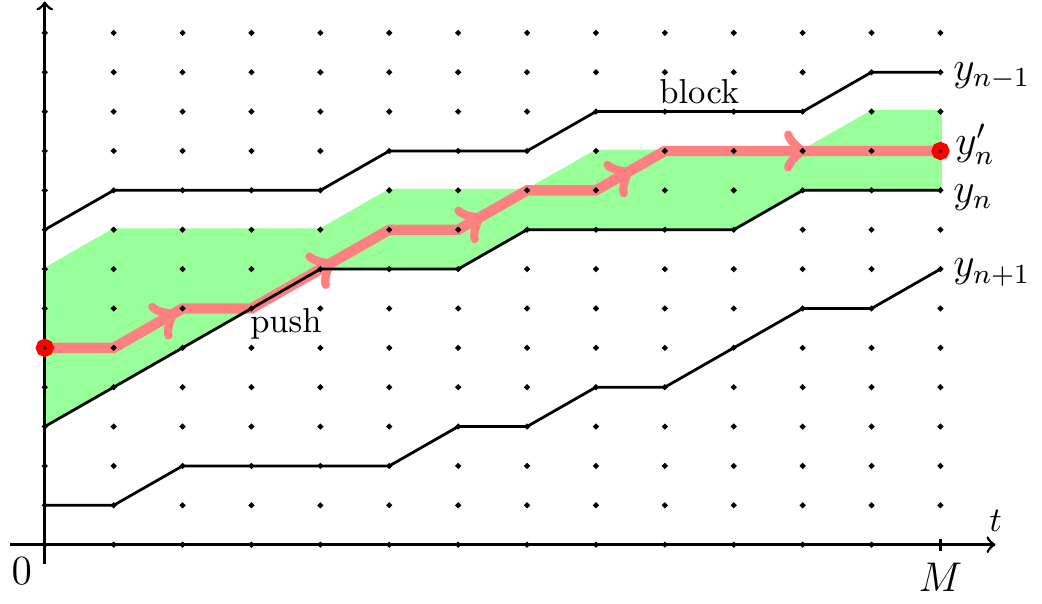}
	\caption{Left: Rewriting history from future to past using the 
		operator $\tilde H_n^{\leftarrow}$. 
	Right: Rewriting 
	history from past to future using the operator 
	$\tilde H_n^{\rightarrow}$. In both cases the allowed chamber for the 
	new trajectory of the $n$-th particle is shaded. 
	When the new trajectory reaches the boundary of this chamber,
it is pushed or blocked depending on the type of the boundary.}
	\label{fig:rewriting_history_for_trajectories}
\end{figure}

\subsection{Rewriting history in particle systems from past to future}
\label{sub:discrete_bij_particle_systems_UP}

The Markov operator $\tilde H_n^{\rightarrow}$
for rewriting history from
past to future is treated very similarly to Markov operator 
$\tilde H_n^{\leftarrow}$ for rewriting history from future to past, as described in \Cref{sub:discrete_bij_particle_systems}.
Here, we only indicate the main notation and definitions.
For instance, we denote particle coordinates by $y_j$ instead of $x_j$ to distinguish from the previous subsection.

Let $n\ge1$ and $|s_{n-1}|\ge |s_n|$ in the 
stochastic higher spin six vertex model with $J=1$ and $u_i=-\beta s_i$.
We first describe the up Markov operator
$\tilde U^{(n)}$ from the independent bijectivisation.
For time $0\le t\le M-1$, it depends on the two-time trajectories
\begin{equation*}
	\mathrm{y}_{n-1}\coloneqq y_{n-1}(t)>
	\mathrm{y}_{n+1}\coloneqq y_{n+1}(t),
	\qquad 
	\mathrm{y}_{n-1}+i_1=y_{n-1}(t+1)
	>
	\mathrm{y}_{n+1}+j_1=y_{n+1}(t+1),
\end{equation*}
where $i_1,j_1\in \left\{ 0,1 \right\}$.
Given the new location $\mathrm{y}_n'$,
the operator $\tilde U^{(n)}$
maps the two-time trajectory 
of the $n$-th particle,
$\mathrm{y}_n\le \mathrm{y}_n+k_1'$, where $k_1'\in \left\{ 0,1 \right\}$,
into a random new trajectory
$\mathrm{y}_n'\le \mathrm{y}_n'+k_1$, where $k_1\in \left\{ 0,1 \right\}$.
See \Cref{fig:Dn_Un_for_particles}, right, for an illustration.
All the coordinates must satisfy
\begin{equation}
	\label{eq:Un_history_action_inequalities}
	\begin{split}
		\mathrm{y}_{n-1}>\mathrm{y}_n>\mathrm{y}_{n+1},\qquad &
		\mathrm{y}_{n-1}+i_1>\mathrm{y}_n+k_1'>\mathrm{y}_{n+1}+j_1,\\
		\mathrm{y}_{n-1}>\mathrm{y}_n'>\mathrm{y}_{n+1},\qquad &
		\mathrm{y}_{n-1}+i_1>\mathrm{y}_n'+k_1>\mathrm{y}_{n+1}+j_1,\\
		\mathrm{y}_{n-1}> \mathrm{y_n}'\ge \mathrm{y}_n,\qquad &
		\mathrm{y}_{n-1}+i_1 > \mathrm{y_n}'+k_1 \ge \mathrm{y}_n+k_1'.
	\end{split}
\end{equation}
The probability to select 
$k_1\in \left\{ 0,1 \right\}$ is 
$p^{\uparrow}_{i_1,j_1}[k_1'\to k_1]$ given in 
\eqref{eq:bij_independent}
with parameters $(s_1,s_2)$ replaced by 
$(s_{n-1},s_n)$,
horizontal edge occupation numbers
$i_1,j_1,k_1'$ specified above, and
\begin{equation*}
	a=\mathrm{y}_{n-1}-\mathrm{y}_n'-1,\qquad 
	b=\mathrm{y}_n'-\mathrm{y}_{n+1}-1,\qquad 
	c=\mathrm{y}_n-\mathrm{y}_{n+1}-1+k_1'-j_1
\end{equation*}
as indicated in \Cref{fig:Dn_Un_for_particles}, right.
There are also blocking and pushing mechanisms present:
\begin{enumerate}[$\bullet$]
	\item (\emph{blocking}) 
		If $a=0$ and $i_1=0$,
		then $k_1=0$ with probability $1$. This means that
		$\mathrm{y}_n'$ is blocked 
		and cannot go up
		due to close proximity to 
		$\mathrm{y}_{n-1}$.
	\item (\emph{pushing up})
		If $c=b+1-j_1$, then $k_1'=1$ and, additionally,  
		$k_1=1$ with probability $1$.
		This means that $\mathrm{y}_n'$ is pushed up by $\mathrm{y}_n$
		due to close proximity.
\end{enumerate}

Applying the operators 
$\tilde U^{(n)}$ sequentially for
$t=0,1,\ldots,M-1 $, we 
arrive at the Markov operator
$\tilde H_n^{\rightarrow}$ for rewriting history from past to future.
Its action may be viewed as a Markov process
which builds the new trajectory
$\{y_n'(t)\}_{0\le t\le M}$ of the $n$-th particle which 
lies in the chamber
\begin{equation*}
	y_n(t)\le y_n'(t)<y_{n-1}(t),\qquad 0\le t\le M.
\end{equation*}
This new trajectory, which replaces the old trajectory $\{y_n(t)\}_{0\le t\le M}$, starts from a fixed initial condition $y_n'(0)$,
and its law depends on the trajectories
of the particles with numbers $n-1,n,n+1$.
See \Cref{fig:rewriting_history_for_trajectories}, right, for an illustration.

Thus, the operator 
$\tilde H_n^{\rightarrow}$
satisfies the second part of \Cref{cor:coupling_trajectories_with_Pn_swap},
where all vertex configurations and operators
are replaced by their exclusion process counterparts
using the gap-particle transformation. Below, in
\Cref{sub:first_particle_general_result},
we explicitly describe 
$\tilde H_n^{\rightarrow}$ 
for $n=1$ when
the transition probabilities are simpler.

\subsection{Resampling the first particle}
\label{sub:first_particle_general_result}

Let us illustrate 
the general results of the previous
\Cref{sub:discrete_bij_particle_systems,sub:discrete_bij_particle_systems_UP} and
consider the case
$n=1$, that is, the system of two particle.
Let $\mathbf{x}(t)=(x_1(t),x_2(t))$, $x_1(t)>x_2(t)$,
be the 
first two particles
in the system given at the beginning of \Cref{sub:discrete_bij_particle_systems}.
More precisely, 
$\mathbf{x}(t)$ 
corresponds (via the gap-particle transformation,
see \Cref{def:gap_particle_transform})
to the stochastic higher spin six vertex model
with $J=1$ and the 
parameters
\begin{equation}
	\label{eq:params_for_2_particles}
	u_i=-\beta s_i,\qquad 
	\beta>0, \qquad |s_0|>|s_1|.
\end{equation}
We will omit $J$ and $\beta$ in the notation 
and simply say that 
$\mathbf{x}(t)$ 
has parameters $(s_0,s_1)$.
Let also $\mathbf{y}(t)=(y_1(t),y_2(t))$
be the process with the swapped parameters
$(s_1,s_0)$.
We will describe two couplings between
$\mathbf{x}(t)$
and
$\mathbf{y}(t)$.
Similar couplings may be explicitly written down for any number of particles, 
but the advantage for $n=1$ is 
that the transition probabilities have a simple form.

We start with rewriting history from future to past.
Let us denote 
\begin{equation}
	\label{eq:d_quantities}
	\mathtt{d}^0_{b,c}
	=
	\frac{(1+\beta s_0^2 q^c)(1-q^{b+1-c})}{\beta(s_0^2-s_1^2 q^{b-c})(1-q^c)}
	,
	\qquad 
	\mathtt{d}^1_{b,c}
	=
	\frac{(1-s_0^2q^c)(1-q^{b-c})}{(\beta+q^c)(s_0^2-s_1^2 q^{b-c-1})},
\end{equation}
where $b,c\in \mathbb{Z}_{\ge0}$.
The quantities \eqref{eq:d_quantities}
take values in $[0,+\infty]$.
In particular, we may have $\mathtt{d}^0_{b,0}=+\infty$ and
$\mathtt{d}^0_{b,b+1}=\mathtt{d}^1_{b,b}=0$, which will respectively correspond to 
blocking and pushing down as in \Cref{sub:discrete_bij_particle_systems}.

Fix $M\in \mathbb{Z}_{\ge1}$. Assume that we are given a chamber 
$\{x_1(t)>x_2(t)\}_{0\le t\le M}$, and also $x_1'(M)$
with
$x_1(M)\ge x_1'(M)>x_2(M)$.
The process $\tilde H_1^{\leftarrow}$
is 
a random walk $x_1'(t)$ in a chamber
(i.e.~$x_1(t)\ge x_1'(t)>x_2(t)$ for all $t$)
which is started from $x_1'(M)$ and
runs in reverse time
$t=M-1,M-2,\ldots,1,0 $. See \Cref{fig:rewriting_history_for_trajectories},
left with $n=1$ and $x_0=+\infty$,
for an illustration.
During time step $t+1\to t$,
this random walk takes steps $0$ or~$-1$ with
probabilities
\begin{equation}
	\label{eq:d_transition_probabilities}
	\frac{\mathtt{d}^{j_1}_{b,c}}{1+\mathtt{d}^{j_1}_{b,c}}
	\quad \textnormal{and}\quad 
	\frac{1}{1+\mathtt{d}^{j_1}_{b,c}},
\end{equation}
respectively, 
where
\begin{equation*}
	b=x_1(t)-x_2(t)-1,
	\qquad
	c=x_1'(t+1)-x_2(t+1)-1,
	\qquad
	j_1=x_2(t+1)-x_2(t)\in \left\{ 0,1 \right\},
\end{equation*}
as in \Cref{fig:Dn_Un_for_particles}, left, with $a=+\infty$.

\begin{proposition}
	\label{prop:d_first_particle_result}
	Fix $M\in \mathbb{Z}_{\ge1}$.
	Let $\mathbf{x}(t)$
	be the system as above, with parameters $(s_0,s_1)$ satisfying 
	\eqref{eq:params_for_2_particles}
	and initial
	conditions so that $x_1(0)>x_2(0)$.
	Also, let $x_1'(M)$ be chosen from the probability distribution
	\begin{equation}
		\label{eq:d_first_particle_result_1}
		\varphi_{q,s_1/s_0^2,s_1^2}(x_1'(M)-x_2(M)-1\mid x_1(M)-x_2(M)-1).
	\end{equation}
	Moreover, let $x_1'(t)$ be the random walk
	in the chamber,
	$x_1(t)\ge x_1'(t)>x_2(t)$,
	with transition probabilities \eqref{eq:d_transition_probabilities} 
	as above. Then, the joint distribution 
	of the new process $\{(x_1'(t),x_2(t))\}_{0\le t\le M}$ is equal to the distribution of the process
	$\mathbf{y}(t)$ 
	with parameters $(s_1,s_0)$
	started from the initial condition
	$y_1'(0)>x_2(0)$,
	where
	$y_1'(0)$ is random and
	chosen from 
	\begin{equation}
		\label{eq:d_first_particle_result_2}
		\varphi_{q,s_1/s_0^2,s_1^2}(y_1'(0)-x_2(0)-1\mid x_1(0)-x_2(0)-1).
	\end{equation}
\end{proposition}
\begin{proof}
	First, note that the distributions
	\eqref{eq:d_first_particle_result_1},
	\eqref{eq:d_first_particle_result_2}
	are precisely given by the corresponding application of the Markov
	swap operator
	$\tilde P^{(1)}$. Then, by
	the first part of \Cref{thm:coupling_trajectories_with_Pn_swap},
	it suffices to show that
	\begin{equation}
		\label{eq:d_first_particle_result_proof_1}
		\mathtt{d}^{0}_{b,c}
		=
		\frac{w^{\mathrm{RHS}}_{*,0}(0)}
		{w^{\mathrm{RHS}}_{*,0}(1)}
		,\qquad 
		\mathtt{d}^{0}_{b,c}=
		\frac{w^{\mathrm{RHS}}_{*,1}(0)}
		{w^{\mathrm{RHS}}_{*,1}(1)}
		,
	\end{equation}
	using the notation of \Cref{sub:discrete_bij_subsection_8_1}.
	These identities 
	are checked in a straightforward
	way using the vertex weights
	in \Cref{fig:vertex_weights_for_bij} with $(s_1,s_2)$ renamed to $(s_0,s_1)$.
	Thus, we have
	\begin{equation}
		\label{eq:d_first_particle_result_proof_2}
		\frac{\mathtt{d}^{j_1}_{b,c}}{1+\mathtt{d}^{j_1}_{b,c}}
		=
		p^{\downarrow}_{*,j_1}[k_1\to 0]
		,
		\qquad 
		\frac{1}{1+\mathtt{d}^{j_1}_{b,c}}
		=
		p^{\downarrow}_{*,j_1}[k_1\to 1]
		.
	\end{equation}
	Recall that star
	in \eqref{eq:d_first_particle_result_proof_1},
	\eqref{eq:d_first_particle_result_proof_2}
	means that there is no dependence on $i_1$ since $a=+\infty$.
	It follows that the random walk
	$x_1'(t)$ with transition probabilities \eqref{eq:d_transition_probabilities}
	is indeed the process $\tilde H_1^{\leftarrow}$.
	This completes the proof.
\end{proof}

Let us similarly write down the 
transition probabilities for the random
walk $\tilde H_1^{\rightarrow}$ in the chamber
$\left\{ y_1(t)>y_2(t) \right\}_{0\le t\le M}$ started
from an initial condition $y_1'(0)$.
For $b,c\in \mathbb{Z}_{\ge0}$, denote:
\begin{equation}
	\label{eq:u_quantities}
	\mathtt{u}^0_{b,c}
	=
	\frac{(1+\beta s_1^2 q^b)(1-q^{b+1-c})}{\beta(s_0^2-s_1^2 q^{b-c})(1-q^{b+1})}
	,
	\qquad 
	\mathtt{u}^1_{b,c}
	=
	\frac
	{
		(1-s_1^2 q^{b-1})(1-q^{b-c})
	}
	{
		(\beta+q^{b})(s_0^2-s_1^2 q^{b-c-1})
	}
	.
\end{equation}
These quantities take values in $[0,+\infty)$.
There is no $+\infty$ values since there is no blocking
because $a=+\infty$ or, equivalently, $y_0 = + \infty$.
In particular, we have
$\mathtt{u}^0_{b,b+1}=\mathtt{u}^1_{b,b}=0$,
corresponding to pushing up rules as in
\Cref{sub:discrete_bij_particle_systems_UP}.

Fix a trajectory of the two-particle system
$\{(y_1(t),y_2(t))\}_{t\ge0}$, where $y_1(t)>y_2(t)$.
Moreover, fix $y_1'(0)$
such that 
$y_1'(0)\ge y_1(0)$.
The process $\tilde H_1^{\rightarrow}$
is 
a random walk $y_1'(t)$ with
$y_1'(t)\ge y_1(t)$ for all $t$,
which runs in forwards time $t=0,1,\ldots $.
See \Cref{fig:rewriting_history_for_trajectories},
right with $n=1$ and $y_0=+\infty$,
for an illustration.
The walk starts from $y_1'(0)$ and,
for the transition $t\to t+1$,
it takes steps $0$ or $+1$ with 
probabilities
\begin{equation}
	\label{eq:u_transition_probabilities}
	\frac{\mathtt{u}^{j_1}_{b,c}}{1+\mathtt{u}^{j_1}_{b,c}}
	\quad \textnormal{and}\quad 
	\frac{1}{1+\mathtt{u}^{j_1}_{b,c}},
\end{equation}
respectively, 
where
\begin{equation*}
	b=y_1'(t)-y_2(t)-1
	,
	\qquad
	c=y_1(t+1)-y_2(t+1)-1
	,
	\qquad
	j_1=y_2(t+1)-y_2(t)
	,
\end{equation*}
as in \Cref{fig:Dn_Un_for_particles}, right, with $a=+\infty$.
The next statement is proven in the same way as \Cref{prop:d_first_particle_result}:
\begin{proposition}
	\label{prop:u_first_particle_result}
	Let $y_1'(0)>y_2(0)$ be fixed,
	and take $y_1(0)$, with $y_2(0)<y_1(0)\le y_1'(0)$,
	to be random and with distribution given by 
	\begin{equation*}
		\varphi_{q,s_1/s_0^2,s_1^2}(y_1(0)-y_2(0)-1\mid y_1'(0)-y_2(0)-1).
	\end{equation*}
	Let $\mathbf{y}(t)$ be the two-particle system
	with initial condition $(y_1(0),y_2(0))$
	and parameters $(s_1,s_0)$
	satisfying \eqref{eq:params_for_2_particles}.
	Given the trajectory $\{\mathbf{y}(t)\}_{\ge0}$,
	construct the random walk $\{y_1'(t)\}_{t\ge0}$
	from the initial condition $y_1'(0)$ with
	$y_1'(t)\ge y_1(t)$ and transition probabilities
	given by \eqref{eq:u_transition_probabilities}.
    Then,  the joint distribution of the new process
	$\{(y_1'(t),y_2(t))\}_{t\ge0}$
	is equal to the joint distribution of the 
	process $\mathbf{x}(t)$ with parameters
	$(s_0,s_1)$ with initial condition 
	$y_1'(0)>y_2(0)$.
\end{proposition}

We generalize a certain parameter symmetry for the partincle systems using the 
couplings between 
the two-particle systems 
$\mathbf{x}(t)$ with parameters $(s_0,s_1)$ and
$\mathbf{y}(t)$ with parameters $(s_1,s_0)$
in \Cref{prop:d_first_particle_result,prop:u_first_particle_result}.
For instance, take both 
$\mathbf{x}(t)$
and
$\mathbf{y}(t)$
with step initial conditions $\mathbf{x}_{step}$,
i.e.~$x_1(0)=y_1(0)=-1$ and
$x_2(0)=y_2(0)=-2$.
Then, the distributions of the 
trajectories of the second particle,
$\{x_2(t)\}_{t\ge0}$
and 
$\{y_2(t)\}_{t\ge0}$,
are the same. In particular, one may show that the distribution of the the second particle is a \emph{symmetric} function on the parameters $(s_0, s_1)$, without using a coupling argument, making the previous statement true. On the other hand, 
this symmetry breaks when the initial configuration is not $\mathbf{x}_{step}$.
The 
following statement restores 
(in a stochastic way)
the symmetry for 
general initial configurations:

\begin{corollary}
	\label{cor:second_particle_symmetry}
	Fix $x_1(0)>x_2(0)$. Let $\mathbf{x}(t)$ be the two-particle system
	with parameters $(s_0,s_1)$, $|s_0|>|s_1|$, started from $(x_1(0),x_2(0))$.
	Let $y_1'(0)$, where $x_2(0)<y_1'(0)\le x_1(0)$ be 
	random and distributed as
	\begin{equation*}
		\varphi_{q,s_1/s_0^2,s_1^2}(y_1'(0)-x_2(0)-1\mid x_1(0)-x_2(0)-1).
	\end{equation*}
	Start $\mathbf{y}(t)$ with parameters 
	$(s_1,s_0)$ from the random initial configuration
	$(y_1'(0),x_2(0))$. Then
	the distributions of the trajectories
	second particle in both systems,
	$\{x_2(t)\}_{t\ge0}$
	and 
	$\{y_2(t)\}_{t\ge0}$,
	coincide.
\end{corollary}
\begin{proof}
	This follows from the history rewriting processes
	in either of
	\Cref{prop:d_first_particle_result,prop:u_first_particle_result}
	since both of these processes keep the trajectory of the 
	second particle intact.
\end{proof}

Let us specialize to $q=0$.
We consider the random walk
$\tilde H_1^{\leftarrow}$ 
for rewriting history from future to past,
described before \Cref{prop:d_first_particle_result}.
One may similarly specialize $\tilde H_1^{\rightarrow}$, 
but we omit this for brevity.
The quantities
\eqref{eq:d_quantities}, 
for $q=0$, specialize as follows
\begin{equation}
	\label{eq:d_quantities_q_is_0}
	\mathtt{d}^0_{b,c}
	\big\vert_{q=0}
	=
	\begin{cases}
		+\infty,& c=0;\\
		\frac{1}{\beta s_0^2},&1\le c\le b-1;\\
		\frac{1}{\beta(s_0^2-s_1^2)},&c=b;\\
		0,&c=b+1,
	\end{cases}
	\qquad 
	\mathtt{d}^1_{b,c}
	\big\vert_{q=0}
	=
	\begin{cases}
		\frac{1-s_0^2}{(1+\beta)s_0^2},&c=0;\\
		\frac{1}{\beta s_0^2}&1\le c\le b-2;\\
		\frac{1}{\beta(s_0^2-s_1^2)}&c=b-1;\\
		0,&c=b.
	\end{cases}
\end{equation}
It follows that $x_1'(t)$, the process in reversed time living in the chamber 
$x_1(t)\ge x_1'(t)>x_2(t)$, is a simple
random walk with location-dependent transition probabilities. 
Namely, in the \emph{bulk} of the 
chamber, it takes a step
$-1$ with probability
$\frac{1}{1+1 / (\beta s_0^2)}
=\frac{\beta s_0^2}{1+\beta s_0^2}$
and a step $0$
with the complementary probability
$\frac{1}{1+\beta s_0^2}$.
At the boundary of the chamber, 
the probabilities need to be suitably modified, see
\Cref{fig:down_dynamics_for_q_is_0} 
for an illustration of all cases
which are determined from \eqref{eq:d_quantities_q_is_0}.

\begin{figure}[htb]
	\centering
	\includegraphics[width=.65\textwidth]{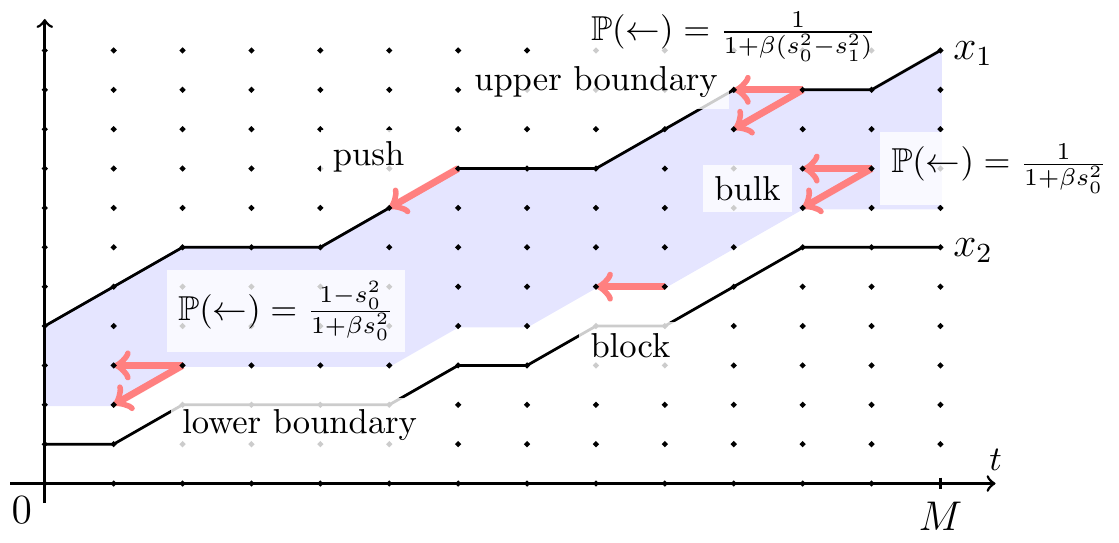}
	\caption{Transition probabilities for the simple
		random walk $x_1'(t)$ for rewriting history from future to past
		in the specialization $q=0$.
		This walk runs in reverse time inside the chamber 
		$x_1(t)\ge x_1'(t)>x_2(t)$.
		Its transition probabilities differ on 
		the boundary
		of the chamber, compared to the bulk.
		In the figure, we only 
		list the probability of the step $0$,
		with the probability of the step $-1$ determined by the complement formula.}
	\label{fig:down_dynamics_for_q_is_0}
\end{figure}

\section{Bijectivisation and rewriting history in continuous time}
\label{sec:ind_bij_cont_time_new}

In this section, we construct
the rewriting history processes for $q$-TASEP and TASEP
evolving in
continuous time using a bijectivisation.
We view the $q$-TASEP as a continuous time limit
of the system with $J=1$ and $u_i=-\beta s_i$
considered in \Cref{sec:discrete_time_bijectivisation}.
Thus, we deal with
is a specialization
of the independent bijectivisation in \Cref{sub:discrete_bij_subsection_8_1}.

\subsection{Limit to continuous time $q$-TASEP}
\label{sub:YBE_for_qTASEP}

Let us take a continuous time limit of the vertex model
from \Cref{sub:discrete_bij_particle_systems}.
Recall that we had set $J=1$ and $u_i=-\beta s_i$, where $\beta>0$.
For this model, the stochastic vertex weights 
and the cross vertex weights for the 
vertical Yang-Baxter equation are given in 
\Cref{fig:vertex_weights_for_bij}.
Now, consider the following limit of the parameters:
\begin{enumerate}[$\bullet$]
	\item First, set $\beta s_i^2=\varepsilon \alpha_i$, for all $i$, 
		where $\alpha_i>0$,
		and $\varepsilon>0$ is fixed for now.
	\item 
		Send $\beta \to+\infty$ and $s_i\to 0$ so that
		$\varepsilon \alpha_i>0$ is fixed.
	\item After this,
		take the limit as $\varepsilon\to0$ 
		and rescale time from discrete to continuous
		as $t=\lfloor \mathsf{t}/\varepsilon \rfloor $,
		where $\mathsf{t}\in \mathbb{R}_{\ge0}$ is the new
		continuous time.
\end{enumerate}
These operations turn the stochastic higher
spin six vertex model 
into the continuous time stochastic $q$-Boson model
\cite{SasamotoWadati1998}, \cite{BorodinCorwin2011Macdonald},
\cite{BorodinCorwinPetrovSasamoto2013} 
with inhomogeneous rates $\alpha_i$. 
Indeed,
the $\varepsilon\to0$
expansions 
of all the vertex weights
in the last operation
are given in 
\Cref{fig:qBoson}.
We see that the cross vertex weights do not depend on $\varepsilon$,
while the weights
in the vertex model itself
are of order $O(\varepsilon)$ or $1-O(\varepsilon)$. 
The weights of type $(g,0;g-1,1)$
correspond to the jump rates in the $q$-Boson model. More precisely,
each stack of $g_n$ vertical arrows at location $n\in \mathbb{Z}_{\ge0}$
(with the agreement $g_0=+\infty$)
emits one horizontal arrow at rate $\alpha_n(1-q^{g_n})$.
This horizontal arrow instantaneously 
travels horizontal distance $1$ and joins the next
stack of arrows at location $n+1$. This is because the 
probability to travel distance at least $2$ is proportional to 
$\varepsilon^2$, which is negligible in the continuous time limit.

The $q$-Boson system corresponds to the
continuous
time $q$-TASEP, where the particle $x_n$ has speed
$\alpha_{n-1}$ for $n\in \mathbb{Z}_{\ge0}$,
via the gap-particle transformation
given in \Cref{def:gap_particle_transform}. That is, each 
$x_n$
jumps to the right by 
one at rate $\alpha_{n-1}(1-q^{g_{n-1}})$,
where $g_{n-1}=x_{n-1}-x_n-1$.

\begin{figure}[htb]
	\centering
	\includegraphics[height=.55\textwidth]{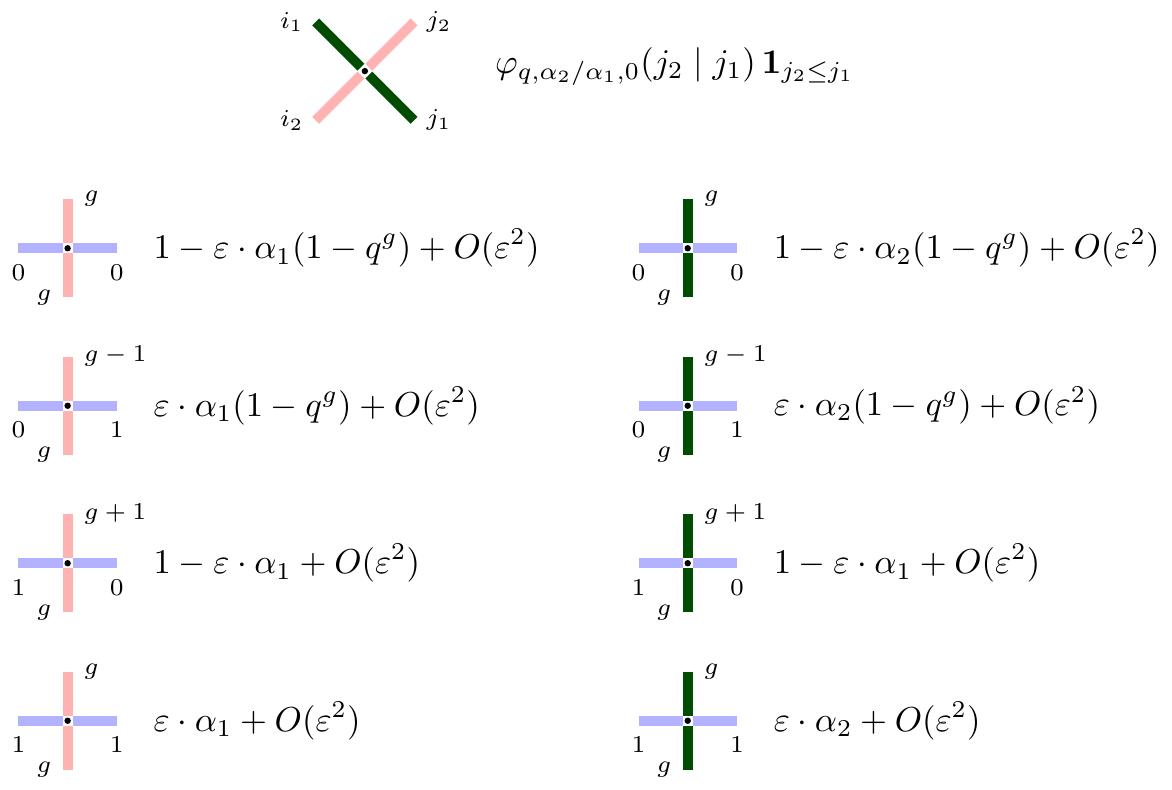}
	\caption{Expansions as $\varepsilon\to0$ of the vertex weights
	entering the Yang-Baxter equation for the continuous
	time $q$-TASEP, where the time is scaled 
	proportionally to $\varepsilon^{-1}$.
	The cross vertex weights are nonnegative
	when $\alpha_1\ge \alpha_2>0$.}
	\label{fig:qBoson}
\end{figure}

\subsection{Independent bijectivisation in continuous time}
\label{sub:continuous_time_bij}

Let us now write down the $\varepsilon\to 0$ expansions, 
under the setting of \Cref{sub:discrete_bij_subsection_8_1},
of the 
transition probabilities
\begin{equation}
	\label{eq:cont_time_things_to_expand}
	p^{\downarrow}_{i_1,j_1}[*\to 0]
	,\qquad 
	p^{\downarrow}_{i_1,j_1}[*\to 1]
	,\qquad 
	p^{\uparrow}_{i_1,j_1}[*\to 0]
	,\qquad 
	p^{\uparrow}_{i_1,j_1}[*\to 1]
\end{equation}
given by
\eqref{eq:bij_independent}, after performing the 
first two steps of the specialization from 
\Cref{sub:YBE_for_qTASEP}.
Here and below ``$*$'' means that 
the transition does not depend on the previous state
as much as possible, which is a feature of the independent
bijectivisation. However, see the blocking and pushing mechanisms
described in 
\Cref{sub:discrete_bij_particle_systems,sub:discrete_bij_particle_systems_UP}.
The resulting expansions of \eqref{eq:cont_time_things_to_expand} would depend on $q$ and
the spectral parameters 
$\alpha_1\ge \alpha_2 > 0$.
We also continue to 
use the notation
\eqref{eq:81_notation_1}--\eqref{eq:81_notation_3}
for the boundary conditions in the Yang-Baxter equation.

\begin{proposition}
	\label{prop:expansions_indep_bij}
	Given the conventions explained before the proposition, we have the following
	$\varepsilon\to0$
	expansions of the down transition probabilities:
	\begin{align}
		\label{eq:expansion_00_down}
		p^{\downarrow}_{00}[*\to 1]
		&
		=
		\mathbf{1}_{c=b+1}+
		\varepsilon\ssp \mathbf{1}_{c\le b}\ssp
		\frac{(\alpha_1-\alpha_2 q^{b-c})(1-q^{a+b+1-c})(1-q^c)}{1-q^{b-c+1}}
		+O(\varepsilon^2).
		\\
		\label{eq:expansion_01_down}
		p^{\downarrow}_{01}[*\to 1]
		&
		=
		\mathbf{1}_{c=b}+
		\varepsilon\ssp
		\mathbf{1}_{c\le b-1}\ssp
		\frac{(\alpha_1-\alpha_2 q^{b-c-1})(1-q^{a+b-c})}{1-q^{b-c}}
		+O(\varepsilon^2);
		\\
		\label{eq:expansion_10_down}
		p^{\downarrow}_{10}[*\to 1]
		&
		=
		\mathbf{1}_{c=b+1}+
		\varepsilon\ssp
		\mathbf{1}_{c\le b}
		\ssp
		\frac{(\alpha_1-\alpha_2 q^{b-c})(1-q^c)}{1-q^{b+1-c}}
		+O(\varepsilon^2);
		\\
		\label{eq:expansion_11_down}
		p^{\downarrow}_{11}[*\to 1]
		&
		=
		\mathbf{1}_{c=b}+
		\varepsilon\ssp
		\mathbf{1}_{c\le b-1}
		\ssp
		\frac{\alpha_1-\alpha_2 q^{b-c-1}}{1-q^{b-c}}+O(\varepsilon^2).
	\end{align}
	For the up transition probabilities,
	we have
	\begin{align}
		\label{eq:expansion_00_up}
		p^{\uparrow}_{00}[*\to 1]
		&
		=
		\mathbf{1}_{c=b+1}
		+\varepsilon\ssp
		\mathbf{1}_{c\le b}
		\ssp
		\frac{(\alpha_1-\alpha_2 q^{b-c})(1-q^a)(1-q^{b+1})}{1-q^{b-c+1}}
		+O(\varepsilon^2);
		\\
		\label{eq:expansion_01_up}
		p^{\uparrow}_{01}[*\to 1]
		&
		=
		\mathbf{1}_{c=b}
		+\varepsilon
		\ssp
		\mathbf{1}_{c\le b-1}
		\ssp
		\frac{(\alpha_1-\alpha_2 q^{b-c-1})(1-q^a)}{1-q^{b-c}}
		+O(\varepsilon^2);
		\\
		\label{eq:expansion_10_up}
		p^{\uparrow}_{10}[*\to 1]
		&
		=
		\mathbf{1}_{c=b+1}
		+\varepsilon
		\ssp
		\mathbf{1}_{c\le b}
		\ssp
		\frac{(\alpha_1-\alpha_2 q^{b-c})(1-q^{b+1})}{1-q^{b-c+1}}
		+O(\varepsilon^2);
		\\
		\label{eq:expansion_11_up}
		p^{\uparrow}_{11}[*\to 1]
		&
		=
		\mathbf{1}_{c=b}+
		\varepsilon\ssp
		\mathbf{1}_{c\le b-1}
		\ssp
		\frac{\alpha_1-\alpha_2 q^{b-c-1}}{1-q^{b-c}}
		+O(\varepsilon^2)
		.
	\end{align}
	In all cases, the complementary probabilities
	follow and are determined by the complement formula, 
	e.g.~$p^{\downarrow}_{i_1,j_1}[*\to 0]=1-p^{\downarrow}_{i_1,j_1}[*\to 1]$.
	The parameters $a,b,c$ in 
	formulas 
	\eqref{eq:expansion_00_down}--\eqref{eq:expansion_11_up}
	satisfy $a,b\ge0$ and $0\le c\le b+\min(a+i_1,1)-j_1$.
\end{proposition}
\begin{proof}
	These expansions are obtained in a straightforward way using 
	\Cref{def:independent_bijectivisation} and
	the explicit formulas 
	for the vertex weights after performing the first two steps
	of the specialization from \Cref{sub:YBE_for_qTASEP}; see \Cref{fig:qBoson}.
\end{proof}

In continuous time, note that 
during each time moment there is \emph{at most 
one jump} of any of the particles, both before and after 
applying a
rewriting history operator. This means that we can eliminate the boundary conditions $i_1=j_1=1$ 
which never occur in continuous time.

Next, consider the case 
when $i_1+j_1=1$. The order $1$ terms in 
\eqref{eq:expansion_01_down}--\eqref{eq:expansion_10_down}
and 
\eqref{eq:expansion_01_up}--\eqref{eq:expansion_10_up}
occur only when both the old trajectory of the $n$-th particle, 
and one of the particles
$x_{n\pm1}$ jump at the same time, which is impossible in continuous time.
The order
$\varepsilon$ terms in 
\eqref{eq:expansion_01_down}--\eqref{eq:expansion_10_down}
and 
\eqref{eq:expansion_01_up}--\eqref{eq:expansion_10_up}
correspond to both the new trajectory of $x_n$ and one of the 
particles $x_{n\pm 1}$ jumping at the same time, 
which is also impossible in continuous time. It follows
multiple events with probability of order $\varepsilon$ occur when $i_1+j_1=1$. Thus, we have that the case $i_1+j_1=1$ cannot lead to new jump in the new trajectory of the $n$-th particle since
such a jump would be an event of probability $O(\varepsilon^2)$,
which vanishes in the continuous time limit.

Therefore, the independent bijectivisation in the continuous time
limit is \emph{completely determined} by the expansions
\eqref{eq:expansion_00_down} and 
\eqref{eq:expansion_00_up} for $i_1=j_1=0$.
In the rest of the current \Cref{sec:ind_bij_cont_time_new},
we use this fact to describe the rewriting history processes
$\tilde H_n^{\leftarrow}$ 
and 
$\tilde H_n^{\rightarrow}$ 
for the continuous time $q$-TASEP and TASEP.

\subsection{Rewriting history from future to past for a parameter swap in $q$-TASEP}
\label{sub:rewriting_history_qTASEP_new}

Here and in the next
\Cref{sub:rewriting_history_qTASEP_UP_new},
we describe the 
 $q$-TASEP's
rewriting history
processes 
$\tilde H_n^{\leftarrow}$ 
and 
$\tilde H_n^{\rightarrow}$ 
from \Cref{def:rewriting_history_operators}.
These processes are based on the independent bijectivisation and are
determined by the 
jump rates coming from
\eqref{eq:expansion_00_down} and 
\eqref{eq:expansion_00_up}, respectively.

Let us start with the process
$\tilde H_n^{\leftarrow}$ 
of rewriting history from future to past. 
Fix $\mathsf{M}\in \mathbb{R}_{\ge0}$, $n\in \mathbb{Z}_{\ge1}$,
and two speeds $\alpha_{n-1}>\alpha_n>0$.
Assume we have three trajectories of consecutive particles,
\begin{equation*}
	x_{n-1}(\mathsf{t})>x_{n}(\mathsf{t})>x_{n+1}(\mathsf{t}),\qquad 
	0\le \mathsf{t}\le \mathsf{M}	
\end{equation*}
and a starting point $x_n'(\mathsf{M})$ with 
\begin{equation*}
	x_n(\mathsf{M})\ge x_n'(\mathsf{M})> x_{n+1}(\mathsf{M}).
\end{equation*}
The process
$x_n'(\mathsf{t})$ starts from
$x_n'(\mathsf{M})$ and
runs in reverse time in the chamber
$x_{n}(\mathsf{t})\ge x_{n}'(\mathsf{t})>x_{n+1}(\mathsf{t})$, 
making jumps down by $1$
in continuous time with rate 
\begin{equation}
	\label{eq:rate_n_down_notation}
	\mathrm{rate}^{\leftarrow}_{n;\alpha_{n-1},\alpha_n}=
	\frac{(\alpha_{n-1}-\alpha_n q^{b-c})(1-q^{a+b+1-c})(1-q^c)}{1-q^{b-c+1}},
\end{equation}
where
(cf. \Cref{fig:Dn_Un_for_particles}, left)
\begin{equation}
	\label{eq:down_process_cont_time_abc_def_new}
	a \coloneqq x_{n-1}(\mathsf{t}-)-x_n(\mathsf{t}-)-1
	,
	\quad 
	b \coloneqq x_n(\mathsf{t}-)-x_{n+1}(\mathsf{t}-)-1,
	\quad 
	c \coloneqq
	x_n'(\mathsf{t})-x_{n+1}(\mathsf{t})-1
	.
\end{equation}
Note that the rate \eqref{eq:rate_n_down_notation}
is a piecewise constant function of the time $\mathsf{t}$,
and the rate changes whenever one of the particles
$x_{n-1},x_{n}$, or $x_{n+1}$ makes a jump.
The particle $x_n'$ is \emph{blocked}
from jumping down by $x_{n+1}$, i.e.~$\mathrm{rate}^{\leftarrow}_{n;\alpha_{n-1},\alpha_n}=0$,
when $c=0$.
The particle $x_n'$ is \emph{pushed down} by a jump of $x_n$, i.e.~$\mathrm{rate}^{\leftarrow}_{n;\alpha_{n-1},\alpha_n}=+\infty$, 
when $c=b+1$.
Note that $a+b+1-c$ is always positive, so 
the factor $1-q^{a+b+1-c}$ does not vanish.
The blocking and pushing
mechanisms are similar to the discrete time case from \Cref{sub:discrete_bij_particle_systems}.
See \Cref{fig:down_new_dynamics}
for an illustration.

\begin{figure}[htpb]
	\centering
	\includegraphics[width=.6\textwidth]{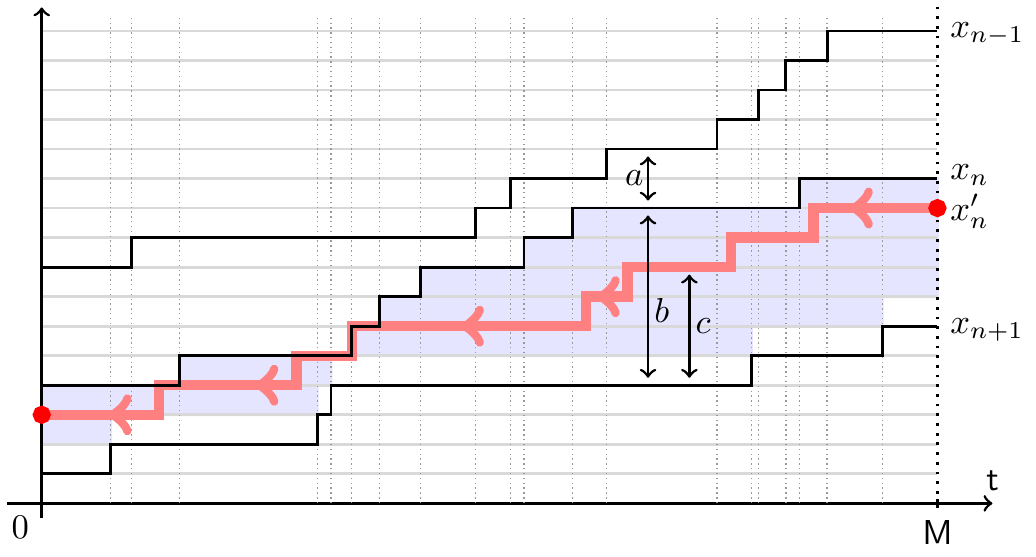}
	\caption{The process
	$\tilde H_n^{\leftarrow}$ 
	of rewriting history from future to past 
	for the continuous time $q$-TASEP 
	and the parameter
	swap $\alpha_{n-1}\leftrightarrow \alpha_n$.
	The allowed chamber for the 
	new trajectory $x_n'(\mathsf{t})$ is shaded.
	The quantities \eqref{eq:down_process_cont_time_abc_def_new}
	indicated in the figure for a particular time interval are
	equal to $a=1,b=5,c=3$ so that the jump rate \eqref{eq:rate_n_down_notation}
	at that particular time interval
	is equal to $(\alpha_{n-1}-\alpha_n q^{2})(1-q^4)$.}
	\label{fig:down_new_dynamics}
\end{figure}

The process 
$\tilde H_n^{\leftarrow}$ described above
produces a coupling of the trajectories of the 
$q$-TASEPs with speed sequences differing by the swap
$\alpha_{n-1}\leftrightarrow \alpha_n$.
Recall that, for the $q$-TASEP, the Markov swap operator 
$\tilde P^{(n)}$ 
depending on the ratio $0\le \alpha_n/\alpha_{n-1}\le 1$
acts on $\mathscr{X}$
by moving the single particle $x_n$ into a random new location
$x_n'$ with probability
\begin{equation}
	\label{eq:qTASEP_swap_cont}
	\varphi_{q,\alpha_{n}/\alpha_{n-1},0}(x_n'-x_{n+1}-1 \mid x_n-x_{n+1}-1).
\end{equation}

Let $\mathbf{x}(0)\in \mathscr{X}$ be an initial condition,
and $\boldsymbol \alpha=(\alpha_0,\alpha_1,\ldots )$ be a 
sequence of particle speeds such that $\alpha_{n-1}>\alpha_n$.
Let $\{\mathbf{x}(\mathsf{t})\}_{0\le \mathsf{t}\le \mathsf{M}}$
be the continuous time $q$-TASEP started from $\mathbf{x}(0)$
with speeds $\boldsymbol\alpha$. 
Let also $\{\mathbf{y}(\mathsf{t})\}_{0\le \mathsf{t}\le \mathsf{M}}$
be the continuous time 
$q$-TASEP started from the random initial condition 
$\delta_{\mathbf{x}(0)}\tilde P^{(n)}$ and evolving with the speeds 
$\sigma_{n-1}\boldsymbol\alpha$, where $\sigma_{n-1}$
is the elementary transposition 
$\alpha_{n-1}\leftrightarrow \alpha_n$.

\begin{proposition}
	\label{prop:cont_time_qTASEP_down_result}
	Given the notation above,
	let 
	$\mathbf{x}'(\mathsf{M})$
	be obtained from 
	$\mathbf{x}(\mathsf{M})$
	by the action of $\tilde P^{(n)}$, that is, by randomly moving $x_n(\mathsf{M})$ 
	to $x_n'(\mathsf{M})$ with probability 
	\eqref{eq:qTASEP_swap_cont}.
	Given the trajectories 
	of the particles $x_j(\mathsf{t})$, $j=n-1,n,n+1$,
	replace the old trajectory $x_n(\mathsf{t})$ by the 
	new one $x_n'(\mathsf{t})$ constructed from the process
	$\tilde H_n^{\leftarrow}$
	started from $x_n'(\mathsf{M})$ and running in reverse time. Then, the resulting trajectory of the whole process
	$\{x_1(\mathsf{t}),\ldots,x_{n-1}(\mathsf{t}),x_n'(\mathsf{t}),x_{n+1}(\mathsf{t}),\ldots  \}_{0\le \mathsf{t}\le \mathsf{M}}$
	is equal in distribution to the trajectory of the process
	$\{\mathbf{y}(\mathsf{t})\}_{0\le \mathsf{t}\le \mathsf{M}}$.
\end{proposition}
\begin{proof}
	This is
	a continuous time limit of
	the general \Cref{thm:coupling_trajectories_with_Pn_swap}
	and \Cref{cor:coupling_trajectories_with_Pn_swap}.
\end{proof}

\begin{remark}[TASEP specialization of $\tilde H_n^{\leftarrow}$]
	\label{rmk:down_new_swap_TASEP}
	The $q$-TASEP turns into the TASEP when $q=0$.
	In that case,
	the dynamics 
	$\tilde H_n^{\leftarrow}$ simplifies. Namely, 
	the blocking and pushing mechanisms stay the same, and the 
	jump rates 
	\eqref{eq:rate_n_down_notation}
	become
	\begin{equation}
		\label{eq:rate_n_down_notation_TASEP}
		\mathrm{rate}^{\leftarrow}_{n;\alpha_{n-1},\alpha_n}\Big\vert_{q=0}=
		\alpha_{n-1}-\alpha_n \mathbf{1}_{b=c}.
	\end{equation}
	Thus, the process $x_n'(\mathsf{t})$
	is a Poisson random walk in the chamber
	$x_{n}(\mathsf{t})\ge x_{n}'(\mathsf{t})>x_{n+1}(\mathsf{t})$,
	running in reverse time and jumping down with rate $\alpha_{n-1}$
	in the bulk and $\alpha_{n-1}-\alpha_n$ at the top boundary of the chamber.
	
	Clearly, \Cref{prop:cont_time_qTASEP_down_result} for $q=0$ holds for the process 
	$\tilde H_n^{\leftarrow}$
	with the jump 
	rates \eqref{eq:rate_n_down_notation_TASEP}.
	This proposition for $q=0$ and $n=1$
	immediately implies \Cref{thm:2cars_general_intro} from the Introduction.
\end{remark}

\subsection{Rewriting history from past to future for a parameter swap in $q$-TASEP}
\label{sub:rewriting_history_qTASEP_UP_new}

Let us now consider the process
$\tilde H_n^{\rightarrow}$ 
of rewriting history from past to future.
Fix $n\in \mathbb{Z}_{\ge1}$
and two speeds $\alpha_{n-1}>\alpha_n>0$.
Assume we have three trajectories of consecutive particles,
\begin{equation*}
	y_{n-1}(\mathsf{t})>y_{n}(\mathsf{t})>y_{n+1}(\mathsf{t}),
\end{equation*}
where $\mathsf{t}$ runs over $\mathbb{R}_{\ge0}$,
and a starting point $y_n'(0)$ so that
$y_n(0)\le y_n'(0)< y_{n-1}(0)$.
The process
$y_n'(\mathsf{t})$ 
starts from $y_n'(0)$ and
runs in forward time $\mathsf{t}\in \mathbb{R}_{\ge0}$
in the chamber
$y_n(\mathsf{t})\le y_n'(\mathsf{t})< y_{n-1}(\mathsf{t})$.
In continuous time, the location of $y_n'(\mathsf{t})$ jumps up by $1$ with rate
\begin{equation}
	\label{eq:rate_n_up_notation}
	\mathrm{rate}^{\rightarrow}_{n;\alpha_{n-1},\alpha_n}=
	\frac{(\alpha_{n-1}-\alpha_n q^{b-c})(1-q^a)(1-q^{b+1})}{1-q^{b-c+1}},
\end{equation}
where
(cf. \Cref{fig:Dn_Un_for_particles}, right)
\begin{equation}
	\label{eq:down_process_cont_time_abc_def_UP_new}
	a \coloneqq y_{n-1}(\mathsf{t}-)-y_n'(\mathsf{t}-)-1
	,
	\quad 
	b \coloneqq y_n'(\mathsf{t}-)-y_{n+1}(\mathsf{t}-)-1,
	\quad 
	c \coloneqq
	y_n(\mathsf{t})-y_{n+1}(\mathsf{t})-1
	.
\end{equation}
The rate \eqref{eq:rate_n_up_notation}
is a piecewise constant function of the time $\mathsf{t}$,
and the rate changes whenever one of the particles
$y_{n-1},y_{n}$, or $y_{n+1}$ makes a jump. 
The particle $y_n'$ is \emph{blocked} from jumping up by $y_{n-1}$,
i.e.~$\mathrm{rate}^{\rightarrow}_{n;\alpha_{n-1},\alpha_n}=0$, 
when $a=0$.
The particle $y_n'$ is \emph{pushed up} by a jump of $y_n$,
i.e.~$\mathrm{rate}^{\rightarrow}_{n;\alpha_{n-1},\alpha_n}=+\infty$,
when $c=b+1$.
The blocking and pushing
mechanisms are similar to the discrete time ones from \Cref{sub:discrete_bij_particle_systems_UP}.
See \Cref{fig:up_new_dynamics}
for an illustration.

\begin{figure}[htb]
	\centering
	\includegraphics[width=.6\textwidth]{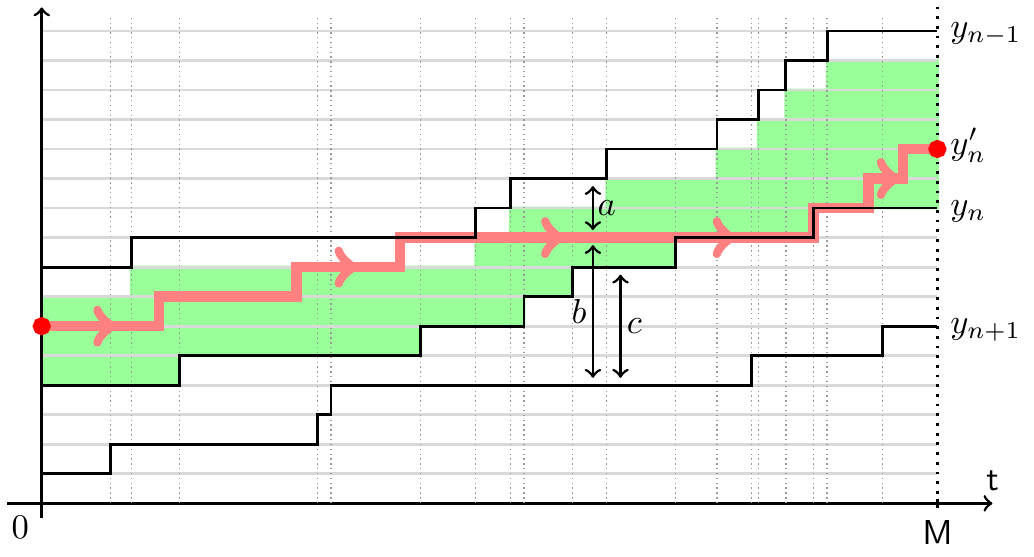}
	\caption{The process
	$\tilde H_n^{\rightarrow}$ 
	of rewriting history from past to future
	for the continuous time $q$-TASEP 
	and the parameter
	swap $\alpha_{n-1}\leftrightarrow \alpha_n$.
	The allowed chamber for the 
	new trajectory $y_n'(\mathsf{t})$ is shaded.
	The quantities \eqref{eq:down_process_cont_time_abc_def_UP_new}
	indicated in the figure for a particular time interval are
	equal to $a=1,b=4,c=3$ so that the jump rate \eqref{eq:rate_n_up_notation} at the indicated time
	is equal to $(\alpha_{n-1}-\alpha_{n} q)(1-q)(1-q^{5})/(1-q^{2})$.}
	\label{fig:up_new_dynamics}
\end{figure}

The process 
$\tilde H_n^{\rightarrow}$ 
produces a coupling of the trajectories of the 
$q$-TASEPs in which speeds differ by the swap
$\alpha_{n-1}\leftrightarrow \alpha_n$.
Recall the notation before \Cref{prop:cont_time_qTASEP_down_result}. 
The process $\mathbf{x}(\mathsf{t})$ is the continuous time $q$-TASEP 
started from a fixed initial configuration $\mathbf{x}(0)\in \mathscr{X}$ and evolves
with the particle speeds
$\boldsymbol\alpha$ so that $\alpha_{n-1}>\alpha_{n}$,
Also, the process
$\mathbf{y}(\mathsf{t})$ is the continuous time $q$-TASEP 
started from the random initial condition
$\delta_{\mathbf{x}(0)}\tilde P^{(n)}$
and evolves with the particle speeds 
$\sigma_{n-1}\boldsymbol\alpha$.

\begin{proposition}
	\label{prop:cont_time_qTASEP_up_result}
	Fix $n\ge1$.
	Given the trajectories 
	of the particles $y_j(\mathsf{t})$, $j=n-1,n,n+1$,
	in $\{\mathbf{y}(\mathsf{t})\}_{\mathsf{t}\ge0}$,
	replace the old trajectory $y_n(\mathsf{t})$ by the 
	new
	one $y_n'(\mathsf{t})$ constructed from the process
	$\tilde H_n^{\rightarrow}$
	started from $y_n'(0)=x_n(0)$. Then, the resulting trajectory of the whole process
	$\{y_1(\mathsf{t}),\ldots,y_{n-1}(\mathsf{t}),y_n'(\mathsf{t}),y_{n+1}(\mathsf{t}),\ldots  \}_{0\le \mathsf{t}\le \mathsf{M}}$
	is equal in distribution with the trajectory of the process
	$\{\mathbf{x}(\mathsf{t})\}_{0\le \mathsf{t}\le \mathsf{M}}$.
\end{proposition}
\begin{proof}
	This statement
	is also
	a continuous time limit of
	the general \Cref{thm:coupling_trajectories_with_Pn_swap}
	and \Cref{cor:coupling_trajectories_with_Pn_swap}, as it was for \Cref{prop:cont_time_qTASEP_down_result}.
\end{proof}

\begin{remark}[TASEP specialization of $\tilde H_n^{\rightarrow}$]
	\label{rmk:up_new_swap_TASEP}
	The $q$-TASEP turns into the TASEP when $q=0$, and
	the dynamics 
	$\tilde H_n^{\rightarrow}$ simplifies in that case. Namely, 
	the blocking and pushing mechanisms stay the same, and the 
	jump rates 
	\eqref{eq:rate_n_up_notation}
	become
	\begin{equation}
		\label{eq:rate_n_up_notation_TASEP}
		\mathrm{rate}^{\rightarrow}_{n;\alpha_{n-1},\alpha_n}\Big\vert_{q=0}
		=\alpha_{n-1}-\alpha_n\mathbf{1}_{b=c}.
	\end{equation}
	Note that here the meaning of $b,c$ differs from that in 
	\eqref{eq:rate_n_down_notation_TASEP}, as the process evolves 
	forward in time $\mathsf{t}$ instead of backwards in time.
	Thus, the process $y_n'(\mathsf{t})$
	is a Poisson random walk in the chamber
	$y_n(\mathsf{t})\le y_n'(\mathsf{t})< y_{n-1}(\mathsf{t})$
	running in forward in time and jumping up with rate $\alpha_{n-1}$
	in the bulk and with rate $\alpha_{n-1}-\alpha_n$ at the bottom boundary of the chamber.
	
	Clearly, \Cref{prop:cont_time_qTASEP_up_result} holds when $q=0$ for the 
	process $\tilde H_n^{\rightarrow}$ with the jump 
	rates \eqref{eq:rate_n_up_notation_TASEP}.
	The case $n=1$ of this proposition with $q=0$
	immediately implies
	\Cref{thm:Poisson_discrete_intro} from the Introduction.
\end{remark}

\section{Limit of rewriting history processes to equal particle speeds}
\label{sec:limit_equal_speeds_new}

In this section, we 
obtain the limits of the rewriting history processes 
from \Cref{sec:ind_bij_cont_time_new}
as the particle speeds $\alpha_i$ become equal.
The corresponding rewriting history processes are
powered by the 
independent bijectivisation for the
intertwining relation 
from \Cref{thm:mapping_qTASEP_TASEP_back_general_IC}.

\subsection{Space of $q$-TASEP trajectories}
\label{sub:space_of_trajectories_def}

Markov operators for rewriting history, as well as their limits
which are continuous time Markov semigroups,
act on the space of continuous time trajectories
which we now describe.

\begin{definition}
	\label{def:cont_time_space_of_trajectories_new}
	Let
	$\mathscr{X}^{[0,\mathsf{M}]}$
	be the \emph{space of trajectories of the continuous time $q$-TASEP}
	over time $0\le \mathsf{t}\le \mathsf{M}$.
	By definition, this space 
	consists of trajectories
	$\mathfrak{x}=\{\mathbf{x}(\mathsf{t})\}_{0\le \mathsf{t}\le\mathsf{M}}$
	satisfying the following conditions:
	\begin{enumerate}[$\bullet$]
		\item For all $\mathsf{t}$ we have
			$\mathbf{x}(\mathsf{t})\in \mathscr{X}$ (recall \Cref{def:state_spaces}).
		\item There exists $N$ such that
			$x_n(\mathsf{t})=-n$ for all $n>N$ and $0\le \mathsf{t}\le\mathsf{M}$.
		\item The trajectory of each particle
			$x_n(\mathsf{t})$ is weakly increasing, piecewise constant,
			and makes increments of size $1$.\footnote{The word \emph{jump} refers to 
			jumps under continuous time Markov processes on spaces of trajectories,
			throughout the current 
			\Cref{sec:limit_equal_speeds_new}.
			When a particle in a given trajectory changes its position (which is a piecewise
			constant function of $\mathsf{t}$), 
			we refer to this as an 
			\emph{increment} of this particle's coordinate.}
		\item At any time moment $0\le \mathsf{t}\le \mathsf{M}$,
			there is at most one such increment.
	\end{enumerate}
	For each $\mathfrak{x}\in \mathscr{X}^{[0,\mathsf{M}]}$, 
	we associate the finite set $T_{\mathfrak{x}}\subset
	(0,\mathsf{M})$ of all $\mathsf{t}$ at which 
	some particle has an increment.
	The space $\mathscr{X}^{[0,\mathsf{M}]}$ has a natural
	topology in which $\mathfrak{x}$ and $\mathfrak{x}'$ are close 
	iff 
	$\mathbf{x}(0)=\mathbf{x}'(0)$,
	$\mathbf{x}(\mathsf{M})=\mathbf{x}'(\mathsf{M})$,
	all particles make all of their increments in the same order, 
	and the increments' times $T_{\mathfrak{x}},T_{\mathfrak{x}'}$ are close
	in the corresponding finite-dimensional space. 
	In this topology, $\mathscr{X}^{[0,\mathsf{M}]}$ has countably 
	many connected components each of which may be identified with an open subset 
	of $\mathbb{R}^d$ for a suitable $d$.
\end{definition}

Markov operators $\tilde \Xi$ on $\mathscr{X}^{[0,\mathsf{M}]}$ which we consider
in the current \Cref{sec:limit_equal_speeds_new}
can map a trajectory $\mathfrak{x}$ to a random new trajectory $\mathfrak{x}'$
such that $T_{\mathfrak{x}'}=(T_{\mathfrak{x}}\setminus T^{-})\cup T^{+}$. That is,
$\tilde\Xi$
removes a random subset $T^-\subset T_{\mathfrak{x}}$
of the existing increment times, and adds random new increment times
belonging to $T^+$. The new increment times belong to continuous intervals, and are
chosen randomly from probability densities with respect to the Lebesgue measure.
Therefore, the operator $\tilde \Xi$ is determined by the transition densities
\begin{equation}
	\label{eq:Markov_op_cont_times_prob_dens_new}
	\frac{
	\tilde \Xi(\mathfrak{x},\mathfrak{x}'+d\ssp T^+)}
	{d\ssp T^+},
\end{equation}
which also incorporate the probabilities for the removed 
increment times.

\subsection{Slowdown operator for the $q$-TASEP}
\label{sub:qTASEP_limit_down_new}

Assume now that the $q$-TASEP particle speeds 
are
\begin{equation}
	\label{eq:alpha_i_geometric_speeds}
	\alpha_i=r^i,\qquad i\in \mathbb{Z}_{\ge0},
\end{equation}
with $0<r<1$.
Taking $r\to 1$ leads to equal particle speeds. We consider this limit below
in \Cref{sub:qTASEP_limit_down_limit_new}.

Recall from \Cref{sub:QTASEP_cont}
that we denote the continuous time Markov semigroup
of the $q$-TASEP with speeds \eqref{eq:alpha_i_geometric_speeds}
by $\{\tilde T_r^{\mathrm{qT}}(\mathsf{t})\}_{\mathsf{t}\in \mathbb{R}_{\ge0}}$,
and the Markov shift operator for this process by 
$\tilde B_{r}^{\mathrm{qT}}$. All these operators
act on the space $\mathscr{X}$ of 
particle configurations.
The (iterated) intertwining relation reads
\begin{equation}
	\label{eq:S10_intertwining_qTASEP_r}
	\tilde T^{\mathrm{qT}}_r(\mathsf{t})
	\bigl(  
	\tilde B^{\mathrm{qT}}_{r}
	\bigr)^{m}
	=
	\bigl(  
	\tilde B^{\mathrm{qT}}_{r}
	\bigr)^{m}\ssp
	\tilde T^{\mathrm{qT}}_r(r^{m}\ssp \mathsf{t}),
\end{equation}
where $\mathsf{t}\in \mathbb{R}_{\ge0}$ and $m\in \mathbb{Z}_{\ge1}$ are arbitrary.

The action of 
$\tilde B_{r}^{\mathrm{qT}}$ is, by definition, the sequential application of 
the
Markov swap operators (denote them by $\tilde P^{(0,n)}_r$)
over $n=1,2,\ldots $, see 
\eqref{eq:B_operator}.
Each swap operator $\tilde P^{(0,n)}_r$ acts by 
randomly moving the particle $x_n$ backwards
and depends on the parameter
$\alpha_n/\alpha_0=r^n$.
Additionally, each swap operator $\tilde P^{(0,n)}_r$
gives rise to two rewriting history processes
which we denote by 
$\tilde H_{\alpha_0,\alpha_n}^{\leftarrow}
=\tilde H_{1,r^n}^{\leftarrow}$ 
and
$\tilde H_{\alpha_0,\alpha_n}^{\rightarrow}
=
\tilde H_{1,r^n}^{\rightarrow}$,
see 
\Cref{sub:continuous_time_bij}, 
by the independent bijectivisation.
Note that these processes depend
not only on the ratio $r^n=\alpha_n/\alpha_0$
but, also, on both these parameters separately, 
see \eqref{eq:rate_n_down_notation} and \eqref{eq:rate_n_up_notation}.
In
\Cref{sub:qTASEP_limit_down_new,sub:qTASEP_limit_down_limit_new},
we focus on the processes
$\tilde H_{\alpha_0,\alpha_n}^{\leftarrow}$ 
of rewriting history from future to past,
and consider the processes $\tilde H_{\alpha_0,\alpha_n}^{\rightarrow}$ below in \Cref{sub:qTASEP_limit_up_new}.

\medskip

Let us describe the slowdown Markov operators
$\tilde{\Xi}_{m,r}^{\leftarrow}$
on the space $\mathscr{X}^{[0,\mathsf{M}]}$ (\Cref{def:cont_time_space_of_trajectories_new}).

\begin{definition}
	\label{def:action_on_trajectories_from_future_to_past_qTASEP}
	Let $\mathfrak{x}=\{\mathbf{x}(\mathsf{t})\}_{0\le \mathsf{t}\le \mathsf{M}}
	\in \mathscr{X}^{[0,\mathsf{M}]}$
	and $m\in \mathbb{Z}_{\ge1}$ be fixed. 
	The action of
	$\tilde{\Xi}_{m,r}^{\leftarrow}$
	on $\mathfrak{x}$ is as follows.
	First, apply the Markov
	shift operator
	$\tilde B^{\mathrm{qT}}_r$ to 
	$\mathbf{x}(\mathsf{M})$ and denote the resulting random configuration by 
	$\mathbf{x}'(\mathsf{M})$. Then, 
	apply the rewriting history Markov operator 
	$\tilde H_{r^m,r^{m+n}}^{\leftarrow}$,
	sequentially for $n=1,2,\ldots $,
	to
	replace the old trajectory 
	$\{x_n(\mathsf{t})\}_{0\le \mathsf{t}\le \mathsf{M}}$
	by the random new trajectory
	$\{x_n'(\mathsf{t})\}_{0\le \mathsf{t}\le \mathsf{M}}$
	given the following data:
	\begin{equation*}
		\{x_{n-1}'(\mathsf{t})\}_{0\le \mathsf{t}\le \mathsf{M}},
		\qquad 
		x_n'(\mathsf{M}), 
		\qquad 
		\{x_{n+1}(\mathsf{t})\}_{0\le \mathsf{t}\le \mathsf{M}}
	\end{equation*}
	with $x_{n-1}'(\mathsf{t})=+\infty$ for $n=1$, by agreement.
	The updates, for $n=1,2,\ldots $, eventually terminate 
	since $x_n(t) = -n$ for $n >N$ if $N$ is large enough
	(recall that this is the property of 
	$\mathscr{X}^{[0,\mathsf{M}]}$).
	The new random trajectory
	$\mathfrak{x}'=\{\mathbf{x}'(\mathsf{t})\}_{0\le \mathsf{t}\le \mathsf{M}}$
	is, by definition, the result of 
	applying the Markov operator 
	$\tilde{\Xi}_{m,r}^{\leftarrow}$
	to $\mathfrak{x}$.
\end{definition}

Let us make two comments regarding
\Cref{def:action_on_trajectories_from_future_to_past_qTASEP}. 
First, the new initial configuration
$\mathbf{x}'(0)$ in $\mathfrak{x}'$
is random and, for $m=1$,  
it is
distributed as $\delta_{\mathbf{x}(0)}\tilde B^{\mathrm{qT}}_r$,
due to the intertwining relation.
Second, there is an important difference 
between 
$\tilde H_{r^{m},r^{m+n}}^{\leftarrow}$ 
and 
$\tilde{\Xi}_{m,r}^{\leftarrow}$.
The former assumes that the new terminal
configuration $x_n'(\mathsf{M})$ is fixed and, 
in the latter,
the terminal configuration evolves randomly.
We use different letters for these
operators because of this.

We now describe the action of the 
operators
$\tilde \Xi_{i,r}^{\leftarrow}$
on the $q$-TASEP measures on trajectories.
Let $\mathfrak{x}=\{\mathbf{x}(\mathsf{t})\}_{0\le \mathsf{t}\le \mathsf{M}}$
be the trajectory of the continuous time $q$-TASEP 
with particle speeds 
\eqref{eq:alpha_i_geometric_speeds}
started from a fixed initial
configuration $\mathbf{x}(0)$.
Fix $m\in \mathbb{Z}_{\ge0}$,
and let 
$\mathfrak{y}=\{\mathbf{y}(\mathsf{t})\}_{0\le \mathsf{t}\le \mathsf{M}}$
be the 
continuous time $q$-TASEP 
with the same speeds
\eqref{eq:alpha_i_geometric_speeds} but, instead, with random initial configuration
$\delta_{\mathbf{x}(0)}\bigl(\tilde B^{\mathrm{qT}}_r\bigr)^m$.

\begin{proposition}
	\label{prop:action_of_Xi_operator}
	Given the above notation,
	apply
	the Markov operators
	$
	\tilde{\Xi}_{0,r}^{\leftarrow},
	\tilde{\Xi}_{1,r}^{\leftarrow},
	\ldots,
	\tilde{\Xi}_{m-1,r}^{\leftarrow}$,
	in this order,
	to the $q$-TASEP trajectory 
	$\mathfrak{x}$.
	Then, the resulting trajectory has the same distribution as 
	$\{\mathbf{y}(r^m\mathsf{t})\}_{0\le \mathsf{t}\le \mathsf{M}}$.
\end{proposition}
Note that the operators
$\tilde{\Xi}_{i,r}^{\leftarrow}$
``slow down'' the time evolution of $q$-TASEP
by shrinking the time variable. 
We call
$\tilde{\Xi}_{i,r}^{\leftarrow}$
the \emph{slowdown operators} 
on trajectories
because of this.
\begin{proof}[Proof of \Cref{prop:action_of_Xi_operator}]
	The result follows by iterating 
	\Cref{prop:cont_time_qTASEP_down_result}
	over all $n$ and, then, repeating $m$ times.
	The result,
	after the application of 
	the composition of the operators
	$
	\tilde{\Xi}_{0,r}^{\leftarrow}
	\tilde{\Xi}_{1,r}^{\leftarrow}
	\ldots
	\tilde{\Xi}_{m-1,r}^{\leftarrow}$,
	is the $q$-TASEP with speeds $(r^m,r^{m+1},\ldots )$
	and initial configuration 
	$\delta_{\mathbf{x}(0)}\bigl(\tilde B^{\mathrm{qT}}_r\bigr)^m$.
	Additionally, note that multiplying all speeds by $r^m$ is the same
	as slowing down the time by the overall factor $r^m$
	since our $q$-TASEP runs in continuous time.
	multiplying all speeds by $r^m$ is the same
	as slowing down the time by the overall factor $r^m$.
	Moreover, this does not affect the random initial configuration 
	$\mathbf{y}(0)$.
	This completes the proof.
\end{proof}

\subsection{Slowdown dynamics for the homogeneous $q$-TASEP}
\label{sub:qTASEP_limit_down_limit_new}

Let us now take the 
limit as in \Cref{sub:QTASEP_cont_back},
\begin{equation}
	\label{eq:r_to_one_limit_S10}
	r\nearrow 1,\qquad 
	m=\lfloor (1-r)^{-1}\tau \rfloor ,
\end{equation}
where $\tau\in \mathbb{R}_{\ge0}$ is the continuous time parameter.
In this limit, 
the $q$-TASEP becomes \emph{homogeneous} with particle speeds
$\alpha_i=1$ for all $i$, and 
the operators 
$\bigl(\tilde B^{\mathrm{qT}}_r\bigr)^m$
turn into the continuous time 
Markov semigroup $\{\tilde B^{\mathrm{qT}}(\tau)\}_{\tau\in \mathbb{R}_{\ge0}}$
on $\mathscr{X}$. This semigroup corresponds to the
backwards $q$-TASEP dynamics 
\cite{petrov2019qhahn}; we recalled the definition in 
\Cref{sub:QTASEP_cont_back}.
Our aim now is to
extend the semigroup
$B^{\mathrm{qT}}(\tau)$ on $\mathscr{X}$
to continuous time Markov dynamics 
on the space 
$\mathscr{X}^{[0,\mathsf{M}]}$
of trajectories
by taking the limit
\eqref{eq:r_to_one_limit_S10}
of the slowdown operators.
The latter dynamics are denoted by
$\tilde \Xi^{\leftarrow}(\tau)$.
An interesting feature of 
$\tilde \Xi^{\leftarrow}(\tau)$,
compared to 
$B^{\mathrm{qT}}(\tau)$,
is that, while the former is a \emph{time-inhomogeneous Markov process}
(its transitions depend on the time variable),
the latter is time-homogeneous.
See \Cref{rmk:time_inhom_slowdown_discussion_new} below for
more discussion.

Let us first define the dynamics
$\tilde \Xi^{\leftarrow}(\tau)$. Then, 
in \Cref{prop:limit_hom_qTASEP_down_new} below,
we show 
that these dynamics are the desired $r\to1$ limit
of the sequential application of the operators
$\tilde{\Xi}_{i,r}^{\leftarrow}$.

\begin{definition}
	\label{def:Xi_down_cont_time}
	Fix $\mathsf{M}\in \mathbb{R}_{\ge0}$.
	The continuous time 
	\emph{slowdown Markov dynamics} 
	$\tilde \Xi^{\leftarrow}(\tau)$
	acting on trajectories 
	$\mathfrak{x}=\{\mathbf{x}(\mathsf{t})\}_{0\le \mathsf{t}\le \mathsf{M}}
	\in\mathscr{X}^{[0,\mathsf{M}]}$
	possesses
	two sources of independent jumps and, also, a random jump propagation 
	mechanism. Almost surely
	there are only finitely many independent jumps
	in finite time.
	The independent jumps are as follows:
	\begin{enumerate}[$\bullet$]
		\item 
			(\emph{terminal jumps})
			The terminal configuration $\mathbf{x}(\mathsf{M})$ in $\mathfrak{x}$ evolves
			according to the backwards $q$-TASEP $\tilde B^{\mathrm{qT}}(\tau)$. 
			Recall, from \Cref{sub:QTASEP_cont_back}, that this means that 
			each particle $x_n(\mathsf{M})$, $n=1,2,\ldots $,
			jumps down to a new location
			$x_n'(\mathsf{M})$, where $x_{n+1}(\mathsf{M})<x_n'(\mathsf{M})<x_n(\mathsf{M})$,
			with rate 
			\begin{equation*}
				\frac{n\ssp (q;q)_{x_n'(\mathsf{M})-x_{n+1}(\mathsf{M})-1}}
				{(1-q^{x_n(\mathsf{M})-x_n'(\mathsf{M})})
				\ssp (q;q)_{x_n(\mathsf{M})-x_{n+1}(\mathsf{M})-1}}.
			\end{equation*}
		\item 
			(\emph{bulk jumps})
			For all $0<\mathsf{t}<\mathsf{M}$,
			each particle $x_n(\mathsf{t})$ can jump down by $1$ according
			to the following mechanism.
			Let
			\begin{equation*}
				T_{\mathfrak{x}}=
				\{
				0=\mathsf{t}_0<\mathsf{t}_1
				<\mathsf{t}_2<\ldots<\mathsf{t}_k<\mathsf{t}_{k+1}=\mathsf{M}\}.
			\end{equation*}
			For $n\ge1$ and $1\le j\le k+1$, attach to each segment
			$[x_n(\mathsf{t}_{j-1}),x_n(\mathsf{t}_{j})]$
			an independent
			exponential clock of (time-inhomogeneous) rate 
			\begin{equation}
				\label{eq:rate_slowdown_qTASEP_new}
				\left( 
					x_n(\mathsf{t}_j)-x_n(\mathsf{t}_{j-1})
				\right)
				\frac{n \ssp e^{-\tau}\ssp
				(1-q^{x_{n-1}(\mathsf{t})-x_n(\mathsf{t})})(1-q^{x_n(\mathsf{t})-x_{n+1}(\mathsf{t})-1})}{1-q},
				\qquad 
				\textnormal{where $\mathsf{t}\in(\mathsf{t}_{j-1},\mathsf{t}_j)$}.
			\end{equation}
			Here $\tau$ is the time variable in the dynamics
			$\tilde \Xi^{\leftarrow}(\tau)$.
			When the clock rings, place a uniformly random point
			$\mathsf{t}_*\in (\mathsf{t}_{j-1},\mathsf{t}_j)$. Then,
			let the
			trajectory of $x_n$ make a new increment at $\mathsf{t}_*$, so that
			$x_n'(\mathsf{t}_*)=x_n(\mathsf{t}_*)-1$
			and 
			$x_n'(\mathsf{t}_*+)=x_n(\mathsf{t}_*+)$.
			Note that if 
			$x_n(\mathsf{t}_*)=x_{n+1}(\mathsf{t}_*)+1$, the particle
			$x_n(\mathsf{t}_*)$ is blocked from jumping down, and the rate 
			\eqref{eq:rate_slowdown_qTASEP_new} vanishes, as it should.
	\end{enumerate}
	
	Let us now describe the \emph{jump propagation mechanism}.
	Assume that a particle $x_n(\mathsf{t}_*)$,
	$\mathsf{t}_*\in(0,\mathsf{M}]$, has jumped down to $x_n'(\mathsf{t}_*)$
	as described above. 
	We either have a terminal jump with $\mathsf{t}_*=\mathsf{M}$
	or a bulk jump with a random $\mathsf{t}_*<\mathsf{M}$.
	The jump then instantaneously 
	propagates left 
	according to a, backwards in time, random walk 
	in the chamber 
	$x_{n}(\mathsf{t})\ge x_{n}'(\mathsf{t})>x_{n+1}(\mathsf{t})$, 
	$0\le \mathsf{t}\le \mathsf{t}_*$. This random walk makes jumps down by $1$ in continuous
	time $\mathsf{t}$ with rate
	\begin{equation}
		\label{eq:rate_slowdown_qTASEP_propagation}
		\frac{e^{-\tau}\ssp (1- q^{b-c})(1-q^{a+b+1-c})(1-q^c)}{1-q^{b-c+1}},
	\end{equation}
	where $a,b,c$ are given in \eqref{eq:down_process_cont_time_abc_def_new}.
	Note that \eqref{eq:rate_slowdown_qTASEP_propagation} vanishes for $b=c$.
	This means that, during the instantaneous jump propagation,
	$x_n'(\mathsf{t})$ 
	either joins the old trajectory of
	$x_n$ and continues to follow it until $\mathsf{t}=0$,
	or $x_n'(\mathsf{t})$
	modifies the
	initial configuration $\mathbf{x}(0)$. In particular, the new trajectory will not deviate from the old trajectory once it has deviated and joined the old trajectory. 
	The trajectories of all other
	particles $x_l$, $l\ne n$, do not change during this instantaneous jump
	propagation.

	We refer to \Cref{fig:slowdown_dynamics} for an illustration
	of the dynamics 
	$\tilde \Xi^{\leftarrow}(\tau)$. 
\end{definition}

\begin{figure}[htb]
	\centering
	\includegraphics[width=.6\textwidth]{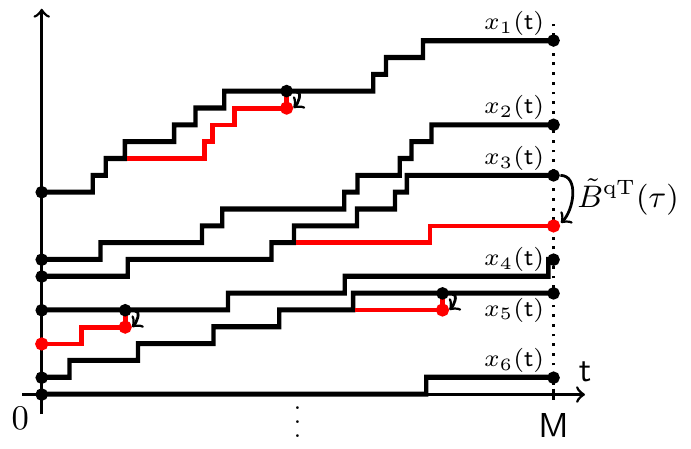}
	\caption{The slowdown dynamics
	$\tilde \Xi^{\leftarrow}(\tau)$
	acting on trajectories. In the figure there are four possible transitions: 
	one initiated at $\mathsf{t}=\mathsf{M}$ by the backwards $q$-TASEP 
	$\tilde B^{\mathrm{qT}}$, and three others
	initiated by bulk independent jumps.
	Note that one of the jump propagations modifies the initial configuration
	$\mathbf{x}(0)$.}
	\label{fig:slowdown_dynamics}
\end{figure}

\begin{proposition}
	\label{prop:limit_hom_qTASEP_down_new}
	Fix $\tau\in \mathbb{R}_{\ge0}$. 
	Then, the limit
	\begin{equation*}
		\lim\nolimits_{r\nearrow 1}
		\tilde{\Xi}_{0,r}^{\leftarrow}
		\tilde{\Xi}_{1,r}^{\leftarrow}
		\ldots
		\tilde{\Xi}_{m-1,r}^{\leftarrow}
		=
		\tilde \Xi^{\leftarrow}(\tau),
		\qquad \textnormal{with $m=\lfloor (1-r)^{-1}\tau \rfloor $},
	\end{equation*}
	converges in the sense of 
	the transition densities \eqref{eq:Markov_op_cont_times_prob_dens_new}
	associated to 
	Markov operators on
	$\mathscr{X}^{[0,\mathsf{M}]}$.
\end{proposition}
\begin{proof}
	Let $m=\lfloor (1-r)^{-1}\tau \rfloor $,
	and consider the action of the slowdown
	operator 
	$\tilde{\Xi}_{m,r}^{\leftarrow}$
	on the $n$-th particle.
	By \Cref{thm:mapping_qTASEP_TASEP_back_general_IC},
	the action of the terminal jumps by
	$\tilde B^{\mathrm{qT}}(\tau)$ follows.
	In particular, we note that the events of a terminal jump
	has probability of order $O(1-r)$ as $r \rightarrow 1$.
	It remains to consider bulk jumps and jump propagation.

	Recall 
	that 
	$\tilde H_{r^{m},r^{m+n}}^{\leftarrow}$ 
	replaces the old trajectory $x_n(\mathsf{t})$ by a random walk
	$x_n'(\mathsf{t})$
	in reverse continuous time from $\mathsf{M}$ to $0$
	which makes steps down by 1 at rates
	$\mathrm{rate}^{\leftarrow}_{n;r^m,r^{m+n}}$,
	see \eqref{eq:rate_n_down_notation}.
	If $b=c$, the rate $\mathrm{rate}^{\leftarrow}_{n;r^m,r^{m+n}}$ is of order $O(1-r)$ as $r \rightarrow 1$. Otherwise, if $b>c$, the rate $\mathrm{rate}^{\leftarrow}_{n;r^m,r^{m+n}}$ is of order $O(1)$ as $r \rightarrow 1$.

	Note that only one event with probability $O(1-r)$ may happen in a single moment
	of the new continuous time $\tau$
	in the Poisson-type limit 
	for 
	$\tilde{\Xi}_{0,r}^{\leftarrow}
	\tilde{\Xi}_{1,r}^{\leftarrow}
	\ldots
	\tilde{\Xi}_{m-1,r}^{\leftarrow}$
	as $r\to1$.
	It is either a terminal jump or a bulk jump. Recall that a terminal jump happens according to $\tilde B^{\mathrm{qT}}(\tau)$.

	For bulk jumps, observe that
	\begin{equation*}
		\mathrm{rate}^{\leftarrow}_{n;r^m,r^{m+n}}=
		\frac{n \ssp e^{-\tau}\ssp
		(1-q^{a+1})(1-q^{c})}{1-q}
		\ssp(1-r)+O(1-r)^2,
	\end{equation*}
	where $a,c$ are given in \eqref{eq:down_process_cont_time_abc_def_new}.
	Therefore, during the continuous time $d\tau$, on a segment of length 
	$d\mathsf{t}$ there is a new independent bulk jump with 
	small probability 
	\begin{equation}\label{eq:2022_11_18_down_jump_proof}
		\frac{n \ssp e^{-\tau}\ssp
		(1-q^{x_{n-1}(\mathsf{t})-x_n(\mathsf{t})})(1-q^{x_n(\mathsf{t})-x_{n+1}(\mathsf{t})-1})}
		{1-q}\ssp
		d\tau \ssp d\mathsf{t}.
	\end{equation}
	Moreover, a bulk jump can happen on at most one interval.
	Averaging \eqref{eq:2022_11_18_down_jump_proof} over the segment 
	$[x_n(\mathsf{t}_{j-1}),x_n(\mathsf{t}_{j})]$ where this jump happens
	immediately
	leads to the desired jump rates 
	\eqref{eq:rate_slowdown_qTASEP_new}.

	Finally,
	once an event with probability $O(1-r)$ occurred for a particle $x_n$ at $\mathsf{t}=\mathsf{t}_*$,
	the rest of the $x_n$'s trajectory to the left of $\mathsf{t}_*$ needs to be 
	instantaneously modified. This modification happens according to the random
	walk $\tilde H_{r^m,r^{m+n}}^{\leftarrow}$. This leads to  further nontrivial jumps down by $1$
	when $b>c$. In the case that the new trajectory joins the old trajectory after a terminal or bulk jump, we will have $b=c$ and the probability that the new trajectory jumps down by one again is of order $O(1-r)$, with an overall probability of $O((1-r)^2)$ for this sequence of events. Thus, the new trajectory will not deviate from the old trajectory after deviating once and joining the old trajectory. Then, in the limit as $r\to1$,
	this jump propagation turns into the Poisson random walk with rates
	\eqref{eq:rate_slowdown_qTASEP_propagation}.
	This 
	completes the proof.
\end{proof}

\begin{remark}
	\label{rmk:time_inhom_slowdown_discussion_new}
	Notice that the time inhomogeneity in
	$\tilde \Xi^{\leftarrow}(\tau)$
	is only present in the bulk jumps and jump propagation,
	but not 
	in the terminal jumps.
	From the proof of \Cref{prop:limit_hom_qTASEP_down_new},
	this is because 
	the 
	rates of bulk jumps for $x_n$
	depend on both parameters $\alpha_m=r^m$, 
	$\alpha_{m+n}=r^{m+n}$ before the $r\to 1$ limit,
	while for the terminal jumps they only depend
	on the ratio $\alpha_{m+n}/\alpha_m=\alpha_n/\alpha_0=r^n$.
\end{remark}

From \Cref{prop:limit_hom_qTASEP_down_new,prop:action_of_Xi_operator}
we immediately get the following slowdown action 
of the dynamics
$\tilde \Xi^{\leftarrow}(\tau)$
on trajectories of the 
homogeneous continuous time $q$-TASEP with all particles speeds equal to one:

\begin{proposition}
	\label{prop:action_Xi_down_cont_new}
	Fix
	$\tau\in \mathbb{R}_{\ge0}$, and let
	$\mathfrak{x}=\{\mathbf{x}(\mathsf{t})\}_{0\le \mathsf{t}\le \mathsf{M}}$ and 
	$\{\mathbf{y}(\mathsf{t})\}_{0\le \mathsf{t}\le \mathsf{M}}$
	be the homogeneous continuous time $q$-TASEPs
	started from a fixed initial
	configuration $\mathbf{x}(0)$ and
	from 
	the random initial configuration
	$\delta_{\mathbf{x}(0)}\tilde B^{\mathrm{qT}}(\tau)$, respectively.
	Apply
	$\tilde \Xi^{\leftarrow}(\tau)$
	to
	$\mathfrak{x}$,
	that is, run the slowdown dynamics 
	(\Cref{def:Xi_down_cont_time})
	for time $\tau$ started from
	$\mathfrak{x}$.
	Then, the resulting random trajectory
	is distributed as
	$\{\mathbf{y}(e^{-\tau}\mathsf{t})\}_{0\le \mathsf{t}\le \mathsf{M}}$.
\end{proposition}

\subsection{Speedup dynamics for the $q$-TASEP with step initial configuration}
\label{sub:qTASEP_limit_up_new}

The slowdown process
$\tilde \Xi^{\leftarrow}(\tau)$
constructed above in \Cref{sub:qTASEP_limit_down_limit_new}
provides a bijectivisation
of the intertwining relation of
\Cref{thm:mapping_qTASEP_TASEP_back_general_IC},
\begin{equation}
	\label{eq:qTASEP_intertwining_S10}
	\tilde T^{\mathrm{qT}}(\mathsf{t})
	\ssp
	\tilde B^{\mathrm{qT}}(\tau)
	=
	\tilde B^{\mathrm{qT}}(\tau)
	\ssp
	\tilde T^{\mathrm{qT}}\bigl(e^{-\tau}\mathsf{t}\bigr).
\end{equation}
One can informally say that the slowdown process
acts on identity \eqref{eq:qTASEP_intertwining_S10}
from the left-hand side to the right-hand side, see \Cref{prop:action_Xi_down_cont_new}.
The slowdown process
contains, in particular, 
the backwards $q$-TASEP 
$\tilde B^{\mathrm{qT}}(\tau)$ running
on the terminal configuration $\mathbf{x}(\mathsf{M})$
of the trajectory. 
In this subsection
we discuss a bijectivisation of 
\eqref{eq:qTASEP_intertwining_S10}
in another direction, from right to left,
by means of a \emph{speedup process}
$\tilde \Xi^{\rightarrow}(\tau)$.
To simplify notation and formulations, 
we only consider the action of the speedup process
on 
trajectories with the step initial configuration
\begin{equation*}
	\mathbf{y}(0)=\mathbf{x}_{step}=\{\ldots,-3,-2,-1 \}.
\end{equation*}

\begin{remark}
	\label{rmk:speedup_process_general_IC}
	In the general case, the initial configuration $\mathbf{y}(0)$ must be random with distribution $\delta_{\mathbf{x}(0)}\ssp\tilde B^{\mathrm{qT}}(\tau)$ where $\mathbf{x}(0)\in\mathscr{X}$ is fixed.
	In particular, 
	the speedup process itself would 
	need as its input
	a deterministic sequence of down particle jumps 
	at the trajectory's initial configuration. That is, such a sequence
	would record
	the transition from $\mathbf{x}(0)$ to $\mathbf{y}(0)$ during time~$\tau$.
	These deterministic jumps
	can be easily incorporated into
	the definitions and the results given below in 
	the current \Cref{sub:qTASEP_limit_up_new},
	similarly to the second part of
	\Cref{thm:coupling_trajectories_with_Pn_swap}
	and \Cref{prop:cont_time_qTASEP_up_result}.
	However, for simplicity, we omit the discussion
	of this more general case and focus only on 
	trajectories with the step initial configuration.
\end{remark}

The Markov operator 
$\tilde \Xi^{\rightarrow}(\tau)$
of the speedup process, 
on trajectories with the step initial configuration,
is constructed as the $r\nearrow 1$ limit 
of the sequential application of the 
corresponding speedup operators
$\tilde{\Xi}_{0,r}^{\rightarrow}
\tilde{\Xi}_{1,r}^{\rightarrow}
\ldots
\tilde{\Xi}_{m-1,r}^{\rightarrow}$,
where $m=\lfloor (1-r)^{-1}\tau \rfloor$.
In contrast with the slowdown operators
(\Cref{def:action_on_trajectories_from_future_to_past_qTASEP}),
the action of 
$\tilde{\Xi}_{i,r}^{\rightarrow}$
is a sequential application 
of the 
history rewriting processes 
$\tilde H_{r^{-i}; r^{-i+n}}^{\rightarrow}$,
where $n$ decreases from a suitably large $N$ down to~$1$.
Here, $N$ depends on the trajectory
$\mathfrak{y}=\{\mathbf{y}(\mathsf{t})\}_{0\le \mathsf{t}\le\mathsf{M}}
\in \mathscr{X}^{[0,\mathsf{M}]}$
to which 
$\tilde{\Xi}_{i,r}^{\rightarrow}$
is applied. More precisely, $N$ is determined so that 
$y_k(\mathsf{t})=-k$ for all $\mathsf{t}\in[0,\mathsf{M}]$ and
$k\ge N-1$, see
\Cref{def:cont_time_space_of_trajectories_new}.
Similarly to \Cref{prop:action_of_Xi_operator}, one can check that 
the action of 
$\tilde{\Xi}_{0,r}^{\rightarrow}
\tilde{\Xi}_{1,r}^{\rightarrow}
\ldots
\tilde{\Xi}_{m-1,r}^{\rightarrow}$
on a trajectory 
$\mathfrak{y}=\{\mathbf{y}(\mathsf{t})\}_{\mathsf{t}\ge 0}$ 
of the $q$-TASEP with 
particle speeds
$(1,r,r^2,\ldots )$ results in a
trajectory of the $q$-TASEP with
particle speeds
$(r^{-m},r^{-m+1},r^{-m+2},\ldots )$;
 the initial configuration is $\mathbf{x}_{step}$ in both processes.
Equivalently, 
one can say that the action of the Markov operator
$\tilde{\Xi}_{0,r}^{\rightarrow}
\tilde{\Xi}_{1,r}^{\rightarrow}
\ldots
\tilde{\Xi}_{m-1,r}^{\rightarrow}$
speeds up the time $\mathsf{t}$ in the $q$-TASEP
with rates $(1,r,r^2,\ldots )$ by the factor $r^{-m}>1$.

The $r\nearrow 1$ limit of the operators 
$\tilde{\Xi}_{0,r}^{\rightarrow}
\tilde{\Xi}_{1,r}^{\rightarrow}
\ldots
\tilde{\Xi}_{m-1,r}^{\rightarrow}$
in the sense of the transition densities
\eqref{eq:Markov_op_cont_times_prob_dens_new}
is obtained very similarly to 
\Cref{sub:qTASEP_limit_down_limit_new},
with an additional simplification
coming from the step initial configuration.
Therefore, here we will only define
the resulting continuous time speedup process
$\tilde{\Xi}^{\rightarrow}(\tau)$,
and formulate an analogue of 
\Cref{prop:action_Xi_down_cont_new}.

\begin{definition}
	\label{def:Xi_up_cont_time}
	Fix $\mathsf{M}\in \mathbb{R}_{\ge0}$. 
	The continuous time \emph{speedup Markov dynamics}
	$\tilde{\Xi}^{\rightarrow}(\tau)$
	acting on trajectories 
	$\mathfrak{y}=\{\mathbf{y}(\mathsf{t})\}
	_{0\le \mathsf{t}\le\mathsf{M}}
	\in \mathscr{X}^{[0,\mathsf{M}]}$
	with $\mathbf{y}(0)=\mathbf{x}_{step}$
	possesses one source of independent jumps
	(the \emph{bulk jumps}), and a mechanism of random jump propagation.
	Let
	\begin{equation*}
		T_{\mathfrak{y}}=\{0=\mathsf{t}_0<\mathsf{t}_1<\ldots<\mathsf{t}_k
		<\mathsf{t}_{k+1}=\mathsf{M}\}.
	\end{equation*}
	For $n\ge1$ and $1\le j\le k+1$,
	attach to each segment 
	$[y_n(\mathsf{t}_{j-1}),y_n(\mathsf{t}_j)]$
	an independent exponential clock of 
	(time-inhomogeneous) rate
	\begin{equation}
		\label{eq:rate_speedup_qTASEP_new}
		\left( y_n(\mathsf{t}_j)-y_n(\mathsf{t}_{j-1}) \right)
		\frac{n\ssp e^{\tau} (1-q^{y_n(\mathsf{t})-y_{n+1}(t)})
		(1-q^{y_{n-1}(\mathsf{t})-y_{n}(\mathsf{t})-1})}{1-q},
		\qquad 
		\textnormal{where $\mathsf{t}\in (\mathsf{t}_{j-1},\mathsf{t}_{j})$}.
	\end{equation}
	Here, $\tau$ is the time variable in the dynamics $\tilde \Xi^{\rightarrow}(\tau)$.
	When the clock rings, place a uniformly random point
	$\mathsf{t}_*\in (\mathsf{t}_{j-1},\mathsf{t}_j)$,
	and let the trajectory of $y_n$ make a new increment at $\mathsf{t}_*$. In particular,
	$y_n'(\mathsf{t}_*)=y_n(\mathsf{t}_*)+1$
	and 
	$y_n'(\mathsf{t}_*-)=y_n(\mathsf{t}_*-)$.
	Note that if 
	$y_n(\mathsf{t}_*)=y_{n-1}(\mathsf{t}_*)-1$,
	the particle
	$y_n(\mathsf{t}_*)$ is blocked from jumping up, and the rate 
	\eqref{eq:rate_speedup_qTASEP_new} vanishes, as it should.
	Almost surely there are only finitely many independent
	jumps in finite time.

	Let us now describe the \emph{jump propagation}. 
	If a particle $y_n(\mathsf{t}_*)$
	has jumped up to $y_n'(\mathsf{t}_*)$,
	then the new trajectory of
	$y_n'$ coincides with that of $y_n$ for $\mathsf{t}<\mathsf{t}_*$. 
	To the right of $\mathsf{t}_*$, instantaneously (in $\tau$)
	continue the new trajectory $y_n'$
	according to a Poisson
	simple random walk in forward time $\mathsf{t}$
	in the chamber
	$y_n(\mathsf{t})\le y_n'(\mathsf{t})< y_{n-1}(\mathsf{t})$
	which makes jumps up by $1$ in continuous time $\mathsf{t}$
	with rate
	\begin{equation}
		\label{eq:rate_speedup_qTASEP_new_2}
		\frac{e^{\tau}\ssp (1-q^{b-c})(1-q^a)(1-q^{b+1})}{1-q^{b-c+1}},
	\end{equation}
	where $a,b,c$ are given in \eqref{eq:down_process_cont_time_abc_def_UP_new}.
	During this instantaneous jump propagation, $y_n'(\mathsf{t})$
	either
	joins the old trajectory of
	$y_n$ and continues to follow it till $\mathsf{t}=\mathsf{M}$,
	or modifies the terminal configuration 
	$\mathbf{y}(\mathsf{M})$. In particular, the new trajectory will not deviate from the old trajectory once it has deviated and joined the old trajectory.
	The trajectories of all other particles
	$y_l$, $l\ne n$, are no affected by this instantaneous jump propagation.
\end{definition}

The jump rates 
\eqref{eq:rate_speedup_qTASEP_new} 
and 
\eqref{eq:rate_speedup_qTASEP_new_2}
in the speedup process
$\tilde{\Xi}^{\rightarrow}(\tau)$
are 
obtained 
in the $r\to 1$ expansion of the jump rates
\eqref{eq:rate_n_up_notation}
(with $\alpha_{n-1}=r^{-m}$, $\alpha_n=r^{-m+n}$)
for $b=c$ and $b>c$, respectively.
The
argument here is very similar to the proof of 
\Cref{prop:limit_hom_qTASEP_down_new}.

The speedup process 
acts on trajectories of the homogeneous $q$-TASEP with the step initial
configuration as follows:

\begin{proposition}
	\label{prop:action_Xi_up_cont}
	Let 
	$\mathfrak{y}=\{\mathbf{y}(\mathsf{t})\}_{\mathsf{t}\ge 0}$ 
	be the homogeneous continuous time $q$-TASEP started from 
	$\mathbf{x}_{step}$.
	Let $\tau,\mathsf{M}\in \mathbb{R}_{\ge0}$.
	Apply the speedup operator 
	$\tilde \Xi^{\rightarrow}(\tau)$
	to $\mathfrak{y}$ on $[0,\mathsf{M}]$. Then, the resulting random trajectory is distrubuted as 
	$\{\mathbf{y}(e^{\tau}\mathsf{t})\}_{0\le \mathsf{t}\le \mathsf{M}}$.
\end{proposition}

\subsection{Slowdown and speedup dynamics for the homogeneous TASEP}
\label{sub:TASEP_equal_speeds}

The slowdown and speedup processes
$\tilde \Xi^{\leftarrow}(\tau)$
and
$\tilde \Xi^{\rightarrow}(\tau)$
for the homogeneous TASEP,  
with particle speeds $\alpha_i=1$ for all $i$,
are obtained by setting $q=0$
in the processes 
from \Cref{def:Xi_down_cont_time,def:Xi_up_cont_time}, respectively.
This greatly simplifies the dynamics.
In this subsection we provide the necessary definitions.

We start with the slowdown process
$\tilde \Xi^{\leftarrow}(\tau)$
acting on the space of trajectories
$\mathscr{X}^{[0,\mathsf{M}]}$ (\Cref{def:cont_time_space_of_trajectories_new}),
where $\mathsf{M}\in \mathbb{R}_{>0}$ is fixed:

\begin{definition}
	\label{def:TASEP_slowdown_ind}
	The continuous time slowdown Markov process
	$\{\tilde \Xi^{\leftarrow}(\tau)\}_{0\le \tau\le \tau_0}$
	for TASEP 
	acts on a
	trajectory 
	$\mathfrak{x}=\{\mathbf{x}(\mathsf{t}\mathsf)\}_{0\le \mathsf{t}\le \mathsf{M}}
	\in\mathscr{X}^{[0,\mathsf{M}]}$
	as follows:
	\begin{enumerate}[$\bullet$]
		\item 
			(\emph{terminal jumps})
			Run the backwards Hammersley dynamics 
			at the terminal configuration $\mathbf{x}(\mathsf{M})$.
			In particular, each 
			particle $\mathsf{x}_n(\mathsf{M})$, $n=1,2,\ldots $,
			jumps down to a new location $\mathsf{x}_n'(\mathsf{M})$,
			with $x_{n+1}(\mathsf{M})<x_n'(\mathsf{M})<x_n(\mathsf{M})$ and rate $n$ per available location to land. This process is time-homogeneous in $\tau$, and it is called the backwards Hammersley process.
		\item 
			(\emph{bulk jumps})
			Take an independent two-dimensional Poisson process
			$\mathfrak{P}_{[n]}$ in 
			$[0,\mathsf{M}]\times [0,\tau_0]$,
			for each $n\ge 1$, with inhomogeneous rate
			$e^{-\tau}$. As $\tau$ increases from $0$ to $\tau_0$, each point 
			$(\mathsf{t}_*,\tau_*)$ of $\mathfrak{P}_{[n]}$
			generates a jump down by $1$ of the trajectory 
			$\{x_n(\mathsf{t})\}_{0\le  \mathsf{t}\le \mathsf{M}}$ at 
			$\mathsf{t}=\mathsf{t}_*$. This means that 
			we set
			$x_n'(\mathsf{t}_*)=x_n(\mathsf{t}_*)-1$
			and 
			$x_n'(\mathsf{t}_*+)=x_n(\mathsf{t}_*+)$
			if
			$x_n(\mathsf{t}_*)>x_{n+1}(\mathsf{t}_*)+1$.
			Otherwise, if $x_n(\mathsf{t}_*)=x_{n+1}(\mathsf{t}_*)+1$, the jump down is blocked,
			and the trajectory $\mathfrak{x}$ is not changed.
		\item (\emph{jump propagation})
			Replace, instantaneously at the same moment $\tau=\tau_*$,
			the old trajectory of $x_n$ by the new one, $x_n'$,
			for all $\mathsf{t}\le \mathsf{t}_*$, where
			$\mathsf{t}_*\le \mathsf{M}$ is the time a 
			terminal or a bulk jump which
			occurred.
			The new trajectory $x_n'$ starts from $x_n'(\mathsf{t}_*)$
			and evolves 
			in backwards time $\mathsf{t}$
			in the chamber 
			$x_{n}(\mathsf{t})\ge x_{n}'(\mathsf{t})>x_{n+1}(\mathsf{t})$, 
			$0\le \mathsf{t}\le \mathsf{t}_*$.
			The dynamics of $x_n'(\mathsf{t})$ is that of the Poisson
			simple random walk which makes jumps down at rate $e^{-\tau_*}$, and gets 
			absorbed by the top wall $x_n(\mathsf{t})$ 
			of the chamber once it reaches it. In other words,
			$x_n'(\mathsf{t})$ 
			either joins the old trajectory of
			$x_n$ and continues to follow it till $\mathsf{t}=0$,
			or 
			it modifies the
			initial configuration $\mathbf{x}(0)$. 
			The trajectories of all other
			particles $x_l$, $l\ne n$, do not change during this instantaneous jump
			propagation.
	\end{enumerate}
\end{definition}

The slowdown process
$\tilde \Xi^{\leftarrow}(\tau)$
satisfies the $q=0$ version of \Cref{prop:action_Xi_down_cont_new},
that is, where the $q$-TASEP is replaced by the TASEP.

\medskip

Let us now describe the speedup process
for TASEP. As in \Cref{sub:qTASEP_limit_up_new},
we restrict attention only to trajectories
$\mathfrak{y}=\{\mathbf{y}(\mathsf{t})\}_{0\le \mathsf{t}\le \mathsf{M}}
\in \mathscr{X}^{[0,\mathsf{M}]}$
with the step initial configuration $\mathbf{y}(0)=\mathbf{x}_{step}$
(see \Cref{rmk:speedup_process_general_IC} for more discussion).
The $q=0$ speedup process is given as follows:

\begin{definition}
	\label{def:TASEP_speedup_ind}
	The continuous time process 
	$\tilde \Xi^{\rightarrow}(\tau)$
	for TASEP
	acts on trajectories 
	$\mathfrak{y}=\{\mathbf{y}(\mathsf{t})\}_{0\le \mathsf{t}\le \mathsf{M}}$
	with the step initial configuration as follows:
	\begin{enumerate}[$\bullet$]
		\item 
			(\emph{bulk jumps})
			Take an independent two-dimensional Poisson process $\mathfrak{P}_{[n]}$ in 
			$[0,\mathsf{M}]\times [0,\tau_0]$, for each $n\ge 1$, of inhomogeneous rate
			$e^{\tau}$. As $\tau$ increases from $0$ to $\tau_0$, each point 
			$(\mathsf{t}_*,\tau_*)$ of $\mathfrak{P}_{[n]}$
			generates a jump up by $1$ of the trajectory 
			$\{t_n(\mathsf{t})\}_{0\le  \mathsf{t}\le \mathsf{M}}$ at 
			$\mathsf{t}=\mathsf{t}_*$. This means 
			that
			$y_n'(\mathsf{t}_*)=y_n(\mathsf{t}_*)+1$
			and 
			$y_n'(\mathsf{t}_*-)=y_n(\mathsf{t}_*-)$
			unless this jump is blocked, when $y_n(\mathsf{t}_*)=y_{n-1}(\mathsf{t}_*)-1$).
			In the case of blocking, the trajectory of $y_n$ is not changed.
		\item 
			(\emph{jump propagation})
			Replace, instantaneously at the same moment $\tau=\tau_*$, the trajectory of $y_n$ with a new trajectory $y_n'$.
			If a particle $y_n(\mathsf{t}_*)$
			has jumped up by $1$ to $y_n'(\mathsf{t}_*)$,
			then the new trajectory of
			$y_n'$ coincides with that of $y_n$ for $\mathsf{t}<\mathsf{t}_*$. 
			To the right of $\mathsf{t}_*$, 
			continue the new trajectory $y_n'$
			according to a Poisson
			simple random walk in forward time $\mathsf{t}$
			in the chamber
			$y_n(\mathsf{t})\le y_n'(\mathsf{t})< y_{n-1}(\mathsf{t})$.
			The random walk makes a jump up by $1$ in continuous time $\mathsf{t}$
			at rate $e^{\tau_*}$
			and gets absorbed by the bottom wall 
			$y_n(t)$ of the chamber once it reaches it. In particular,
			$y_n'$
			either
			joins the old trajectory of
			$y_n$ and continues to follow it till $\mathsf{t}=\mathsf{M}$,
			or it modifies the terminal configuration 
			$\mathbf{y}(\mathsf{M})$.
			The trajectories of all other particles
			$y_l$, $l\ne n$, are no affected by this instantaneous jump propagation.
	\end{enumerate}
\end{definition}

The process
$\tilde \Xi^{\rightarrow}(\tau)$
satisfies the $q=0$ version of \Cref{prop:action_Xi_up_cont},
that is, the same statement where the word ``$q$-TASEP'' is replaced by ``TASEP''.

\medskip

The action of the
speedup dynamics on the first particle
produces an interesting 
coupling between standard Poisson processes on the positive half-line with 
different slopes (here by ``slope'' we mean the slope of the counting function
of the Poisson process, which is usually referred to as ``rate'', ``intensity'', or ``density'').
First, observe that the trajectory
$\{y_1(\mathsf{t})\}_{\mathsf{t}\ge 0}$ 
of the first TASEP particle is the 
continuous time
Poisson simple random walk 
with slope $1$.
By \Cref{prop:action_Xi_up_cont} for $q=0$,
after running the speedup process $\tilde \Xi^{\rightarrow}$
for time $\tau>0$,
the resulting trajectory
$\{y_1'(\mathsf{t})\}_{\mathsf{t} \ge0}$ 
of the first particle 
is
the continuous time Poisson simple random walk 
with slope $e^\tau$.
Moreover, note that 
the trajectory of the 
first particle evolves \emph{independently} from all other particles,
under the speedup dynamics.
This independence is a result of setting $q=0$, 
since for $q>0$ the jump rates in 
$\tilde \Xi^{\rightarrow}$ depend on the second trajectory $y_2(\mathsf{t})$.

Let us now describe the evolution of
the trajectory for the first particle.
We may assume that $\mathsf{M}=+\infty$
and give the description for the full trajectories, that is, 
where $\mathsf{t}\in[0,+\infty)$.
Change the time of the speedup process as 
$\hat \tau=e^{\tau}-1$. 
Then, the two-dimensional Poisson process $\mathfrak{P}_{[1]}$
of rate $e^\tau$, with $0\le \tau\le \tau_0$,
turns into the 
homogeneous two-dimensional Poisson process 
$\hat{\mathfrak{P}}_{[1]}$ of rate $1$,
with $0\le \tau\le \hat\tau_0=e^{\tau_0}-1$.

For each point
$(\mathsf{t}_*,\hat{\tau}_*)$
of
$\hat{\mathfrak{P}}_{[1]}$, 
set $y_1'(\mathsf{t}_*)=y_1(\mathsf{t}_*)+1$. Then, instantaneously replace a piece of the trajectory of the first particle
for $\mathsf{t}\ge \mathsf{t}_*$ by an independent 
Poisson random walk in continuous time $\mathsf{t}$
of slope $\hat\tau_*+1$. In particular, $y_1'$ is started from $y_1'(\mathsf{t}_*)$, and the new walk continues until it reaches the old trajectory of $y_1$ at some time $\mathsf{t}_\circ\in[0,+\infty]$.
Note that the case $\mathsf{t}_{\circ}= \infty$ happens with positive probability. Then, if $\mathsf{t}_\circ=\infty$, we independently
resample the whole trajectory to the right of $\mathsf{t}_*$ by a trajectory
with the new slope (jump rate) $\hat \tau_*+1$.
We may think that each jump of the trajectory is an instantaneous
``avalanch'' whose slope $\hat\tau_*+1$ grows with time. This resembles the avalanch
processes in, e.g.,
\cite{povolotsky2003asymmetric},
\cite{le2022equivalence},
but we will not explore this connection further in the present paper.
See \Cref{fig:Hammersley_strikes_back} for an illustration of the process.

\Cref{prop:action_Xi_up_cont}
for $q=0$ immediately implies the following coupling of 
Poisson processes with different slopes:

\begin{proposition}
	\label{prop:neat_last_fact}
	Take a Poisson simple random walk of slope $1$. Apply the
	dynamics on single-particle trajectories
	described above, evolving in continuous time $\hat\tau$. Then,
	the distribution of the resulting trajectory
	at each time $\hat \tau$ 
	is the Poisson simple random walk of slope $\hat \tau+1$.
\end{proposition}

\begin{figure}[htpb]
	\centering
	\includegraphics[width=.7\textwidth]{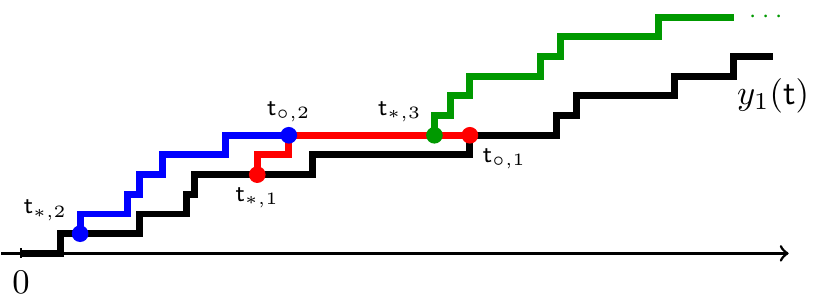}
	\caption{Three consecutive jumps in the process 
	$\tilde \Xi^{\rightarrow}$ acting on the Poisson simple random walk
	$y_1(\mathsf{t})$. Each $i$-th jump
	originates from $\mathsf{t}_{*,i}$, 
	which are the horizontal coordinates of the points of a two-dimensional
	Poisson process $\hat{\mathfrak{P}}_{[1]}$,
	and terminates at $\mathsf{t}_{\circ,i}\in[0,+\infty]$.
	One can say that the trajectory grows by instantaneous ``avalanches''.}
	\label{fig:Hammersley_strikes_back}
\end{figure}

We obtain a Markov chain
which preserves the Poisson process on $[0,1]$ as a corollary.  
The construction of this Markov process is based on the fact that,
after one jump in the process
$\tilde \Xi^{\rightarrow}$, we can dilate the horizontal line to 
decrease
the slope of the Poisson simple random walk
back to slope $1$.

\begin{definition}[{Markov chain preserving the Poisson random walk on $[0,1]$}]
	\label{def:Poisson_Markov}
	Let $y_1(\mathsf{t})$ be a trajectory 
	of the Poisson random walk of slope $1$ on $[0,1]$
	with $y_1(0)=0$.
	Independently, sample two random variables:
	$\mathsf{t}_*\in [0,1]$ with uniform distribution, and 
	$\zeta>0$ with exponential distribution with parameter $1$. 
	\begin{enumerate}
		\item If $\mathsf{t}_*>1 / (1+\zeta)$, do nothing
			and denote $Y_1(\mathsf{t})=y_1(\mathsf{t})$.
			This event happens with probability $1-e \Gamma(0,1)\approx 0.4$,
			where $\Gamma(0,1)=\int_1^\infty e^{-t}t^{-1}dt$ is the incomplete Gamma function.
		\item If $\mathsf{t}_*\le 1 / (1+\zeta)$,
			start a new independent Poisson
			random walk $y_1'(\mathsf{t})$ of slope $\zeta+1$ 
			from the location $(\mathsf{t}_*,y_1(\mathsf{t}_*)+1)$ and running in 
			forward time $\mathsf{t}\in[\mathsf{t}_*, 1 / (1+\zeta)]$.
			When the new walk reaches the old trajectory of $y_1(\mathsf{t})$
			or $\mathsf{t}$ reaches the coordinate $1 / (1+\zeta)$, stop the update.
			Denote the resulting trajectory by $Y_1(\mathsf{t})$, $0\le \mathsf{t}\le 1$.
	\end{enumerate}
	Forget the configuration of $Y_1(\mathsf{t})$
	for $\mathsf{t}\in (1 / (1+\zeta) , 1]$
	after the update in both cases. 
	Then,
	dilate the segment $[0,1 / (1+\zeta)]$ 
	by means of multiplication by $1+\zeta$. 
	In particular, the
	result of the application of the Markov transition operator 
	is, by definition, the trajectory 
	$Z_1(\mathsf{t})=Y_1(\mathsf{t} / (1+\zeta))$, $0\le \mathsf{t}\le 1$.
\end{definition}

\begin{corollary}
	\label{cor:Poisson_dilation_lemma}
	Let $\{y_1(\mathsf{t})\}_{0\le \mathsf{t}\le 1}$, 
	be a trajectory 
	of the Poisson simple random walk on $[0,1]$ with $y_1(0)=0$. 
	Then, 
	the trajectory $\{Z_1(\mathsf{t})\}_{0\le \mathsf{t}\le 1}$ described in \Cref{def:Poisson_Markov}
	has the same distribution as $\{y_1(\mathsf{t})\}_{0\le \mathsf{t}\le 1}$.
\end{corollary}

To the best of our knowledge,
\Cref{prop:neat_last_fact}
and \Cref{cor:Poisson_dilation_lemma}
are new. 
It is possible to isolate the proof
of \Cref{prop:neat_last_fact}
from the rest of the paper. That proof would essentially follow by
iterating the
Yang-Baxter equation and taking Poisson-type limits. It would be 
interesting to find direct proofs
of
\Cref{prop:neat_last_fact}
and \Cref{cor:Poisson_dilation_lemma}
which would not rely on discrete models and Poisson limits.


\medskip

\textsc{L. Petrov, University of Virginia, Charlottesville, VA 22904, USA}

E-mail: \texttt{lenia.petrov@gmail.com}

\medskip

\textsc{A. Saenz, Oregon State University, Corvallis, OR 97331, USA}

E-mail: \texttt{saenzroa@oregonstate.edu}

\end{document}